\documentclass[smallextended]{svjour3}       
\usepackage [latin1]{inputenc}
\smartqed  

\usepackage[toc,page,title,titletoc,header]{appendix}

\usepackage{graphicx}
\usepackage{subfigure}
\usepackage[misc]{ifsym}
\usepackage{amssymb,amsfonts,mathrsfs,amsmath,multirow,mathtools,color,array}
\usepackage{cite,hyperref,caption}
\usepackage{algorithm}
\usepackage{epstopdf}
\usepackage{diagbox}
\usepackage{listings}
\usepackage{threeparttable}
\usepackage{algorithmic}
\allowdisplaybreaks[4]

\newcommand{\be}{\begin{equation}}
\newcommand{\ee}{\end{equation}}
\newcommand\bea{\begin{eqnarray}}
\newcommand\eea{\end{eqnarray}}
\newcommand{\bean}{\begin{eqnarray*}}
\newcommand{\eean}{\end{eqnarray*}}

\usepackage{cases}
\newcommand\bcase{\begin{numcases}{}}
\newcommand\ecase{\end{numcases}}

\begin{document}

\title{Optimal convergence rate of the explicit Euler method for convection-diffusion equations II: high dimensional cases
}
\titlerunning{Optimal convergence rate of the explicit Euler method for convection-diffusion equations}        
\author{Qifeng Zhang \and Jiyuan Zhang \and Zhi-zhong Sun
}
\authorrunning{Q. Zhang, J. Zhang and Z. Sun} 
\institute{                            
		\Letter\; Qifeng Zhang \at Department of Mathematics, Zhejiang Sci-Tech University, Hangzhou, 310018, China; \\ \email{zhangqifeng0504@zstu.edu.cn}
\and Jiyuan Zhang \at Department of Mathematics, Zhejiang Sci-Tech University, Hangzhou, 310018, China\\
               \email{z018283@126.com}
\and Zhi-zhong Sun \at Department of Mathematics, Southeast University, Nanjing, 210096, China\\
               \email{zzsun@seu.edu.cn}
}

\date{Received: date / Accepted: date}

\maketitle

\begin{abstract}
This is the second part of study on the optimal convergence rate of
the explicit Euler discretization in time for the convection-diffusion equations [Appl. Math. Lett. \textbf{131} (2022) 108048]
which focuses on high-dimensional linear/nonlinear cases under Dirichlet or Neumann boundary conditions.
Several new corrected difference schemes are proposed based on the
explicit Euler discretization in temporal derivative and central difference discretization in spatial derivatives.
The priori estimate of the corrected scheme with application to constant convection coefficients is provided at length by the maximum principle and the optimal convergence rate four is proved when the step ratios along each direction equal to $1/6$.
The corrected difference schemes have essentially improved {\rm \textbf{CFL}} condition and the numerical accuracy comparing with the classical difference schemes. Numerical examples involving two-/three-dimensional linear/nonlinear problems under Dirichlet/Neumann boundary conditions such as the Fisher equation, the Chafee-Infante equation, the Burgers' equation and classification to name a few substantiate the good properties claimed for the corrected difference scheme.

\end{abstract}

\section{Introduction}\label{Sec1}
\setcounter{equation}{0}
Historically, the explicit Euler method has been one of the most classical and oldest numerical methods, which is compulsory part in our textbooks and monographs, see e.g., \cite{HNW1987,LT2003,Sun2022,Th1995}.
It and other improved versions were frequently used as the time integrator for the numerical solutions of ordinary differential equations or partial differential equations \cite{KK1991,MT2006,H1988,VR1999,SR2002}.

However, most of the time we ignore the fact that explicit Euler method could generate superconvergence with a specific step-ratio when it is applied to solve the convection-diffusion problems \cite{ZZS2022}.
Subsequently to the result in \cite{ZZS2022}, as the second part of this series of study, we further investigate the optimal convergence rate of the explicit Euler method with application to the convection-diffusion equations in two-dimensional and three-dimensional cases.

In this paper, we will first consider the numerical procedure for the initial-boundary value problem of the anisotropic convection-diffusion equation with variable convection coefficients as follows
\begin{subequations}
\label{eqn1}
\begin{numcases}{}
u_{t}= \nabla \!\cdot (A \nabla u)  + B \cdot \nabla u + f(\mathbf{x},t),\quad  \mathbf{x}\in \Omega,\; t\in (0,T],  \label{eqn1a} \\
u(\mathbf{x},0)=\varphi(\mathbf{x}),\quad \mathbf{x}\in \bar{\Omega},\label{eqn1b} \\
u(\mathbf{x},t)=\alpha(\mathbf{x},t), \quad \mathbf{x}\in\Gamma,\quad t\in (0,T], \label{eqn1c}
\end{numcases}
\end{subequations}
where $\mathbf{x} \in \Omega \subset \mathbb{R}^2$ or $\mathbb{R}^3$, $A$ is a diagonal matrix with positive diagonal elements and $B$ is a two- or three-dimensional vector-value function. $\bar{\Omega}$ is the closure of $\Omega$ and $\Gamma$ is the boundary of $\Omega$. $f(\mathbf{x},t)$, $\varphi(\mathbf{x})$ and $\alpha(\mathbf{x},t)$ are given smooth functions satisfying the consistent conditions. Then we will further move our attention to a more general nonlinear convection-diffusion equation
\begin{subequations}
\label{eqn4.1}
\begin{numcases}{}
u_{t}+\nabla\! \cdot \mathbf{F}(u)=\nabla\! \cdot (A \nabla u)+r(u),\quad \mathbf{x}\in \Omega,\; t\in (0,T],  \label{eqn4.1a} \\
u(\mathbf{x},0)=\varphi(\mathbf{x}),\quad \mathbf{x}\in \bar{\Omega},\label{eqn4.1b} \\
u(\mathbf{x},t)=\alpha(\mathbf{x},t), \quad \mathbf{x}\in\Gamma,\quad t\in (0,T], \label{eqn4.1c}
\end{numcases}
\end{subequations}
where the flux $\mathbf{F}(u) = (f(u),g(u))$ or $(f(u),g(u),h(u))$, which occurs in many applications such as semi-linear or quasi-linear problems:
the Fisher equation \cite{QS1998}($ u_t=\Delta u + u(1-u)$), the Chafee-Infante equation \cite{CI1974} ($u_t=\Delta u + u(1-u^2)$), the scalar viscous Burgers' equation \cite{MS2007} ($u_t+ uu_x+uu_y  = \mu\Delta u$, $\mu\neq0$ denotes the viscous coefficient), and so on.

Though there have been plenty of discussions for numerical methods to solve these issues, we here only focus on the simplest numerical discretization of the problem \eqref{eqn1} or \eqref{eqn4.1} composing of the standard centered finite difference method for the spatial derivatives and the forward Euler method for the temporal derivative because of its simplicity, time-saving and easy to implement. Scholars generally consider the scheme resulting from such discretization as a low order, conditionally stable and therefore impractical numerical method. However, by applying a corrected difference technique to the local truncation errors, we recover the optimal convergence rate four when the specific step-ratios are utilized and the exact solution satisfies a certain of regularity. Our finding makes the corrected scheme suitable to simulate the physical phenomenon with a very high precision. There are several efforts and novelty in the current paper for our corrected difference scheme. Specifically:
\begin{itemize}
  \item [{\rm (I)}] The explicit Euler method combines with a corrected second-order centered difference discretization exports to a superconvergent numerical scheme when the step-ratios $r_x = r_y = r_z=1/6$, where $r_x$, $r_y$ and $r_z$ denote the step-ratios along $x$-, $y$- and $z$-direction, respectively. Moreover, the corrected difference scheme owns higher numerical accuracy compared with the standard centered difference scheme even if the step-ratios $(r_x, r_y,r_z)\neq(1/6,1/6,1/6)$.
  \item [{\rm (II)}] The priori estimate is demonstrated for the corrected difference scheme under the infinite norm based on the maximum principle in the case of constant convection coefficients, which naturally leads to {\rm{\textbf{CFL}}} condition and optimal convergence for the corrected scheme. In detail, the stability for the two-dimensional diffusion problem is doubled compared with the classical difference scheme. For three-dimensional diffusion problem, the stability slightly decreases but shows new stable region, see also Figure \ref{sketch2} in Section \ref{SecDiffu}.
  \item [{\rm (III)}] Most notably, we discover that the corrected difference scheme could be extended to solve more general high-dimensional nonlinear convection-diffusion equations such as the Fisher equation, the Chafee-Infante equation, the viscous Burgers' equation and classification to name a few.
  \item [{\rm (IV)}] It is worth noting that the the corrected difference scheme is fully explicit and very convenient to be implemented. Moreover, it is not constrained by the boundary conditions.
  Numerical examples in a variety of scenarios are carried out for linear/nonlinear problems under Dirichlet and Neumann boundary conditions to confirm the designed convergence rate.
\end{itemize}

Throughout the whole paper, we always assume that the exact solutions to \eqref{eqn1} or \eqref{eqn4.1} are sufficiently smooth in the sense that $u(\mathbf{x},t)\in \mathcal{C}^{6,6,4}(\bar\Omega\times [0,T])$ for two-dimensional problem and $u(\mathbf{x},t)\in \mathcal{C}^{6,6,6,4}(\bar\Omega\times [0,T])$ for three-dimensional problem.

The rest of the paper is arranged as follows. In Section \ref{Sec2}, we derive a corrected explicit difference scheme for a linear two-dimensional convection-diffusion equation involving several special cases. In Section \ref{Sec3}, the priori estimate with constant convection coefficients is discussed and the optimal convergence rate with fourth-order accuracy is proved.
Extending the corrected difference scheme to the nonlinear convection-diffusion problem with the Dirichlet boundary conditions  and  the Neumann boundary conditions are available in Section \ref{Sec4}  and Section \ref{Sec5} respectively. Our corrected technique is also applied to three-dimensional convection-diffusion problem in Section \ref{SecDiffu}.
Extensive numerical examples including linear/nonlinear cases under Dirichlet/Neumann boundary conditions are carried out to verify the theoretical results in Section \ref{Sec6} before a short concluding remarks in Section \ref{Sec7}.

\section{The derivation of the explicit difference scheme}\label{Sec2}
\setcounter{equation}{0}
In this section, we focus on the construction  of the numerical scheme to solve two-dimensional convection-diffusion problem \eqref{eqn1}. Here $\mathbf{x}=(x,y)$, $A = {\rm \mathbf{diag}}(a,b)$ with $a,b>0$ constant diffusion coefficients and $B = (c(\mathbf{x}),d(\mathbf{x}))$ the convection coefficients. The spatial domain is set to $\Omega = (L_1,R_1)\times(L_2,R_2)\subset \mathbb{R}^2$.

We start by introducing some basic notations in the context of finite difference method.
Firstly, the domain $\bar{\Omega}\times[0,T]$ is subdivided into a number of small elements by passing
orthogonal lines through the region. For this purpose,
 we take three positive integers $m_1$, $m_2$, $n$ and let $h_x = (R_1-L_1)/m_1$, $h_y = (R_2-L_2)/m_2$, $\tau = T/n$, $x_i =L_1 +ih_x$, $0\leq i \leq m_1$, $y_j = L_2+jh_y$, $0\leq j \leq m_2$, $t_k = k\tau$, $0 \leq k \leq n$.
Define the step-ratios $r_x = a\tau /h_x^2$, $r_y = b\tau/h_y^2$.
Denote $\Omega_{h\tau}=\{(\mathbf{x}_{ij},t_k)\,|\,0\leq i\leq m_1, \ 0\leq j\leq m_2, \ 0\leq k\leq n\}$ with $\mathbf{x}_{ij} = (x_i,y_j)$, $\omega = \{(i,j)\,|\,\mathbf{x}_{ij}\in \Omega\}$, $\gamma = \{(i,j)\,|\,\mathbf{x}_{ij}\in \Gamma\}$ and
$\bar\omega =\omega \cup \gamma.$
For any grid function
$v=\{v_{ij}^k\,|\, 0\leq i\leq m_1, 0\leq j \leq m_2, 0\le k\le n\}$ on $\Omega_{h\tau},$
define $\delta_t v_{ij}^{k+\frac{1}{2}} = (v_{ij}^{k+1}-v_{ij}^k)/\tau$,
 $\delta_x v_{i+1/2,j}^{k} = (v_{i+1,j}^k-v_{ij}^k)/h_x$, $\delta_x^2 v_{ij}^{k} = (v_{i+1,j}^k-2v_{ij}^k+v_{i-1,j}^k)/h_x^2$,
 $\Delta_xv_{ij}^{k} = (v_{i+1,j}^k-v_{i-1,j}^k)/(2h_x)$, $\|v\|_{\infty} = \max_{(i,j)\in \bar{\omega}}\limits|v_{ij}|$.
 Similarly, we could define $\delta_y v_{i,j+1/2}^{k}$, $\delta_y^2 v_{ij}^{k}$ and $\Delta_yv_{ij}^{k}$.

Firstly, using \eqref{eqn1a}, we have
\begin{align}
  u_{tt} = & \;a(u_t)_{xx} + b(u_t)_{yy} + c(u_t)_x + d(u_t)_y +f_t(\mathbf{x},t)\notag\\
         = & \;a \Big(a u_{xx} + bu_{yy} + c(\mathbf{x})u_x + d(\mathbf{x})u_y + f(\mathbf{x},t)\Big)_{xx}\notag\\
           & \;+ b\Big(a  u_{xx} + bu_{yy} + c(\mathbf{x})u_x + d(\mathbf{x})u_y + f(\mathbf{x},t)\Big)_{yy}\notag\\
           & \;+ c\Big(a  u_{xx} + bu_{yy} + c(\mathbf{x})u_x + d(\mathbf{x})u_y + f(\mathbf{x},t)\Big)_x\notag\\
           & \;+ d\Big(a  u_{xx} + bu_{yy} + c(\mathbf{x})u_x + d(\mathbf{x})u_y + f(\mathbf{x},t)\Big)_y \notag\\
           & \;+ f_t(\mathbf{x},t)\notag\\
         = & \; a^2u_{xxxx} + 2abu_{xxyy} + b^2u_{yyyy} \notag\\
           & \; + 2ac(\mathbf{x})u_{xxx} + 2ad(\mathbf{x})u_{xxy} + 2bc(\mathbf{x})u_{xyy} + 2bd(\mathbf{x})u_{yyy}\notag\\
           & \; +\Big(2ac_x(\mathbf{x})+c^2(\mathbf{x})\Big)u_{xx}  + \Big(2bd_y(\mathbf{x})+bd_{yy}(\mathbf{x})\Big)u_{yy}\notag\\
           & \; + \Big(2ad_x(\mathbf{x})+2bc_y(\mathbf{x})+2c(\mathbf{x})d(\mathbf{x})\Big)u_{xy}\notag\\
           & \; +\Big(ac_{xx}(\mathbf{x}) + bc_{yy}(\mathbf{x}) + c(\mathbf{x})c_x(\mathbf{x}) + d(\mathbf{x})c_y(\mathbf{x})\Big)u_x\notag\\
           & \; +\Big(ad_{xx}(\mathbf{x}) + bd_{yy}(\mathbf{x}) + c(\mathbf{x})d_x(\mathbf{x}) + d(\mathbf{x})d_y(\mathbf{x})\Big)u_y\notag\\
           & \; + af_{xx}(\mathbf{x},t) + bf_{yy}(\mathbf{x},t) + c(\mathbf{x})f_x(\mathbf{x},t)+d(\mathbf{x})f_y(\mathbf{x},t)\notag\\
           & \; + f_t(\mathbf{x},t).\label{2Deqn2}
\end{align}

Define the grid function $$U=\{ U_{ij}^k\,|\,0\leq i\leq m_1,\; 0 \leq j \leq m_2,\; 0\leq k\leq n\}$$ with $U_{ij}^k=u(\mathbf{x}_{ij},t_k)$.
Considering \eqref{eqn1a} at the point $(\mathbf{x}_{ij},t_k)$ and using the Taylor formula, we have
\begin{align}
 & u_{t}(\mathbf{x}_{ij},t_k)=a  u_{xx}(\mathbf{x}_{ij},t_k) + bu_{yy}(\mathbf{x}_{ij},t_k) + c_{ij}u_x(\mathbf{x}_{ij},t_k) + d_{ij}u_y(\mathbf{x}_{ij},t_k) + f_{ij}^k,\notag\\
 & \qquad \qquad\qquad\qquad\qquad\qquad(i,j)\in \omega,\; 0\leq k\leq n, \label{2Deqn3}
\end{align}
where
\begin{align*}
   c_{ij} = c(\mathbf{x}_{ij}),\quad d_{ij} = d(\mathbf{x}_{ij}),\quad f_{ij}^k = f(\mathbf{x}_{ij},t_k).
\end{align*}

The forward difference quotient is utilized for the discretization of the temporal derivative and the central difference quotient for the discretization of the spatial derivatives in \eqref{2Deqn3}, which results in
\begin{align}
   &\;\delta_t U_{ij}^{k+\frac{1}{2}} - a\delta_x^2U_{ij}^k - b\delta_y^2 U_{ij}^k -c_{ij}\Delta_xU_{ij}^k -d_{ij}\Delta_yU_{ij}^k\notag\\
  =&\; f_{ij}^k+\frac{\tau}{2}u_{tt}(\mathbf{x}_{ij},t_k) -\frac{a h_x^2}{12}u_{xxxx}(\mathbf{x}_{ij},t_k)- \frac{bh_y^2}{12}u_{yyyy}(\mathbf{x}_{ij},t_k)\notag\\
   &\;-\frac{h_x^2}{6}c_{ij}u_{xxx}(\mathbf{x}_{ij},t_k)-\frac{h_y^2}{6}d_{ij}u_{yyy}(\mathbf{x}_{ij},t_k) + O(\tau^2+h_x^4+h_y^4)\notag\\
  =&\; p_{ij}^k+\frac{ah_x^2}{2}\Big(r_x-\frac{1}{6}\Big)u_{xxxx}(\mathbf{x}_{ij},t_k) + \frac{bh_y^2}{2}\Big(r_y-\frac{1}{6}\Big)u_{yyyy}(\mathbf{x}_{ij},t_k) \notag\\
   &\; + c_{ij}h_x^2\Big(r_x-\frac{1}{6}\Big)u_{xxx}(\mathbf{x}_{ij},t_k) + d_{ij}h_y^2\Big(r_y-\frac{1}{6}\Big)u_{yyy}(\mathbf{x}_{ij},t_k)\notag\\
   &\; + \frac{\tau}{2}\left[(2ac_x + c^2)_{ij}\delta_x^2 U_{ij}^k + (2ad_x+2bc_y + 2cd)_{ij}\Delta_x\Delta_yU_{ij}^k+(2bd_y +d^2)_{ij}\delta_y^2 U_{ij}^k\right.\notag\\
   &\; + \left.(ac_{xx}+bc_{yy}+cc_x+dc_y)_{ij}\Delta_x U_{ij}^k + (ad_{xx} + bd_{yy} + cd_x + dd_y)_{ij}\Delta_yU_{ij}^k\right]\notag\\
   &\; + \tau ab\delta_x^2\delta_y^2 U_{ij}^k + \tau ad_{ij}\delta_x^2\Delta_y U_{ij}^k +\tau bc_{ij}\delta_y^2\Delta_xU_{ij}^k + O(\tau^2+\tau h_x^2 + \tau h_y^2+h_x^4+h_y^4), \label{2Deqn4}
\end{align}
where \eqref{2Deqn2} is used in the second inequality and
\begin{align*}
  p_{ij}^k = f_{ij}^k +\frac{\tau}{2}\left[a(f_{xx})_{ij}^k + b(f_{yy})_{ij}^k + (c f_{x})_{ij}^k + (d f_{y})_{ij}^k + (f_t)_{ij}^k \right]
\end{align*}
with $(f_x)_{ij}^k = f_x(\mathbf{x}_{ij},t_k),\; (f_y)_{ij}^k = f_y(\mathbf{x}_{ij},t_k),\;(f_{xx})_{ij}^k = f_{xx}(\mathbf{x}_{ij},t_k),\; (f_{yy})_{ij}^k = f_{yy}(\mathbf{x}_{ij},t_k),\;(f_t)_{ij}^k = f_t(\mathbf{x}_{ij},t_k).$

Omitting the small terms in \eqref{2Deqn4},
a finite difference scheme for \eqref{eqn1} reads
\begin{subequations}
\label{2Deqn1}
\begin{numcases}{}
  \delta_tu_{ij}^{k+\frac{1}{2}} = \Big[a+\frac{\tau}{2}(2ac_x+c^2)_{ij}\Big]\delta_x^2u_{ij}^k + \Big[b+\frac{\tau}{2}(2bd_y+d^2)_{ij}\Big]\delta_y^2u_{ij}^k\notag\\
  \quad +\Big[c_{ij} +\frac{\tau}{2}(ac_{xx}+bc_{yy}+cc_x+c_yd)_{ij}\Big]\Delta_x u_{ij}^k+\Big[d_{ij} +\frac{\tau}{2}(ad_{xx}+bd_{yy}+cd_x+dd_y)_{ij}\Big]\Delta_y u_{ij}^k\notag\\
  \quad +\frac{\tau}{2}(2ad_x+2bc_y+2cd )_{ij}\Delta_x\Delta_y u_{ij}^k
  +\tau\Big( ab\delta_x^2\delta_y^2 u_{ij}^k + \tau ad_{ij} \delta_x^2 \Delta_yu_{ij}^k +  bc_{ij}\delta_y^2\Delta_xu_{ij}^k\Big) + p_{ij}^k,\notag\\
  \qquad \qquad \qquad \qquad \qquad (i,j)\in \omega,\; 0\leq k \leq n-1,\label{2Deqn1a} \\
  u_{ij}^0=\varphi(\mathbf{x}_{ij}), \qquad \qquad \ \! (i,j)\in \bar\omega,\; \label{2Deqn1b}\\
  u_{ij}^k=\alpha(\mathbf{x}_{ij},t_k),\qquad\quad \! (i,j)\in \gamma, \; 1 \leq k \leq n.\label{2Deqn1c}
\end{numcases}
\end{subequations}

It is easy to see that the local truncation error for \eqref{2Deqn1a} is 
\begin{subequations}
\label{Local_err}
\begin{numcases}{}
  O(\tau^2+h_x^4+h_y^4), {~~\rm if~~}  r_x = r_y = 1/6,  \label{Local_erra}\\
  O(\tau^2+h_x^2+h_y^2), {~~\rm otherwise.~~}\label{Local_errd}
\end{numcases}
\end{subequations}

\begin{remark}
The difference scheme \eqref{2Deqn1} has improved several classical numerical schemes, for example:
  \begin{itemize}
    \item [{\rm \textbf{(I)}.}] $c=d=0:$ the difference scheme \eqref{2Deqn1} reduces to
    \begin{subequations}
\label{2Deqn5}
\begin{numcases}{}
  \delta_tu_{ij}^{k+\frac{1}{2}} = a\delta_x^2u_{ij}^k + b\delta_y^2u_{ij}^k
  +\tau ab\delta_x^2\delta_y^2 u_{ij}^k + q_{ij}^k, \qquad (i,j)\in \omega,\; 0\leq k \leq n-1,\label{2Deqn5a} \\
  u_{ij}^0=\varphi({\mathbf{x}_{ij}}), \quad \qquad \ \! (i,j)\in \bar\omega,\; \label{2Deqn5b}\\
  u_{ij}^k=\alpha(\mathbf{x}_{ij},t_k),\qquad (i,j)\in \gamma, \; 1 \leq k \leq n,\label{2Deqn5c}
\end{numcases}
\end{subequations}
where
\begin{align*}
  q_{ij}^k = f_{ij}^k +\frac{\tau}{2}\left[a(f_{xx})_{ij}^k + b(f_{yy})_{ij}^k + (f_t)_{ij}^k \right].
\end{align*}
Comparing with the classical Euler difference scheme
    \begin{subequations}
\label{2Deqn6}
\begin{numcases}{}
  \delta_tu_{ij}^{k+\frac{1}{2}} = a\delta_x^2u_{ij}^k + b\delta_y^2u_{ij}^k
   + f_{ij}^k, \qquad (i,j)\in \omega,\; 0\leq k \leq n-1,\label{2Deqn6a} \\
  u_{ij}^0=\varphi(\mathbf{x}_{ij}), \quad \qquad \ \! (i,j)\in \bar \omega,\; \label{2Deqn6b}\\
  u_{ij}^k=\alpha(\mathbf{x}_{ij},t_k),\qquad (i,j)\in \gamma, \; 1 \leq k \leq n,\label{2Deqn6c}
\end{numcases}
\end{subequations}
the local truncation error for the difference scheme \eqref{2Deqn5} is the same as \eqref{Local_err}.
However, the local truncation error for the classical Euler difference scheme \eqref{2Deqn6}
is only two globally whatever the step-ratios are.
    \item [{\rm \textbf{(II)}.}] $c=d={\rm constant:}$ the difference scheme \eqref{2Deqn1} simplifies into
    \begin{subequations}
\label{2Deqn7}
\begin{numcases}{}
  \delta_tu_{ij}^{k+\frac{1}{2}} = \left(a+\frac{\tau}{2}c^2\right)\delta_x^2u_{ij}^k + \left(b+\frac{\tau}{2}d^2\right)\delta_y^2u_{ij}^k
    +c\Delta_x u_{ij}^k +d\Delta_y u_{ij}^k+\tau cd\Delta_x\Delta_y u_{ij}^k\notag\\
  \quad +\tau \Big( ab\delta_x^2\delta_y^2 u_{ij}^k + ad  \delta_x^2 \Delta_y u_{ij}^k + bc \delta_y^2\Delta_xu_{ij}^k\Big) + r_{ij}^k, \quad (i,j)\in \omega,\; 0\leq k \leq n-1,\label{2Deqn7a} \\
  u_{ij}^0=\varphi(\mathbf{x}_{ij}), \quad \qquad \ \! (i,j)\in \bar \omega,\; \label{2Deqn7b}\\
  u_{ij}^k=\alpha(\mathbf{x}_{ij},t_k),\qquad (i,j)\in \gamma, \; 1 \leq k \leq n,\label{2Deqn7c}
\end{numcases}
\end{subequations}
where
\begin{align*}
  r_{ij}^k = f_{ij}^k +\frac{\tau}{2}\left[a(f_{xx})_{ij}^k + b(f_{yy})_{ij}^k + c(f_{x})_{ij}^k + d(f_{y})_{ij}^k + (f_t)_{ij}^k \right].
\end{align*}
  \end{itemize}
  The classical Euler difference scheme for the problem \eqref{eqn1} is
\begin{subequations}
\label{2Deqn8}
\begin{numcases}{}
  \delta_tu_{ij}^{k+\frac{1}{2}} - a\delta_x^2u_{ij}^k - b\delta_y^2u_{ij}^k
    -c\Delta_x u_{ij}^k - d\Delta_y u_{ij}^k   = f_{ij}^k, \quad (i,j)\in \omega,\; 0\leq k \leq n-1,\label{2Deqn8a} \\
  u_{ij}^0=\varphi(\mathbf{x}_{ij}), \quad \qquad \ \! (i,j)\in \bar\omega,\; \label{2Deqn8b}\\
  u_{ij}^k=\alpha(\mathbf{x}_{ij},t_k),\qquad (i,j)\in \gamma, \; 0 \leq k \leq n.\label{2Deqn8c}
\end{numcases}
\end{subequations}
Comparing \eqref{2Deqn7} with \eqref{2Deqn8}, similar theoretical results can be obtained for the difference scheme \eqref{2Deqn7}.

We call each difference scheme \eqref{2Deqn1}, \eqref{2Deqn5} or \eqref{2Deqn7} as the \textbf{corrected difference scheme}.
\end{remark}

\section{The priori estimate and the optimal convergence}\label{Sec3}
\setcounter{equation}{0}
For the sake of brevity, we take the difference scheme \eqref{2Deqn7} as an example to illustrate the priori estimate.
In the meantime, the problem \eqref{eqn1} becomes
\begin{subequations}
\label{eqn2}
\begin{numcases}{}
u_{t}=a  u_{xx} + bu_{yy} + c u_x + d u_y + f(\mathbf{x},t),\quad  \mathbf{x} \in \Omega,\; t\in (0,T],  \label{eqn2a} \\
u(\mathbf{x},0)=\varphi(\mathbf{x}),\quad \mathbf{x}\in \bar{\Omega},\label{eqn2b} \\
u(\mathbf{x},t)=\alpha(\mathbf{x},t), \quad \mathbf{x}\in\Gamma,\quad t\in (0,T]. \label{eqn2c}
\end{numcases}
\end{subequations}

\begin{theorem}[Priori estimate]\label{thm1}
 Let $\{u^k_{ij}\,|\,0\leq i\leq m_1,\ 0\leq j \leq m_2, \ 0 \leq k \leq n\}$ be the solution of the difference scheme \eqref{2Deqn7} with the constraint $\alpha (\mathbf{x}, t)\equiv 0$.
 When
\begin{subequations}
\label{CFL_cond}
\begin{numcases}{}
\max\{r_x, r_y\} \leq 1/2,\label{CFL_cond1}\\
|c|h_x \leq a \cdot \min\Big\{2, \sqrt{(1-2r_x)(1-2r_y)/2}\Big/r_x\Big\},\label{CFL_cond2}\\
|d|h_y \leq b \cdot \min\Big\{2, \sqrt{(1-2r_x)(1-2r_y)/2}\Big/r_y\Big\},\label{CFL_cond3}
\end{numcases}
\end{subequations}
it holds
\begin{equation*}
\|u^k\|_\infty\leq \|\varphi\|_{\infty}+\tau\sum\limits_{l=0}^{k-1}\|r^l\|_{\infty},\quad 1\leq k\leq n.
\end{equation*}
 \end{theorem}
\begin{proof}
The difference scheme \eqref{2Deqn7a} can be rewritten as
\begin{align}
 u_{ij}^{k+1}= &\; S_1u_{i+1,j}^{k} +S_2u_{i-1,j}^{k}+S_3u_{i,j+1}^{k} +S_4u_{i,j-1}^{k}
 +S_5 u_{ij}^{k}\notag\\
& +S_6 u_{i+1,j+1}^{k}+S_7 u_{i+1,j-1}^k+S_8  u_{i-1,j+1}^{k} +S_9  u_{i-1,j-1}^{k}+\tau r_{ij}^{k},\notag\\
&  \qquad \qquad (i, j)\in\omega,\; 0\le k\le n-1,\label{2Deqn7bb}
\end{align}
where
\begin{align*}
  & S_1 =  (1-2r_y)\Big(1+\frac{c}{2a}h_x\Big)r_x+\frac{c^2}{2a^2r_x^2}h_x^2 ,\quad
  S_2 =  (1-2r_y)\Big(1-\frac{c}{2a}h_x\Big)r_x+\frac{c^2}{2a^2r_x^2}h_x^2 ,\\
  & S_3 =  (1-2r_x)\Big(1+\frac{d}{2b}h_y\Big)r_y+\frac{d^2}{2b^2r_y^2}h_y^2 ,\quad S_4 =  (1-2r_x)\Big(1-\frac{d}{2b}h_y\Big)r_y+\frac{d^2}{2b^2r_y^2}h_y^2 ,\\
  & S_5 =  (1-2r_x)(1-2r_y)- \frac{c^2r_x^2}{a^2}h_x^2-\frac{d^2r_y^2}{b^2}h_y^2,\\
  & S_6 = \Big(1+\frac{c}{2a}h_x\Big)\Big(1+\frac{d}{2b}h_y\Big)r_xr_y,\quad
   S_7 = \Big(1+\frac{c}{2a}h_x\Big)\Big(1-\frac{d}{2b}h_y\Big)r_xr_y,\\
  & S_8 = \Big(1-\frac{c}{2a}h_x\Big)\Big(1+\frac{d}{2b}h_y\Big)r_xr_y,\quad
   S_9 = \Big(1-\frac{c}{2a}h_x\Big)\Big(1-\frac{d}{2b}h_y\Big)r_xr_y.
\end{align*}
When \eqref{CFL_cond} holds,
all the coefficients $S_i\ (i=1,2,\cdots,9)$ on the right hand side of \eqref{2Deqn7bb} are nonnegative. Meanwhile $\sum_{i=1}^9 S_i = 1.$ Therefore it holds
 \begin{align*}
   |u_{ij}^{k+1} | \leq \|u^k\|_{\infty} + \tau\|r^k\|_{\infty},\quad (i,j)\in \omega,\;0\le k\le n-1.
 \end{align*}
   Whence
  \begin{align*}
    \|u^{k+1} \|_{\infty} \leq \|u^k\|_{\infty} + \tau\|r^k\|_{\infty},\quad 0\le k\le n-1.
  \end{align*}
  By recursion, we have
  \begin{align*}
    \|u^k\|_{\infty} \leq \|u^0\|_{\infty} + \tau \sum_{l=0}^{k-1}\|r^l\|_{\infty}= \|\varphi\|_{\infty} + \tau \sum_{l=0}^{k-1}\|r^l\|_{\infty},  \quad 1\leq k \leq n.
  \end{align*}
\end{proof}
\begin{remark}
 Based on the analysis in the above, we have the following conclusions.
  \begin{itemize}
    \item   The condition \eqref{CFL_cond} is called as Courant-Friedrichs-Lewy ({\rm \textbf{CFL}}) condition. The proof above implies that the difference scheme  \eqref{2Deqn7} is stable under {\rm \textbf{CFL}} condition \eqref{CFL_cond}.
    \item   When $c=d=0$ in \eqref{eqn2}, {\rm \textbf{CFL}} condition \eqref{CFL_cond} is simplified to \eqref{CFL_cond1}. {\rm \textbf{CFL}} condition \eqref{CFL_cond1} of the corrected difference scheme \eqref{2Deqn5} for the diffusion problem
           obviously improves that of \eqref{2Deqn6}, for which {\rm \textbf{CFL}} condition is $r_x+r_y \leq 1/2$.
  \end{itemize}
\end{remark}
Furthermore, we have the convergence result.
\begin{theorem}\label{thm2}
 Let $\{U^k_{ij}\,|\,0\leq i\leq m_1, \ 0\leq j\leq m_2, \ 0 \leq k \leq n\}$ be the solution of \eqref{eqn2}
 and $\{u^k_{ij}\,|\,0\leq i\leq m_1,\ 0\leq j \leq m_2, \ 0 \leq k \leq n\}$ be the solution of the difference scheme \eqref{2Deqn7}. Denote
$e^k_{ij}=U^k_{ij}-u^k_{ij},\; (0\leq i\leq m_1,\ 0 \leq j \leq m_2, \ 0 \leq k \leq n).$
When \eqref{CFL_cond} holds,
then there is a constant $c_1$ such that
\begin{equation*}
\|e^k\|_\infty\leq \left\{
\begin{array}{ll}
 c_1(\tau^2+h_x^4+h_y^4),\quad  & {\rm \quad if \quad} r_x=r_y=\frac{1}{6},\quad  0 \leq k \leq n,\\ [1\jot]
 c_1(\tau^2+h_x^2+h_y^2),\quad  & {\rm \quad otherwise. \quad}
\end{array}\right.
\end{equation*}
 \end{theorem}
 This theorem can be easily proved in combination with Theorem \ref{thm1} and the local truncation error \eqref{Local_err}.

\section{The nonlinear problems}\label{Sec4}
\setcounter{equation}{0}
The corrected technique can be extended to two-dimensional nonlinear convection-diffusion equations such as semi-linear and quasi-linear parabolic problems.
In what follows, we consider the numerical solution of the nonlinear convection-diffusion equation \eqref{eqn4.1}.
With the help of \eqref{eqn4.1a} we have
\begin{align}\label{eqn4.2}
 u_{tt} =&\; -f(u)_{xt}-g(u)_{yt}+au_{xxt}+bu_{yyt}+r(u)_t \nonumber\\
        =&\; -\big(f'(u)u_t\big)_x-\big(g'(u)u_t\big)_y+a(u_t)_{xx}+b(u_t)_{yy}+r'(u)u_t \nonumber\\
        =&\; a^2u_{xxxx}+b^2u_{yyyy}+2abu_{xxyy}-2af'(u)u_{xxx}\notag\\
         &\;-2bg'(u)_{yyy}-2bf'(u)u_{xyy}-2ag'(u)u_{xxy} \nonumber\\
         &\; -4af''(u)u_xu_{xx}-4bg''(u)u_yu_{yy}-2ag''(u)u_yu_{xx}\notag\\
         &\; -2bf''(u)u_xu_{yy}-2ag''(u)u_xu_{xy}-2af''(u)u_yu_{xy} \nonumber\\
         &\; +\big(f'(u)^2+2ar'(u)\big)u_{xx}+\big(g'(u)^2+2br'(u)\big)u_{yy}+2f'(u)g'(u)u_{xy} \nonumber\\
         &\; -af'''(u)u_x^3-bg'''(u)u_y^3-ag'''(u)u_yu_x^2-bf'''(u)u_xu_y^2 \nonumber\\
         &\; +\big(2f'(u)f''(u)+ar''(u)\big)u_x^2+\big(2g'(u)g''(u)+br''(u)\big)u_y^2\notag\\
         &\;+2\big(f''(u)g'(u)+f'(u)g''(u)\big)u_xu_y   -\big(f''(u)r(u)+2f'(u)r'(u)\big)u_x\notag\\
         &\;-\big(g''(u)r(u)+2g'(u)r'(u)\big)u_y+r'(u)r(u).
\end{align}
Considering \eqref{eqn4.1a} at the point $(\textbf{x}_{ij},t_k)$ and combining the forward difference quotient for the temporal derivative with the central difference discretization for the spatial derivatives, it follows that
\begin{align*}
  \delta_t U_{ij}^{k+\frac{1}{2}}
 =&\; -f'(U_{ij}^k)\Delta_xU_{ij}^k-g'(U_{ij}^k)\Delta_yU_{ij}^k+a\delta_x^2U_{ij}^k+b\delta_y^2U_{ij}^k+r(U_{ij}^k)\\
  &\; +\Big\{\frac{\tau}{2}u_{tt}+\frac{h_x^2}{6}f'(u)u_{xxx}+\frac{h_y^2}{6}g'(u)u_{yyy}
      -\frac{ah_x^2}{12}u_{xxxx}-\frac{bh_y^2}{12}u_{yyyy}\Big\}(\mathbf{x}_{ij},t_k)\\
  &\; +O(\tau^2+h_x^4+h_y^4)\\
 =&\; -f'(U_{ij}^k)\Delta_xU_{ij}^k-g'(U_{ij}^k)\Delta_yU_{ij}^k+a\delta_x^2U_{ij}^k+b\delta_y^2U_{ij}^k+r(U_{ij}^k)\\
  &\; +\Big\{h_x^2\big(r_x-\frac{1}{6}\big)\big(\frac{a}{2}u_{xxxx}-u_{xxx}\big)+h_y^2\big(r_y-\frac{1}{6}\big)\big(\frac{b}{2}u_{yyyy}
  -u_{yyy}\big)\Big\}(\mathbf{x}_{ij},t_k)\\
  &\; +\tau\Big\{abu_{xxyy}-bf'(u)u_{xyy}-ag'(u)u_{xxy}\Big\}(\mathbf{x}_{ij},t_k)\\
  &\; -\tau\Big\{2af''(u)u_xu_{xx}+2bg''(u)u_yu_{yy}+ag''(u)u_yu_{xx}\\
  &\; \quad +bf''(u)u_xu_{yy}+ag''(u)u_xu_{xy}+af''(u)u_yu_{xy}\Big\}(\mathbf{x}_{ij},t_k)\\
  &\; +\frac{\tau}{2}\Big\{\big(f'(u)^2+2ar'(u)\big)u_{xx}+\big(g'(u)^2+2br'(u)\big)u_{yy}+2f'(u)g'(u)u_{xy}\Big\}(\mathbf{x}_{ij},t_k)\\
  &\; -\frac{\tau}{2}\Big\{af'''(u)u_x^3+bg'''(u)u_y^3+ag'''(u)u_yu_x^2+bf'''(u)u_xu_y^2\Big\}(\mathbf{x}_{ij},t_k)\\
  &\; +\frac{\tau}{2}\Big\{\big(2f'(u)f''(u)+ar''(u)\big)u_x^2+\big(2g'(u)g''(u)\\
  &\;\quad +br''(u)\big)u_y^2+2\big(f''(u)g'(u)+f'(u)g''(u)\big)u_xu_y\Big\}(\mathbf{x}_{ij},t_k)\\
  &\; -\frac{\tau}{2}\Big\{\big(f''(u)r(u)+2f'(u)r'(u)\big)u_x+\big(g''(u)r(u)+2g'(u)r'(u)\big)u_y-r'(u)r(u)\Big\}(\mathbf{x}_{ij},t_k)\\
  &\;+O(\tau^2+h_x^4+h_y^4),
  \end{align*}
   where the second equality has utilized \eqref{eqn4.2}.
  Further applying the centered difference discretization to the remaining partial derivatives including first-order, second-order and mixed derivatives in space, and then rearranging the corresponding result, we have
  \begin{align*}
\delta_t U_{ij}^{k+\frac{1}{2}}  =
  &\; \Big[a+\frac{\tau}{2}\big(f'(U_{ij}^k)^2+2ar'(U_{ij}^k)\big)\Big]\delta_x^2U_{ij}^k
  +\Big[b+\frac{\tau}{2}\big(g'(U_{ij}^k)^2+2br'(U_{ij}^k)\big)\Big]\delta_y^2U_{ij}^k\\
  &\; -\Big[f'(U_{ij}^k)+\frac{\tau}{2}\big(f''(U_{ij}^k)r(U_{ij}^k)+2f'(U_{ij}^k)r'(U_{ij}^k)\big)\Big]\Delta_xU_{ij}^k\\
  &\; -\Big[g'(U_{ij}^k)+\frac{\tau}{2}\big(g''(U_{ij}^k)r(U_{ij}^k)+2g'(U_{ij}^k)r'(U_{ij}^k)\big)\Big]\Delta_yU_{ij}^k+r(U_{ij}^k)\\
&\; +\Big\{h_x^2\big(r_x-\frac{1}{6}\big)\big(\frac{a}{2}u_{xxxx}-u_{xxx}\big)
+h_y^2\big(r_y-\frac{1}{6}\big)\big(\frac{b}{2}u_{yyyy}-u_{yyy}\big)\Big\}(\mathbf{x}_{ij},t_k)\\
  &\; +\tau\Big[ ab\delta_x^2\delta_y^2U_{ij}^k-b  f'(U_{ij}^k)\Delta_x\delta_y^2U_{ij}^k-a  g'(U_{ij}^k)\Delta_y\delta_x^2U_{ij}^k\\
  &\; -2a  f''(U_{ij}^k)\Delta_xU_{ij}^k\delta_x^2U_{ij}^k-2b  g''(U_{ij}^k)\Delta_yU_{ij}^k\delta_y^2U_{ij}^k-a  g''(U_{ij}^k)\Delta_yU_{ij}^k\delta_x^2U_{ij}^k\\
  &\; -b  f''(U_{ij}^k)\Delta_xU_{ij}^k\delta_y^2U_{ij}^k-a  g''(U_{ij}^k)\Delta_xU_{ij}^k\Delta_x\Delta_yU_{ij}^k-b f''(U_{ij}^k)\Delta_yU_{ij}^k\Delta_x\Delta_yU_{ij}^k\\
  &\; +  f'(U_{ij}^k)g'(U_{ij}^k)\Delta_x\Delta_yU_{ij}^k-\frac{a}{2}f'''(U_{ij}^k)(\Delta_xU_{ij}^k)^3
  -\frac{b}{2}g'''(U_{ij}^k)(\Delta_yU_{ij}^k)^3\\
  &\; -\frac{a}{2}g'''(U_{ij}^k)\Delta_yU_{ij}^k(\Delta_xU_{ij}^k)^2-\frac{b}{2}f'''(U_{ij}^k)\Delta_xU_{ij}^k(\Delta_yU_{ij}^k)^2\\
  &\; +\frac{1}{2}\big(2f'(U_{ij}^k)f''(U_{ij}^k)+ar''(U_{ij}^k)\big)(\Delta_xU_{ij}^k)^2  +\frac{1}{2}\big(2g'(U_{ij}^k)g''(U_{ij}^k)+br''(U_{ij}^k)\big)(\Delta_yU_{ij}^k)^2\\
  &\; +\big(f''(U_{ij}^k)g'(U_{ij}^k)+f'(U_{ij}^k)g''(U_{ij}^k)\big)\Delta_xU_{ij}^k\Delta_yU_{ij}^k\\
  &\; +\frac{1}{2}r'(U_{ij}^k)r(U_{ij}^k)\Big]+O(\tau^2+\tau h_x^2+\tau h_y^2+\tau h_x^2h_y^2+h_x^4+h_y^4).
\end{align*}
 Therefore, a corrected difference scheme for \eqref{eqn4.1} reads
\begin{subequations}
\label{eqn4.3}
\begin{numcases}{}
  \;\delta_t u_{ij}^{k+\frac{1}{2}}
 = \Big[a+\frac{\tau}{2}\big(f'(u_{ij}^k)^2+2ar'(u_{ij}^k)\big)-2a\tau f''(u_{ij}^k)\Delta_xu_{ij}^k-a\tau g''(u_{ij}^k)\Delta_yu_{ij}^k\Big]\delta_x^2u_{ij}^k \nonumber\\
  \;\qquad \qquad +\Big[b+\frac{\tau}{2}\big(g'(u_{ij}^k)^2+2br'(u_{ij}^k)\big)-2b\tau g''(u_{ij}^k)\Delta_yu_{ij}^k-b\tau f''(u_{ij}^k)\Delta_xu_{ij}^k\Big]\delta_y^2u_{ij}^k \nonumber\\
  \;\qquad \qquad -\Big[f'(u_{ij}^k)+\frac{\tau}{2}\big(f''(u_{ij}^k)r(u_{ij}^k)+2f'(u_{ij}^k)r'(u_{ij}^k)\big)\Big]\Delta_xu_{ij}^k \nonumber\\
  \; \qquad \qquad-\Big[g'(u_{ij}^k)+\frac{\tau}{2}\big(g''(u_{ij}^k)r(u_{ij}^k)+2g'(u_{ij}^k)r'(u_{ij}^k)\big)\Big]\Delta_yu_{ij}^k  +r(u_{ij}^k)\notag\\
  \; \qquad \qquad+\tau ab\delta_x^2\delta_y^2u_{ij}^k-b\tau f'(u_{ij}^k)\Delta_x\delta_y^2u_{ij}^k-a\tau g'(u_{ij}^k)\Delta_y\delta_x^2u_{ij}^k \nonumber\\
  \;\qquad \qquad +\frac{\tau}{2}\big[2f'(u_{ij}^k)f''(u_{ij}^k)+ar''(u_{ij}^k)-af'''(u_{ij}^k)\Delta_xu_{ij}^k-ag'''(u_{ij}^k)\Delta_yu_{ij}^k\big](\Delta_xu_{ij}^k)^2 \nonumber\\
  \; \qquad \qquad +\frac{\tau}{2}\big[2g'(u_{ij}^k)g''(u_{ij}^k)+br''(u_{ij}^k)-bg'''(u_{ij}^k)\Delta_yu_{ij}^k-bf'''(u_{ij}^k)\Delta_xu_{ij}^k\big](\Delta_yu_{ij}^k)^2 \nonumber\\
  \;\qquad \qquad +\tau\big(f'(u_{ij}^k)g'(u_{ij}^k)-ag''(u_{ij}^k)\Delta_xu_{ij}^k-bf''(u_{ij}^k)\Delta_yu_{ij}^k\big)\Delta_x\Delta_yu_{ij}^k \nonumber\\
  \; \qquad \qquad+\tau\big(f''(u_{ij}^k)g'(u_{ij}^k)+f'(u_{ij}^k)g''(u_{ij}^k)\big)\Delta_xu_{ij}^k\cdot\Delta_yu_{ij}^k +\frac{\tau}{2}r'(u_{ij}^k)r(u_{ij}^k),\notag \\
  \;\qquad \qquad \qquad \qquad \qquad \qquad (i,j)\in \omega,\; 0\leq k \leq n-1,\label{eqn4.3a}\\
  \; u_{ij}^0=\varphi(\mathbf{x}_{ij}), \quad\ \! (i,j)\in \bar\omega,\; \label{eqn4.3b}\\
  \;u_{ij}^k=\alpha(\mathbf{x}_{ij},t_k), \quad (i,j)\in \gamma, \; 1 \leq k \leq n.\label{eqn4.3c}
\end{numcases}
\end{subequations}
\section{Problems under Neumann boundary value conditions}\label{Sec5}
\setcounter{equation}{0}
  In this section, we extend our idea to the general nonlinear convection-diffusion equation under Neumann boundary value conditions as
\begin{figure}[ht]
	\begin{center}
	\includegraphics[width=9cm,height=6cm]{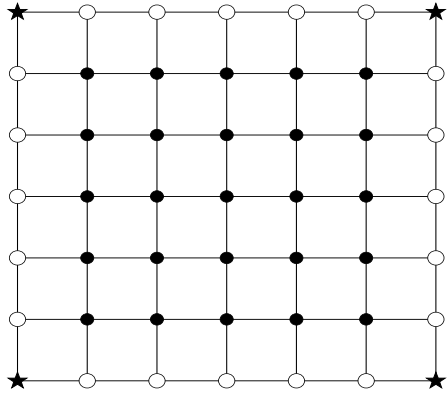}
	\caption{The sketch map in a rectangular domain}
	\label{sketch}
\end{center}
\end{figure}
\begin{subequations}
\label{eqn6.1}
\begin{numcases}{}
u_{t}+f(u)_x+g(u)_y=au_{xx}+bu_{yy}+r(u),\quad \mathbf{x}\in \Omega,\; t\in (0,T],  \label{eqn6.1a} \\
u(\mathbf{x},0)=\varphi(\mathbf{x}),\quad \mathbf{x}\in \bar{\Omega},\label{eqn6.1b} \\
u_x(L_1,y,t)=\alpha_1(y,t),\quad u_x(R_1,y,t)=\alpha_2(y,t), \quad y\in [L_2,R_2],\quad t \in (0,T], \label{eqn6.1c} \\
u_y(x,L_2,t)=\beta_1(x,t),\quad u_y(x,R_2,t)=\beta_2(x,t), \quad x\in (L_1,R_1),\quad t\in (0,T]. \label{eqn6.1d}
\end{numcases}
\end{subequations}
Since \eqref{eqn6.1a} is the same to \eqref{eqn4.1a}, it is necessary to discrete the boundary value conditions \eqref{eqn6.1c}--\eqref{eqn6.1d}
with a fourth-order algorithm. For example, based on the fourth-order backward difference formula \cite{HNW1987}, we have
\begin{subequations}
\label{eqna}
\begin{numcases}{}
 u_x(L_1,y_j,t_k) = \frac{1}{h_x}\Big(-\frac{25}{12}U_{0j}^k + 4U_{1j}^k-3U_{2j}^k+\frac{4}{3}U_{3j}^k-\frac{1}{4}U_{4j}^k\Big)+O(h_x^4),\notag\\
  \qquad \qquad \qquad \qquad \qquad \qquad 0 \leq j \leq m_2,\; 0\leq k\leq n, \label{eqn6.2a}\\
 u_x(R_1,y_j,t_k) = \frac{1}{h_x}\Big(\frac{1}{4}U_{m_1-4,j}^k-\frac{4}{3}U_{m_1-3,j}^k+3U_{m_1-2,j}^k-4U_{m_1-1,j}^k+\frac{25}{12}U_{m_1,j}^k\Big)+O(h_x^4),\nonumber\\
 \qquad \qquad \qquad \qquad \qquad \qquad 0 \leq j \leq m_2,\; 0\leq k\leq n, \label{eqn6.3a}\\
 u_y(x_i,L_2,t_k) = \frac{1}{h_y}\Big(-\frac{25}{12}U_{i0}^k + 4U_{i1}^k-3U_{i2}^k+\frac{4}{3}U_{i3}^k-\frac{1}{4}U_{i4}^k\Big)+O(h_y^4),\notag\\
  \qquad \qquad \qquad \qquad \qquad \qquad  1 \leq i \leq m_1-1,\; 0\leq k\leq n, \label{eqn6.4a}\\
 u_y(x_i,R_2,t_k) = \frac{1}{h_y}\Big(\frac{1}{4}U_{i,m_2-4}^k-\frac{4}{3}U_{i,m_2-3}^k+3U_{i,m_2-2}^k-4U_{i,m_2-1}^k+\frac{25}{12}U_{i,m_2}^k\Big)+O(h_y^4),\nonumber\\
 \qquad \qquad \qquad \qquad \qquad \qquad  1 \leq i \leq m_1-1,\; 0\leq k\leq n. \label{eqn6.5a}
\end{numcases}
\end{subequations}

Omitting the small terms in \eqref{eqn6.2a}--\eqref{eqn6.5a}, we have the discrete schemes for the boundary value conditions \eqref{eqn6.1c}--\eqref{eqn6.1d} as
\begin{subequations}
\label{eqn_boundary}
\begin{numcases}{}
 \frac{1}{h_x}\Big(-\frac{25}{12}u_{0j}^k + 4u_{1j}^k-3u_{2j}^k+\frac{4}{3}u_{3j}^k-\frac{1}{4}u_{4j}^k\Big)=\alpha_1(y_j,t_k),\notag\\
  \qquad \qquad \qquad \qquad 0 \leq j \leq m_2,\; 0\leq k\leq n, \label{eqn6.2aa}\\
 \frac{1}{h_x}\Big(\frac{1}{4}u_{m_1-4,j}^k-\frac{4}{3}u_{m_1-3,j}^k+3u_{m_1-2,j}^k-4u_{m_1-1,j}^k+\frac{25}{12}u_{m_1,j}^k\Big)=\alpha_2(y_j,t_k),\notag\\
\qquad \qquad \qquad \qquad  0 \leq j \leq m_2,\; 0\leq k\leq n, \label{eqn6.3aa}\\
 \frac{1}{h_y}\Big(-\frac{25}{12}u_{i0}^k + 4U_{i1}^k-3u_{i2}^k+\frac{4}{3}u_{i3}^k-\frac{1}{4}u_{i4}^k\Big)=\beta_1(x_i,t_k),\notag\\
  \qquad \qquad \qquad \qquad  1 \leq i \leq m_1-1,\; 0\leq k\leq n, \label{eqn6.4aa}\\
 \frac{1}{h_y}\Big(\frac{1}{4}u_{i,m_2-4}^k-\frac{4}{3}u_{i,m_2-3}^k+3U_{i,m_2-2}^k-4u_{i,m_2-1}^k+\frac{25}{12}u_{i,m_2}^k\Big) = \beta_2(x_i,t_k),\notag\\
 \qquad \qquad \qquad \qquad  1 \leq i \leq m_1-1,\; 0\leq k\leq n. \label{eqn6.5aa}
\end{numcases}
\end{subequations}

In the numerical implement, we compute the numerical solutions at $(k+1)th$-level by the following three steps. The sketch map is referred to
Figure \ref{sketch}:
\begin{enumerate}
  \item [{\rm $\mathbf{(i)}$}] Compute the numerical solutions on the solid points by \eqref{eqn4.3} at $kth$-level;
  \item [{\rm $\mathbf{(ii)}$}] Compute the numerical solutions on the hollow points by \eqref{eqn6.2aa}--\eqref{eqn6.5aa} and the numerical solutions on solid points;
  \item [{\rm $\mathbf{(iii)}$}] Compute the numerical solutions on the star points by \eqref{eqn6.2aa}--\eqref{eqn6.3aa} and the numerical solutions on hollow points.
\end{enumerate}
\section{Three-dimensional diffusion problem} \label{SecDiffu}
\setcounter{equation}{0}
We consider the three-dimensional convection-diffusion problem as
\begin{subequations}
\label{3D_diffu1}
\begin{numcases}{}
u_{t}=\kappa_1  u_{xx} + \kappa_2 u_{yy} + \kappa_3 u_{zz}+\lambda_1 u_x +\lambda_2 u_y+\lambda_3 u_z+ f(\mathbf{x},t),\quad  \mathbf{x} \in \Omega,\; t\in (0,T],  \label{3Deqn2a} \\
u(\mathbf{x},0)=\varphi(\mathbf{x}),\quad \mathbf{x}\in \bar{\Omega},\label{3Deqn2b} \\
u(\mathbf{x},t)=\alpha(\mathbf{x},t), \quad \mathbf{x}\in\Gamma,\; t\in (0,T], \label{3Deqn2c}
\end{numcases}
\end{subequations}
where $\mathbf{x}=(x,y,z)$ and $\kappa_i > 0\; (i=1,2,3)$ denote constant diffusion coefficients.
The spatial domain is set as  $\Omega = (L_1,R_1)\times(L_2,R_2)\times(L_3,R_3)\subset \mathbb{R}^3$.

To ease the notation, we only state the details for $\lambda_i=0\;(i=1,2,3)$.
In addition to notations defined in Section \ref{Sec2}, we take one more integer $m_3$ and let
$h_z = (R_3-L_3)/m_3$, $z_l = L_3+lh_z$, $0\leq l \leq m_3$, $r_z = c\tau/h_z^2$.
Denote $\Omega_{h\tau}=\{(\mathbf{x}_{ijl},t_k)\,|\,0\leq i\leq m_1, \ 0\leq j\leq m_2, \ 0\leq l \leq m_3, \ 0\leq k\leq n\}$ with $\mathbf{x}_{ijl} = (x_i,y_j,z_l)$, $\omega = \{(i,j,l)\,|\,\mathbf{x}_{ijl}\in \Omega\}$, $\gamma = \{(i,j,l)\,|\,\mathbf{x}_{ijl}\in \Gamma\}$ and
$\bar\omega =\omega \cup \gamma.$
For any grid function
$v=\{v_{ijl}^k\,|\, 0\leq i\leq m_1, 0\leq j \leq m_2, 0\leq l \leq m_3, 0\le k\le n\}$ on $\Omega_{h\tau},$
define $\delta_t v_{ijl}^{k+\frac{1}{2}} = (v_{ijl}^{k+1}-v_{ijl}^k)/\tau$,
$\delta_x^2 v_{ijl}^{k} = (v_{i+1,j,l}^k-2v_{ijl}^k+v_{i-1,j,l}^k)/h_x^2$.
Analogously, we could define $\delta_y^2 v_{ijl}^{k}$ and $\delta_z^2 v_{ijl}^{k}$.
Denote $f_{ijl}^k = f(\mathbf{x}_{ijl},t_k)$ and define the grid function
$$U=\{ U_{ijl}^k\,|\,0\leq i\leq m_1,\; 0 \leq j \leq m_2,\;  0 \leq l \leq m_3,\; 0\leq k\leq n\}\ {\rm with\ }U_{ijl}^k=u(\mathbf{x}_{ijl},t_k).$$

Considering \eqref{3Deqn2a} at the point $(\mathbf{x}_{ijl},t_k)$ and similar to that of two dimension, we have
\begin{align}
   &\;\delta_t U_{ijl}^{k+\frac{1}{2}} - \kappa_1\delta_x^2U_{ijl}^k - \kappa_2\delta_y^2 U_{ijl}^k - \kappa_3\delta_z^2 U_{ijl}^k \notag\\
  =&\; f_{ijl}^k+\Big(\frac{\tau}{2}u_{tt} - \frac{\kappa_1 h_x^2}{12}u_{xxxx} - \frac{\kappa_2 h_y^2}{12}u_{yyyy} - \frac{\kappa_3h_z^2}{12}u_{zzzz}\Big)(\mathbf{x}_{ijl},t_k)
       + O(\tau^2+h_x^4+h_y^4+h_z^4)\notag\\
  =&\; p_{ijl}^k+\left\{\frac{\kappa_1 h_x^2}{2}\Big(r_x-\frac{1}{6}\Big)u_{xxxx} + \frac{\kappa_2 h_y^2}{2}\Big(r_y-\frac{1}{6}\Big)u_{yyyy}
       + \frac{\kappa_3 h_z^2}{2}\Big(r_z-\frac{1}{6}\Big)u_{zzzz}\right\}(\mathbf{x}_{ijl},t_k) \notag\\
   &\; + \tau \kappa_1\kappa_2\delta_x^2\delta_y^2 U_{ijl}^k + \tau \kappa_2\kappa_3\delta_y^2\delta_z^2 U_{ijl}^k + \tau \kappa_1\kappa_3\delta_x^2\delta_z^2 U_{ijl}^k
       + O(\tau^2+\tau h_x^2+\tau h_y^2+\tau h_z^2+h_x^4+h_y^4+h_z^4),\notag
\end{align}
where
\begin{align*}
  p_{ijl}^k = f_{ijl}^k +\frac{\tau}{2}\left[a(f_{xx})_{ijl}^k + b(f_{yy})_{ijl}^k + c(f_{zz})_{ijl}^k + (f_t)_{ijl}^k \right]
\end{align*}
with $(f_{xx})_{ijl}^k = f_{xx}(\mathbf{x}_{ijl},t_k),\; (f_{yy})_{ijl}^k = f_{yy}(\mathbf{x}_{ijl},t_k),\; (f_{zz})_{ijl}^k = f_{zz}(\mathbf{x}_{ijl},t_k),\;(f_t)_{ijl}^k = f_t(\mathbf{x}_{ijl},t_k).$

Therefore, a corrected difference scheme for \eqref{3D_diffu1} is constructed as
\begin{subequations}
\label{3D_diffu2}
\begin{numcases}{}
   \delta_t u_{ijl}^{k+\frac{1}{2}} = \kappa_1\delta_x^2u_{ijl}^k + \kappa_2\delta_y^2 u_{ijl}^k + \kappa_3\delta_z^2 u_{ijl}^k
   + p_{ijl}^k  \notag
   \\ \quad  + \tau \kappa_1\kappa_2\delta_x^2\delta_y^2 u_{ijl}^k+ \tau \kappa_2\kappa_3\delta_y^2\delta_z^2 u_{ijl}^k + \tau \kappa_1\kappa_3\delta_x^2\delta_z^2 u_{ijl}^k, \quad (i,j,l)\in\omega,\; 0 \leq k \leq n-1, \label{3Deqn3a}\\
     u_{ijl}^0=\varphi(\mathbf{x}_{ijl}), \quad \qquad \ \! (i,j,l)\in \bar\omega,\; \label{3Deqn3b}\\
   u_{ijl}^k=\alpha(\mathbf{x}_{ijl},t_k),\qquad (i,j,l)\in \gamma, \; 0 \leq k \leq n.\label{3Deqn3c}
\end{numcases}
\end{subequations}

By using the maximum principle and similar to the proof in two dimension, we have the convergence result for the corrected difference scheme \eqref{3D_diffu2}.
\begin{theorem}\label{thm2}
 Let $\{U^k_{ijl}\,|\,0\leq i\leq m_1, \ 0\leq j\leq m_2, \ 0\leq l\leq m_3, \ 0 \leq k \leq n\}$ be the solution of \eqref{3D_diffu1} with $\lambda_i=0\;(i=1,2,3)$
 and $\{u^k_{ijl}\,|\,0\leq i\leq m_1,\ 0\leq j \leq m_2, \ 0\leq l\leq m_3, \ 0 \leq k \leq n\}$ be the solution of the corrected difference scheme \eqref{3D_diffu2}. Denote
$e^k_{ijl}=U^k_{ijl}-u^k_{ijl},\; (0\leq i\leq m_1,\ 0 \leq j \leq m_2, \ 0\leq l\leq m_3, \ 0 \leq k \leq n).$
When
\begin{align}
\max\{r_x,r_y,r_z\} \leq {1/4},\label{3d_CFL}
\end{align}
then there is a constant $c_2$ such that
\begin{equation*}
\|e^k\|_\infty\leq \left\{
\begin{array}{ll}
 c_2(\tau^2+h_x^4+h_y^4+h_z^4),\quad        & {\rm \quad when \quad} r_x=r_y=r_z=1/6,\quad  0 \leq k \leq n,\\ [1\jot]
 c_2(\tau^2+h_x^2+h_y^2+h_z^2),\quad  & {\rm \quad otherwise}.
\end{array}\right.
\end{equation*}
 \end{theorem}
        \vspace{-10mm}
          \begin{figure}[ht]
	\begin{center}
	\includegraphics[width=9cm,height=6cm]{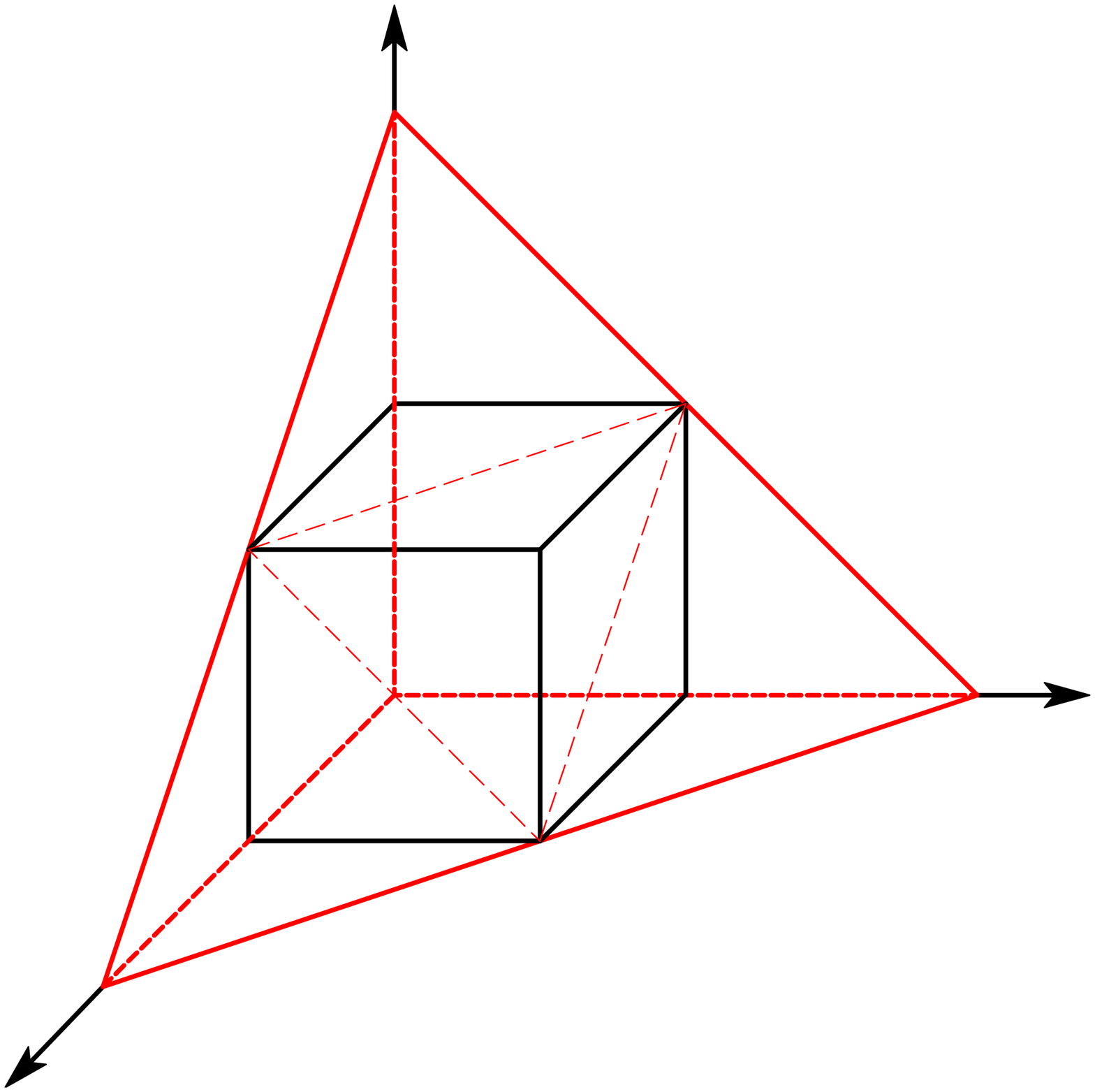}
	\caption{The sketch of the stable domain for the three-dimensional diffusion problem. Red triangular pyramid denotes the stability of the classical difference scheme \eqref{scheme_3D_class} and black cube denotes the stability of the corrected difference scheme \eqref{3D_diffu2}}
	\label{sketch2}
\end{center}
\end{figure}
 \begin{remark}\label{rem6.3}
 The following results are easily obtained from the convergence proof.
  \begin{itemize}
    \item \eqref{3d_CFL} is also {\rm \textbf{CFL}} condition for the stability of the corrected difference scheme \eqref{3D_diffu2}. {\rm \textbf{CFL}} condition of the classical difference scheme
\begin{subequations}
\label{scheme_3D_class}
\begin{numcases}{}
   \delta_t u_{ijl}^{k+\frac{1}{2}} = \kappa_1\delta_x^2u_{ijl}^k + \kappa_2\delta_y^2 u_{ijl}^k + \kappa_3\delta_z^2 u_{ijl}^k
   + p_{ijl}^k, \\
     u_{ijl}^0=\varphi(\mathbf{x}_{ijl}), \quad \qquad \ \! (i,j,l)\in \bar\omega,\;  \\
   u_{ijl}^k=\alpha(\mathbf{x}_{ijl},t_k),\qquad (i,j,l)\in \gamma, \; 0 \leq k \leq n,
\end{numcases}
\end{subequations}
        requires $r_x+r_y+r_z \leq 1/2$, see e.g., \cite{Guo1988}. The stability regions between the corrected difference scheme (cube) and the classical difference scheme (rectangular triangular pyramid) are also referred to Figure \ref{sketch2}.
    \item For the convection-diffusion problem with constant convection coefficients \eqref{3D_diffu1} in three dimension, {\rm \textbf{CFL}} condition becomes
        \begin{subequations}
\begin{numcases}{}
\max\{r_x, r_y,r_z\} \leq 1/4,\notag\\
 |\lambda_1|h_x \leq \kappa_1 \cdot \min\Big\{2, \sqrt{(1-2r_x)(1-2r_y)(1-2r_z)/3}\Big/r_x\Big\},\notag\\
  |\lambda_2|h_y \leq \kappa_2 \cdot \min\Big\{2, \sqrt{(1-2r_x)(1-2r_y)(1-2r_z)/3}\Big/r_y\Big\},\notag\\
   |\lambda_3|h_z \leq \kappa_3 \cdot \min\Big\{2, \sqrt{(1-2r_x)(1-2r_y)(1-2r_z)/3}\Big/r_z\Big\}.\notag
\end{numcases}
\end{subequations}
We omit details here for sake of brevity.
  \end{itemize}
\end{remark}

\section{Numerical examples}\label{Sec6}
\setcounter{equation}{0}
In this section, we verify the accuracy and test {\rm \textbf{CFL}} condition of the numerical simulation for several problems including linear and nonlinear cases in different scenarios. We first focus on the two-dimensional case. For this purpose,
we denote the difference solution $\{u_{ij}^k\}$ with the chosen fixed spacial step sizes $(h_x,h_y)$ and temporal step size $\tau$ by $\{u_{ij}^k(h_x,h_y,\tau)\}.$ Similarly, $\{u_{2i,2j}^{4k}(h_x/2,h_y/2,\tau/4)\}$ denotes the difference solution with the chosen fixed spacial step sizes $(h_x/2,h_y/2)$ and temporal step size $\tau/4$.
Let $r_x=r_y:=r$. The numerical errors in $L^{\infty}$-norm and global convergence orders are defined as
{\footnotesize{\begin{equation}\label{eqn5.1}
{\rm Ord}_G=\left\{
  \begin{array}{ll}
  \log_2\left(\frac{E_\infty(h_x,h_y,\tau)}{E_\infty(h_x/2,h_y/2,\tau/4)}\right),  & {\rm where~~~}   E_\infty(h_x,h_y,\tau)\!=\!\max_{\substack{(\mathbf{x}_{ij},t_k)\in \Omega_{h\tau}}}\limits\big| u(\mathbf{x}_{ij},t_k) -u_{ij}^k(h_x,h_y,\tau)\big|, \\
  &{\rm ~if~the~exact~solution~is~known;}\\ [1\jot]
  \log_2\left(\frac{E_\infty(h_x,h_y,\tau)}{E_\infty(h_x/2,h_y/2,\tau/4)}\right),  & {\rm where~~~}   E_\infty(h_x,h_y,\tau)\!=\!\max_{\substack{(\mathbf{x}_{ij},t_k)\in \Omega_{h\tau}}}\limits\big| u_{ij}^{k}(h_x,h_y,\tau) \\
  &-u_{2i,2j}^{4k}(h_x/2,h_y/2,\tau/4)\big|, {\rm ~if~the~exact~solution~is~unknown.}
  \end{array}\right.
\end{equation}
}
}

In the numerical implementation, we first fixed the step-ratio $r$ and temporal step size $\tau$, then determine $h_x$ and $h_y$ by $r$ and $\tau$.

Similarly, the numerical errors and  global convergence orders for the three-dimensional case could be defined.
The step-ratio is abbreviated as ``R'' and the grid parameter as ``P''.
\vspace{-5mm}
\subsection{Linear problems}
\vspace{-4mm}
\begin{example}\label{exam1}
   Firstly, the model problem \eqref{eqn1} with $c=d=0$ is solved by the correct difference scheme \eqref{2Deqn5} and classical Euler difference method \eqref{2Deqn6} respectively, with the parameters $L_1=L_2=0$, $R_1=R_2=1$, $T=1$. The initial value condition, boundary value condition and the source term $f(\mathbf{x},t)$ are determined by the exact solution $u(\mathbf{x},t)=\exp(1/2(x+y)-t)$, see e.g.,  {\rm \cite{Sun2022}}.

     Two sets of diffusion coefficients {\rm \textbf{Case~I:}} $a=4,\ b =1${\rm \, and\ }{\rm \textbf{Case~II:}} $a=1,\ b = 0.0001$
   are used. Numerical results are listed in Tables \ref{tab1} and \ref{table1}.

In {\rm \textbf{Case I}}, we see clearly that when the step ratio $r = 1/6$,
the corrected difference scheme \eqref{2Deqn5} obtains the fourth-order accuracy approximatively. Otherwise, it is only second-order accurate globally. When the step-ratio increases to $r=1/2$ gradually, the corrected difference scheme \eqref{2Deqn5} still works. However, the classical difference scheme \eqref{2Deqn6} fails even if the step-ratio is $r=1/3.99$ because of the restriction of {\rm \textbf{CFL}} condition. These numerical observations are surprisingly consistent with the theoretical results.
Moreover, the numerical results in Table \ref{tab1} are more accurate than those in Table \ref{table1} even using the same step-ratio, which display the advantage of the corrected Euler difference scheme \eqref{2Deqn5}.
In {\rm \textbf{Case II}}, similar results to {\rm \textbf{Case I}} are observed from the third column in Table \ref{tab1} and Table \ref{table1} even though the diffusion varies greatly from direction to direction.

\begin{table}[H]
\begin{center}
\caption{The errors in $L^{\infty}$-norm versus grid sizes reduction and convergence orders of the corrected difference scheme \eqref{2Deqn5} for the linear diffusion equation in Example \ref{exam1}}
\vspace{-4mm}
\setlength\tabcolsep{1.9mm}{
\begin{tabular}{|c|ccccc|ccccc|}
\hline
\multicolumn{0}{|c|}{ }&\multicolumn{5}{c|}{\textbf{Case~I}}&\multicolumn{5}{c|}{\textbf{Case~II}}\\
         \cline{2-6}\cline{7-11}\diagbox{R}{P}
         & $m_1$  & $m_2$ & $n$    & ${E}_{\infty}(h_x,h_y,\tau)$& ${\rm Ord_G}$  & $m_1$  & $m_2$ & $n$  & ${E}_{\infty}(h_x,h_y,\tau)$& ${\rm Ord_G}$  \\
         \cline{1-6}\cline{7-11}
         & $5$  & $10$    & $700$    & $2.2598e-6$    & $*$      &$5$ & $500$    & $175$   & $3.7141e-6$  & $*$           \\
         & $10$ & $20$    & $2800$   & $5.7458e-7$    & $1.9756$ &$10$& $1000$   & $700$   & $8.4007e-7$  & $2.1444$ \\
  $r=1/7$& $20$ & $40$    & $11200$  & $1.4348e-7$    & $2.0016$ &$20$& $2000$   & $2800$  & $2.0296e-7$  & $2.0493$ \\
         & $40$ & $80$    & $44800$  & $3.5942e-8$    & $1.9971$ &$40$& $4000$   & $11200$ & $5.0363e-8$  & $2.0107$ \\
         & $80$ & $160$   & $179200$ & $8.9853e-9$    & $2.0000$ &$80$& $8000$   & $44800$ & $1.2563e-8$  & $2.0032$ \\
 \hline
         & $5$  & $10$    & $600$    & $2.0937e-8$    & $*$               &$5$ & $500$    & $150$   & $7.8615e-7$   & $*$           \\
         & $10$ & $20$    & $2400$   & $1.3370e-9$    & $\textbf{3.9690}$ &$10$& $1000$   & $600$   & $5.0407e-8$   & $\textbf{3.9631}$ \\
$r=\textbf{1/6}$& $20$ & $40$& $9600$& $8.3554e-11$    & $\textbf{4.0001}$ &$20$& $2000$   & $2400$  & $3.1499e-9$   & $\textbf{4.0002}$ \\
         & $40$ & $80$    & $38400$  & $5.1935e-12$    & $\textbf{4.0079}$ &$40$& $4000$   & $9600$  & $1.9706e-10$   & $\textbf{3.9986}$ \\
         & $80$ & $160$   & $153600$ & $1.6276e-13$    & $\textbf{4.9959}$ &$80$& $8000$   & $38400$ & $1.2082e-11$   & $\textbf{4.0277}$ \\
 \hline
         & $5$  & $10$   & $500$     & $3.1255e-6$    & $*$      &$5$ & $500$    & $125$   & $3.2597e-6$   &  $*$           \\
         & $10$ & $20$   & $2000$    & $8.0197e-7$    & $1.9624$ &$10$& $1000$   & $500$   & $1.0517e-6$   &  $1.6320$ \\
  $r=1/5$& $20$ & $40$   & $8000$    & $2.0072e-7$    & $1.9983$ &$20$& $2000$   & $2000$  & $2.7637e-7$   &  $1.9281$ \\
         & $40$ & $80$   & $32000$   & $5.0310e-8$    & $1.9963$ &$40$& $4000$   & $8000$  & $7.0022e-8$   &  $1.9807$ \\
         & $80$ & $160$  & $128000$  & $1.2579e-8$    & $1.9998$ &$80$& $8000$   & $32000$ & $1.7558e-8$   &  $1.9957$ \\
  \hline
         & $5$  & $10$    & $200$    & $3.2085e-5$    & $*$      &$5$ & $500$    & $50$    & $3.6891e-5$  & $*$           \\
         & $10$ & $20$    & $800$    & $8.0721e-6$    & $1.9909$ &$10$& $1000$   & $200$   & $1.0791e-5$  & $1.7734$ \\
  $r=1/2$& $20$ & $40$    & $3200$   & $2.0106e-6$    & $2.0054$ &$20$& $2000$   & $800$   & $2.7808e-6$  & $1.9563$ \\
         & $40$ & $80$    & $12800$  & $5.0330e-7$    & $1.9981$ &$40$& $4000$   & $3200$  & $7.0129e-7$  & $1.9874$ \\
         & $80$ & $160$   & $51200$  & $1.2580e-7$    & $2.0003$ &$80$& $8000$   & $12800$ & $1.7565e-7$  & $1.9973$ \\
  \hline
         & $5$  & $10$    & $199$    & $3.2334e-5$    & $*$       &$5$ & $500$    & $50$    & $9.5017e-5$  & $*$           \\
         & $10$ & $20$    & $796$    & $8.1327e-6$    & $1.9912$  &$10$& $1000$   & $199$   & $1.0872e-5$  & $3.1276$ \\
  $r=1/1.99$&$20$&$40$    & $3184$   & $4.4720e-4$    & $-5.7810$ &$20$& $2000$   & $796$   & $2.8052e-6$  & $1.9544$ \\
         & $40$ & $80$    & $12736$  & $5.5224e+76$   & $-266.06$ &$40$& $4000$   & $3184$  & $5.0471e+9$  &$-50.676$ \\
         & $80$ & $160$   & $50944$  & {\rm Inf}      & {\rm -Inf}&$80$& $8000$   & $12736$ & $1.5555e+90$ &$-267.38$ \\
  \hline
\end{tabular}\label{tab1}
}
\end{center}
\end{table}
\vspace{-12mm}
\begin{table}[H]
\caption{The errors in $L^{\infty}$-norm versus grid sizes reduction and convergence orders of the classical Euler
difference scheme \eqref{2Deqn6} for the linear diffusion equation in Example \ref{exam1}}
\centering
\vspace{-4mm}
\setlength\tabcolsep{1.9mm}{
\begin{tabular}{|c|ccccc|ccccc|}
\hline
\multicolumn{0}{|c|}{ }&\multicolumn{5}{c|}{\textbf{Case~I}}&\multicolumn{5}{c|}{\textbf{Case~II}}\\
         \cline{2-6}\cline{7-11}\diagbox{R}{P}
         & $m_1$  & $m_2$ & $n$    & ${E}_{\infty}(h_x,h_y,\tau)$& ${\rm Ord_G}$  & $m_1$  & $m_2$ & $n$  & ${E}_{\infty}(h_x,h_y,\tau)$& ${\rm Ord_G}$  \\
         \cline{1-6}\cline{7-11}
         & $5$  & $10$    & $700$    & $3.0144e-6$    & $*$      &$5$ & $500$    & $175$   & $2.7831e-4$  & $*$      \\
         & $10$ & $20$    & $2800$   & $7.7458e-7$    & $1.9604$ &$10$& $1000$   & $700$   & $7.1417e-5$  & $1.9623$ \\
  $r=1/7$& $20$ & $40$    & $11200$  & $1.9395e-7$    & $1.9977$ &$20$& $2000$   & $2800$  & $1.7854e-5$  & $2.0000$ \\
         & $40$ & $80$    & $44800$  & $4.8617e-8$    & $1.9961$ &$40$& $4000$   & $11200$ & $4.4692e-6$  & $1.9981$ \\
         & $80$ & $160$   & $179200$ & $1.2156e-8$    & $1.9998$ &$80$& $8000$   & $44800$ & $1.1173e-6$  & $2.0000$ \\
 \hline
         & $5$  & $10$    & $600$    & $9.2716e-7$    & $*$      &$5$ & $500$    & $150$   & $3.2822e-4$  & $*$      \\
         & $10$ & $20$    & $2400$   & $2.3637e-7$    & $1.9718$ &$10$& $1000$   & $600$   & $8.4248e-5$  & $1.9620$ \\
  $r=1/6$& $20$ & $40$    & $9600$   & $5.9067e-8$    & $2.0006$ &$20$& $2000$   & $2400$  & $2.1063e-5$  & $1.9999$ \\
         & $40$ & $80$    & $38400$  & $1.4799e-8$    & $1.9969$ &$40$& $4000$   & $9600$  & $5.2727e-6$  & $1.9981$ \\
         & $80$ & $160$   & $153600$ & $3.6999e-9$    & $1.9999$ &$80$& $8000$   & $38400$ & $1.3182e-6$  & $2.0000$ \\
 \hline
         & $5$  & $10$   & $500$     & $1.9943e-6$    & $*$      &$5$ & $500$    & $125$   & $3.9804e-4$  & $*$      \\
         & $10$ & $20$   & $2000$    & $5.1709e-7$    & $1.9474$ &$10$& $1000$   & $500$   & $1.0221e-4$  & $1.9614$ \\
  $r=1/5$& $20$ & $40$   & $8000$    & $1.2977e-7$    & $1.9945$ &$20$& $2000$   & $2000$  & $2.5556e-5$  & $1.9998$ \\
         & $40$ & $80$   & $32000$   & $3.2546e-8$    & $1.9953$ &$40$& $4000$   & $8000$  & $6.3975e-6$  & $1.9981$ \\
         & $80$ & $160$  & $128000$  & $8.1388e-9$    & $1.9996$ &$80$& $8000$   & $32000$ & $1.5994e-6$  & $2.0000$ \\
  \hline
         & $5$  & $10$    & $400$    & $6.3753e-6$    & $*$      &$5$ & $500$    & $100$   & $5.0264e-4$  & $*$           \\
         & $10$ & $20$    & $1600$   & $1.6472e-6$    & $1.9525$ &$10$& $1000$   & $400$   & $1.2914e-4$  & $1.9606$ \\
  $r=1/4$& $20$ & $40$    & $6400$   & $4.1301e-7$    & $1.9958$ &$20$& $2000$   & $1600$  & $3.2294e-5$  & $1.9996$ \\
         & $40$ & $80$    & $25600$  & $1.0356e-7$    & $1.9957$ &$40$& $4000$   & $6400$  & $8.0846e-6$  & $1.9980$ \\
         & $80$ & $160$   & $102400$ & $2.5897e-8$    & $1.9997$ &$80$& $8000$   & $25600$ & $2.0212e-6$  & $2.0000$ \\
  \hline
         & $5$  & $10$    & $399$    & $6.4302e-6$    & $*$       &$5$ & $500$    & $100$   & $4.3540e-4$  & $*$           \\
         & $10$ & $20$    & $1596$   & $1.6614e-6$    & $1.9525$  &$10$& $1000$   & $399$   & $1.2948e-4$  & $1.7496$ \\
  $r=1/3.99$&$20$&$40$    & $6384$   & $4.1656e-7$    & $1.9958$  &$20$& $2000$   & $1596$  & $3.2379e-5$  & $1.9996$ \\
         & $40$ & $80$    & $25536$  & $2.1130e+18$   & $-82.069$ &$40$& $4000$   & $6384$  & $1.1756e-5$  & $1.4617$ \\
         & $80$ & $160$   & $102144$ & $1.0174e+18$   & $-550.39$ &$80$& $8000$   & $25536$ & $3.4068e+34$ &$-131.09$ \\
  \hline
\end{tabular}\label{table1}
}
\end{table}
\end{example}

\begin{example}\label{exam2}
   Then we consider the problem \eqref{eqn1} by the corrected difference scheme \eqref{2Deqn6} and the classical difference scheme \eqref{2Deqn7}. The parameters are taken as $L_1=L_2=0$, $R_1=R_2=1$ and $T=1$. The initial value condition, boundary value conditions and the source term are determined by the exact solution $u(\mathbf{x},t) = \exp(x+y-t)$. We take diffusion and convection coefficients $a,\ b,\ c,\ d$ as
   \begin{itemize}
    \item [] {\rm \textbf{Case~I:~}} $a=4,\ b =1,\ c=-10,\ d=20${\rm ;}
    \quad {\rm \textbf{Case~II:~}} $a=1,\ b = 0.01,\ c=-1, \ d=2$.
   \end{itemize}

   Numerical results are listed in Tables \ref{tab2} and \ref{table2}. When the step-ratio is $1/6$, the convergence rate is approximately fourth-order for the above two sets of parameters, which confirm the theoretical findings. For other step-ratios, the convergence rate is only two globally in both cases. Moreover, the corrected difference scheme \eqref{2Deqn6} is more accurate than the classical difference scheme \eqref{2Deqn7} whatever the step sizes are, which displays the superiority of the corrected difference scheme \eqref{2Deqn6}.
\vspace{-5mm}
\begin{table}[H]
\caption{The errors in $L^{\infty}$-norm versus grid sizes reduction and convergence orders of the corrected
difference scheme \eqref{2Deqn7} for the linear convection-diffusion equation in Example \ref{exam2}}
\centering
\vspace{-4mm}
\setlength\tabcolsep{2.1mm}{
\begin{tabular}{|c|ccccc|ccccc|}
\hline
\multicolumn{0}{|c|}{ }&\multicolumn{5}{c|}{\textbf{Case~I}}&\multicolumn{5}{c|}{\textbf{Case~II}}\\
         \cline{2-6}\cline{7-11}\diagbox{R}{P}
         & $m_1$  & $m_2$ & $n$    & ${E}_{\infty}(h_x,h_y,\tau)$& ${\rm Ord_G}$  & $m_1$  & $m_2$ & $n$  & ${E}_{\infty}(h_x,h_y,\tau)$& ${\rm Ord_G}$  \\
         \cline{1-6}\cline{7-11}
         & $5$  & $10$    & $700$    & $9.8608e-5$    & $*$      &$5$ & $50$    & $175$   & $7.2462e-5$  & $*$           \\
         & $10$ & $20$    & $2800$   & $1.8204e-5$    & $2.4374$ &$10$& $100$   & $700$   & $1.7282e-5$  & $2.0680$ \\
  $r=1/7$& $20$ & $40$    & $11200$  & $4.1716e-6$    & $2.1256$ &$20$& $200$   & $2800$  & $4.2680e-6$  & $2.0176$ \\
         & $40$ & $80$    & $44800$  & $1.0174e-6$    & $2.0358$ &$40$& $400$   & $11200$ & $1.0653e-6$  & $2.0023$ \\
         & $80$ & $160$   & $179200$ & $2.5287e-7$    & $2.0084$ &$80$& $800$   & $44800$ & $2.6613e-7$  & $2.0011$ \\
 \hline
         & $5$  & $10$    & $600$    & $4.0615e-5$    & $*$      &$5$     & $50$         & $150$   & $5.1442e-6$   & $*$           \\
         & $10$ & $20$    & $2400$   & $2.5579e-6$    & $\textbf{3.9890}$ &$10$& $100$   & $600$   & $3.2166e-7$   & $\textbf{3.9994}$ \\
$r=\textbf{1/6}$& $20$ & $40$& $9600$& $1.6117e-7$    & $\textbf{3.9883}$ &$20$& $200$   & $2400$  & $2.0103e-8$   & $\textbf{4.0000}$ \\
         & $40$ & $80$    & $38400$  & $1.0079e-8$    & $\textbf{3.9993}$ &$40$& $400$   & $9600$  & $1.2583e-9$   & $\textbf{3.9978}$ \\
         & $80$ & $160$   & $153600$ & $6.3068e-10$   & $\textbf{3.9982}$ &$80$& $800$   & $38400$ & $7.8984e-11$  & $\textbf{3.9938}$ \\
 \hline
         & $5$  & $10$   & $500$     & $4.0159e-5$    & $*$      &$5$ & $50$    & $125$   & $3.2297e+2$   &  $*$           \\
         & $10$ & $20$   & $2000$    & $1.9326e-5$    & $1.0552$ &$10$& $100$   & $500$   & $2.3429e-5$   &  $2371.7$ \\
  $r=1/5$& $20$ & $40$   & $8000$    & $5.4520e-6$    & $1.8257$ &$20$& $200$   & $2000$  & $5.9274e-6$   &  $1.9828$ \\
         & $40$ & $80$   & $32000$   & $1.4000e-6$    & $1.9613$ &$40$& $400$   & $8000$  & $1.4885e-6$   &  $1.9936$ \\
         & $80$ & $160$  & $128000$  & $3.5249e-7$    & $1.9898$ &$80$& $800$   & $32000$ & $3.7239e-7$   &  $1.9989$ \\
  \hline
\end{tabular}\label{tab2}
}
\end{table}
\vspace{-10mm}
\begin{table}[H]
\caption{The errors in $L^{\infty}$-norm versus grid sizes reduction and convergence orders of the classical Euler
difference scheme \eqref{2Deqn8}  for the linear convection-diffusion equation in Example \ref{exam2}}
\vspace{-3mm}
\centering
\setlength\tabcolsep{2.2mm}{
\begin{tabular}{|c|ccccc|ccccc|}
\hline
\multicolumn{0}{|c|}{ }&\multicolumn{5}{c|}{\textbf{Case~I}}&\multicolumn{5}{c|}{\textbf{Case~II}}\\
         \cline{2-6}\cline{7-11}\diagbox{R}{P}
         & $m_1$  & $m_2$ & $n$    & ${E}_{\infty}(h_x,h_y,\tau)$& ${\rm Ord_G}$  & $m_1$  & $m_2$ & $n$  & ${E}_{\infty}(h_x,h_y,\tau)$& ${\rm Ord_G}$  \\
         \cline{1-6}\cline{7-11}
         & $5$  & $10$    & $700$    & $4.6670e-4$    & $*$      &$5$ & $50$    & $175$   & $8.9757e-4$  & $*$      \\
         & $10$ & $20$    & $2800$   & $1.1650e-4$    & $2.0021$ &$10$& $100$   & $700$   & $2.2506e-4$  & $1.9957$ \\
  $r=1/7$& $20$ & $40$    & $11200$  & $2.9304e-5$    & $1.9912$ &$20$& $200$   & $2800$  & $5.6308e-5$  & $1.9989$ \\
         & $40$ & $80$    & $44800$  & $7.3252e-6$    & $2.0001$ &$40$& $400$   & $11200$ & $1.4101e-5$  & $1.9975$ \\
         & $80$ & $160$   & $179200$ & $1.8322e-6$    & $1.9993$ &$80$& $800$   & $44800$ & $3.5254e-6$  & $1.9999$ \\
 \hline
         & $5$  & $10$    & $600$    & $4.6947e-4$    & $*$      &$5$ & $50$    & $150$   & $9.6777e-4$  & $*$      \\
         & $10$ & $20$    & $2400$   & $1.1720e-4$    & $2.0021$ &$10$& $100$   & $600$   & $2.4273e-4$  & $1.9953$ \\
  $r=1/6$& $20$ & $40$    & $9600$   & $2.9479e-5$    & $1.9912$ &$20$& $200$   & $2400$  & $6.0733e-5$  & $1.9988$ \\
         & $40$ & $80$    & $38400$  & $7.3691e-6$    & $2.0001$ &$40$& $400$   & $9600$  & $1.5209e-5$  & $1.9975$ \\
         & $80$ & $160$   & $153600$ & $1.8431e-6$    & $1.9993$ &$80$& $800$   & $38400$ & $3.8026e-6$  & $1.9999$ \\
 \hline
         & $5$  & $10$   & $500$     & $4.7336e-4$    & $*$      &$5$ & $50$    & $125$   & $1.0660e-3$  & $*$      \\
         & $10$ & $20$   & $2000$    & $1.1817e-4$    & $2.0020$ &$10$& $100$   & $500$   & $2.6747e-4$  & $1.9947$ \\
  $r=1/5$& $20$ & $40$   & $8000$    & $2.9724e-5$    & $1.9912$ &$20$& $200$   & $2000$  & $6.6929e-5$  & $1.9987$ \\
         & $40$ & $80$   & $32000$   & $7.4305e-6$    & $2.0001$ &$40$& $400$   & $8000$  & $1.6761e-5$  & $1.9975$ \\
         & $80$ & $160$  & $128000$  & $1.8585e-6$    & $1.9993$ &$80$& $800$   & $32000$ & $4.1906e-6$  & $1.9999$ \\
  \hline
\end{tabular}\label{table2}
}
\end{table}
\end{example}
\vspace{-8mm}
\begin{example}\label{exam3}
We test the convergence rate and solution behavior to the problem \eqref{eqn1} with the variable coefficients $c(\mathbf{x}) = \sin(x+y)$ and $d(\mathbf{x})=\cos(x+y)$ by the difference scheme \eqref{2Deqn1}. The initial condition is taken as $\varphi(\mathbf{x}) = \exp(-x^2-y^2)$ and boundary value conditions are homogeneous. The exact solution is unknown. The parameters are taken as $L_1=L_2=-5$, $R_1=R_2=5$, $T=1$ and the diffusion coefficients are respectively taken as
\vspace{-3mm}
\begin{itemize}
  \item []{\rm \textbf{Case~I:~}} $a=4,  b =1;$\quad {\rm \textbf{Case~II:~}} $a=1, b = 0.01;$
  \item []{\rm \textbf{Case~III:~}} $a=100,  b =100;$\quad {\rm \textbf{Case~IV:~}} $a = 0.01,  b = 0.01.$
\end{itemize}

The numerical results are shown in Tables \ref{tab3a}--\ref{tab3b} and Figure \ref{fig1}. Since the exact solution is unknown, we use the second method in \eqref{eqn5.1} to test the convergence rate. As we see from Tables \ref{tab3a} and \ref{tab3b}, all the results are in agreement with our theoretical findings.
Moreover, numerical surfaces are demonstrated in Figure \ref{fig1} with the optimal step-ratio $r=1/6$, which further confirm the resolution performance of the corrected difference scheme \eqref{2Deqn1}. It is worth mentioning that the smaller diffusion coefficients $a$ and $b$ are, the more dense grids required $($see {\rm \textbf{Case IV}}  in Table \ref{tab3b}$)$.
   \vspace{-5mm}
\begin{table}[H]
\caption{The errors in $L^{\infty}$-norm versus grid sizes reduction and convergence orders of the corrected
difference scheme \eqref{2Deqn1}  for the linear convection-diffusion equation with the variable coefficients in Example \ref{exam3}}
\vspace{-3mm}
\centering
\setlength\tabcolsep{2.2mm}{
\begin{tabular}{|c|ccccc|ccccc|}
\hline
\multicolumn{0}{|c|}{ }&\multicolumn{5}{c|}{\textbf{Case~I}}&\multicolumn{5}{c|}{\textbf{Case~II}}\\
         \cline{2-6}\cline{7-11}\diagbox{R}{P}
         & $m_1$  & $m_2$ & $n$    & ${E}_{\infty}(h_x,h_y,\tau)$& ${\rm Ord_G}$  & $m_1$  & $m_2$ & $n$  & ${E}_{\infty}(h_x,h_y,\tau)$& ${\rm Ord_G}$  \\
         \cline{1-6}\cline{7-11}
         & $10$  & $20$    & $28$    & $3.5063e-04$    & $*$      &$10$ & $100$    & $7$    & $4.7662e-03$  & $*$           \\
         & $20$  & $40$    & $112$   & $9.5061e-05$    & $1.8830$ &$20$ & $200$    & $28$   & $1.0877e-03$  & $2.1315$ \\
  $r=1/7$& $40$  & $80$    & $448$   & $2.4605e-05$    & $1.9499$ &$40$ & $400$    & $112$  & $3.3119e-04$  & $1.7156$ \\
         & $80$  & $160$   & $1792$  & $6.1907e-06$    & $1.9908$ &$80$ & $800$    & $448$  & $8.6900e-05$  & $1.9303$ \\
         & $160$ & $320$   & $7168$  & $1.5513e-06$    & $1.9966$ &$160$& $1600$   & $1792$ & $2.2006e-05$  & $1.9815$ \\
 \hline
         & $10$  & $20$    & $24$    & $1.3708e-04$    & $*$      &$10$   & $100$         & $6$    & $1.0013e-02$   & $*$           \\
         & $20$  & $40$    & $96$    & $8.6521e-06$    & $\textbf{3.9858}$ &$20$& $200$   & $24$   & $5.7511e-04$   & $\textbf{4.1218}$ \\
$r=\textbf{1/6}$ & $40$ & $80$& $384$& $5.4266e-07$    & $\textbf{3.9949}$ &$40$& $400$   & $96$   & $3.5329e-05$   & $\textbf{4.0249}$ \\
         & $80$  & $160$    & $1536$ & $3.3796e-08$    & $\textbf{4.0051}$ &$80$& $800$   & $384$  & $2.2085e-06$   & $\textbf{3.9997}$ \\
         & $160$ & $320$   & $6144$  & $2.1103e-09$    & $\textbf{4.0013}$ &$160$& $1600$& $1536$  & $1.3778e-07$   & $\textbf{4.0027}$ \\
 \hline
         & $10$  & $20$   & $20$     & $7.2761e-04$    & $*$      &$10$ & $100$    & $5$    & $2.1408e-02$   &  $*$           \\
         & $20$  & $40$   & $80$     & $1.4741e-04$    & $2.3033$ &$20$ & $200$    & $20$   & $2.7078e-03$   &  $2.9830$ \\
  $r=1/5$& $40$  & $80$   & $320$    & $3.5263e-05$    & $2.0636$ &$40$ & $400$    & $80$   & $5.4266e-04$   &  $2.3190$ \\
         & $80$  & $160$  & $1280$   & $8.7176e-06$    & $2.0161$ &$80$ & $800$    & $320$  & $1.2652e-04$   &  $2.1007$ \\
         & $160$ & $320$  & $5120$   & $2.1751e-06$    & $2.0028$ &$160$& $1600$   & $1280$ & $3.1116e-05$   &  $2.0236$ \\
\hline
\end{tabular}\label{tab3a}
}
\end{table}
\vspace{-12mm}
\begin{table}[H]
\caption{The errors in $L^{\infty}$-norm versus grid sizes reduction and convergence orders of the corrected
difference scheme \eqref{2Deqn1}  for the linear convection-diffusion equation with the variable coefficients in Example \ref{exam3}}
\vspace{-3mm}
\centering
\setlength\tabcolsep{2.1mm}{
\begin{tabular}{|c|ccccc|ccccc|}
\hline
\multicolumn{0}{|c|}{ }&\multicolumn{5}{c|}{\textbf{Case~III}}&\multicolumn{5}{c|}{\textbf{Case~IV}}\\
         \cline{2-6}\cline{7-11}\diagbox{R}{P}
         & $m_1$  & $m_2$ & $n$    & ${E}_{\infty}(h_x,h_y,\tau)$& ${\rm Ord_G}$  & $m_1$  & $m_2$ & $n$  & ${E}_{\infty}(h_x,h_y,\tau)$& ${\rm Ord_G}$  \\
         \cline{1-6}\cline{7-11}
         & $10$  & $10$    & $700$    & $5.8207e-12$    & $*$      &$80$ & $80$      & $4$    & $7.8466e-02$  & $*$           \\
         & $20$  & $20$    & $2800$   & $1.3892e-12$    & $2.0670$ &$160$& $160$     & $18$   & $2.1898e-03$  & $5.1632$ \\
  $r=1/7$& $40$  & $40$    & $11200$   & $3.4475e-13$   & $2.0106$ &$320$& $320$     & $72$   & $1.5272e-04$  & $3.8419$ \\
         & $80$  & $80$    & $44800$  & $8.6028e-14$    & $2.0027$ &$640$ & $640$    & $287$  & $4.0011e-05$  & $1.9324$ \\
         & $160$ & $160$   & $179200$  & $2.1497e-14$   & $2.0007$ &$1280$& $1280$   & $1147$ & $1.1807e-05$  & $1.7607$ \\
 \hline
         & $10$  & $10$    & $600$    & $2.0539e-13$    & $*$               &$80$  & $80$    & $4$    & $7.8466e-02$   & $*$  \\
         & $20$  & $20$    & $2400$    & $7.0272e-15$   & $\textbf{4.8693}$ &$160$ & $160$   & $15$   & $3.6061e-03$   & $\textbf{4.4436}$ \\
$r=\textbf{1/6}$ & $40$ & $40$& $9600$& $4.3983e-16$    & $\textbf{3.9979}$ &$320$ & $320$   & $61$   & $2.2628e-04$   & $\textbf{3.9942}$ \\
         & $80$  & $80$    & $38400$ & $2.7499e-17$     & $\textbf{3.9995}$ &$640$ & $640$   & $246$  & $1.3218e-05$   & $\textbf{4.0975}$ \\
         & $160$ & $160$   & $153600$  & $1.7195e-18$   & $\textbf{3.9993}$ &$1280$& $1280$  & $983$  & $8.5304e-07$   & $\textbf{3.9538}$ \\
 \hline
         & $10$  & $10$   & $500$     & $7.4220e-12$    & $*$      &$80$   & $80$     & $3$    & $1.9008e-01$   &  $*$           \\
         & $20$  & $20$   & $2000$     & $1.9139e-12$   & $1.9553$ &$160$  & $160$    & $13$   & $5.2300e-03$   &  $5.1837$ \\
  $r=1/5$& $40$  & $40$   & $8000$    & $4.8071e-13$    & $1.9933$ &$320$  & $320$    & $51$   & $5.5828e-04$   &  $3.2277$ \\
         & $80$  & $80$   & $32000$   & $1.2032e-13$    & $1.9983$ &$640$  & $640$    & $205$  & $8.6098e-05$   &  $2.6969$ \\
         & $160$ & $160$  & $128000$   & $3.0088e-14$   & $1.9996$ &$1280$ & $1280$   & $819$  & $1.8490e-05$   &  $2.2192$ \\
  \hline
\end{tabular}\label{tab3b}
}
\end{table}

\begin{figure}[htbp]
\centering
 \subfigure[Numerical surface]{\centering
\includegraphics[width=0.31\textwidth]{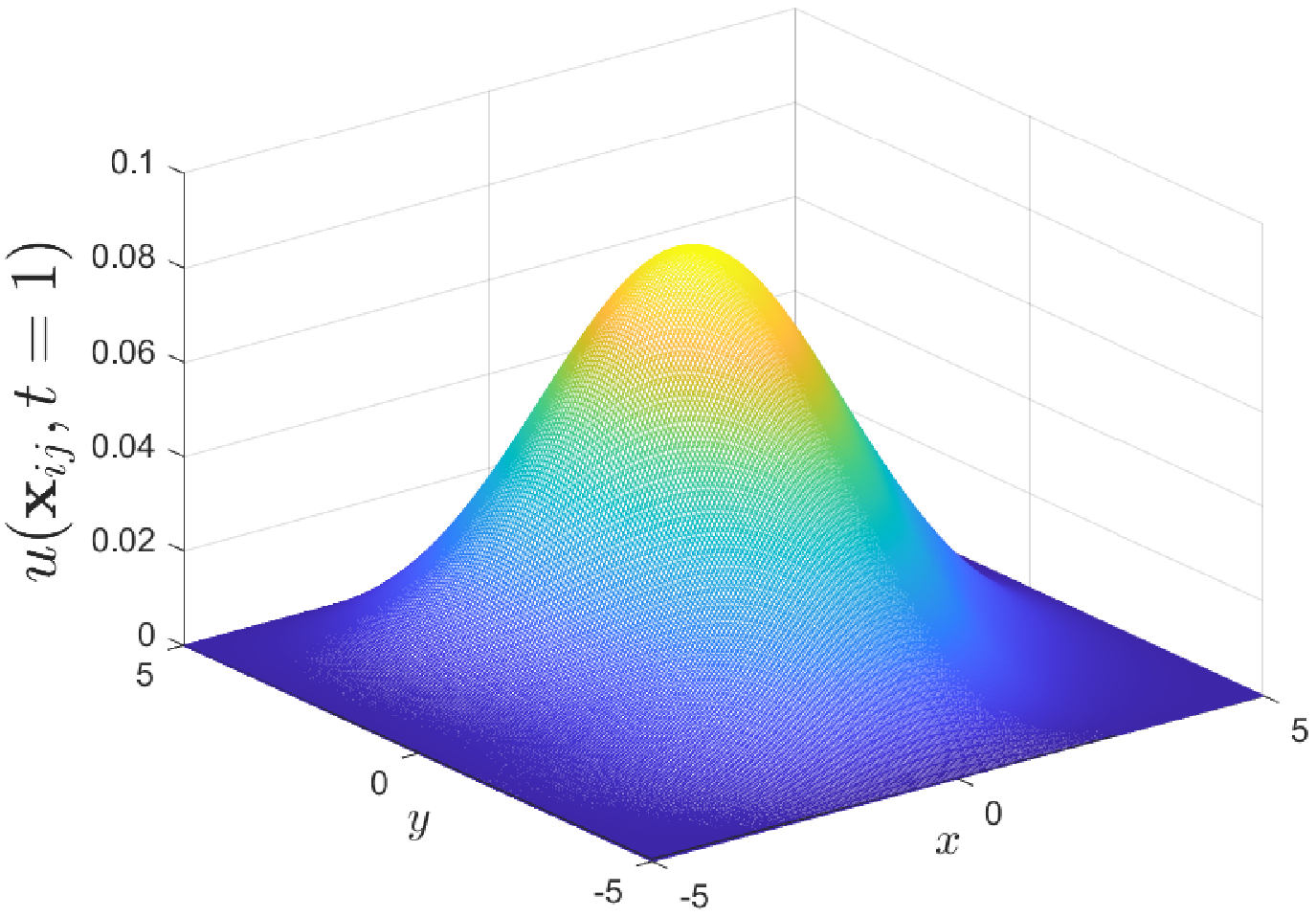}
}\subfigure[Numerical surface]{\centering
\includegraphics[width=0.31\textwidth]{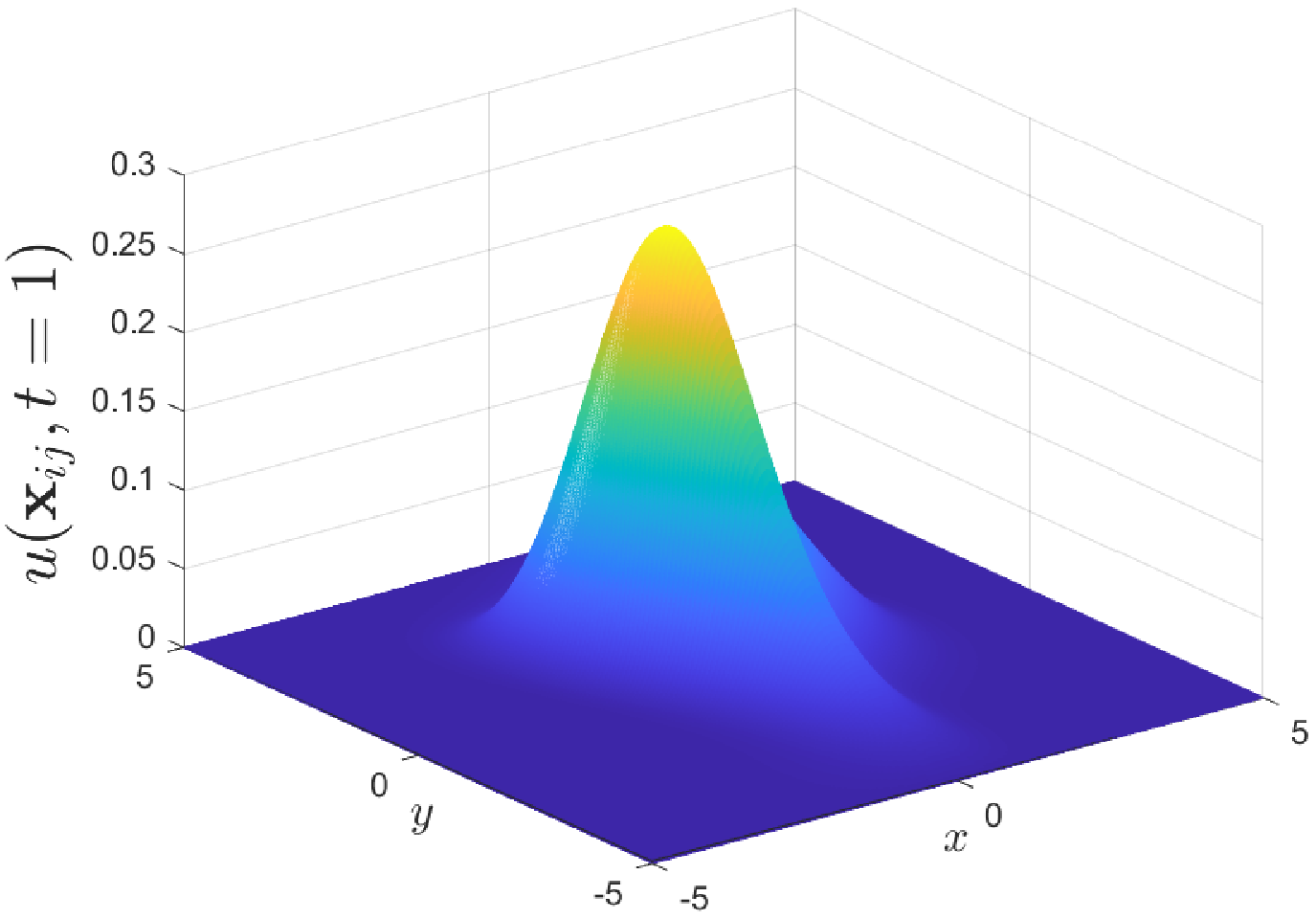}
}\subfigure[Numerical surface]{\centering
\includegraphics[width=0.31\textwidth]{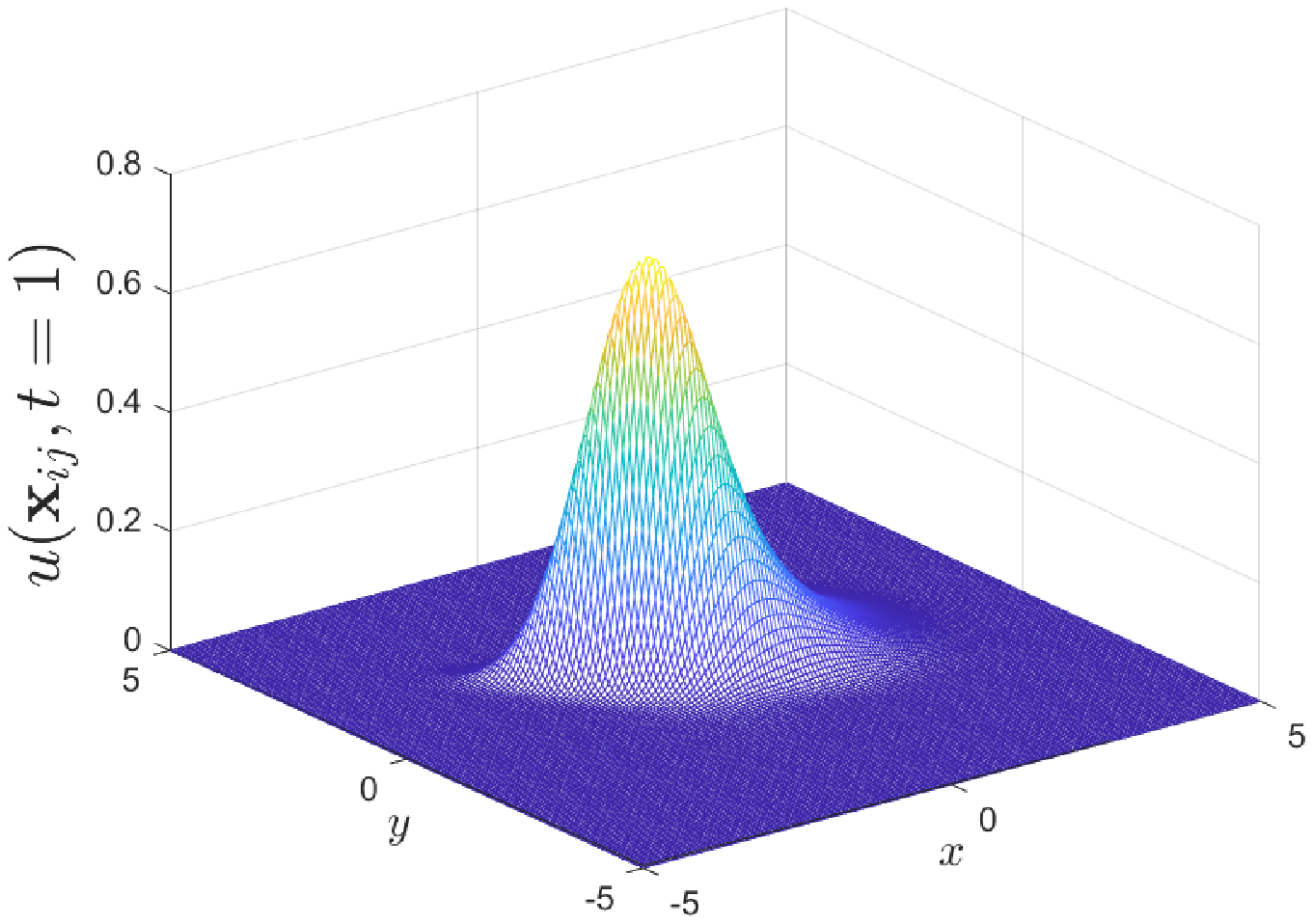}
}\\
\subfigure[Contour]{\centering
\includegraphics[width=0.31\textwidth]{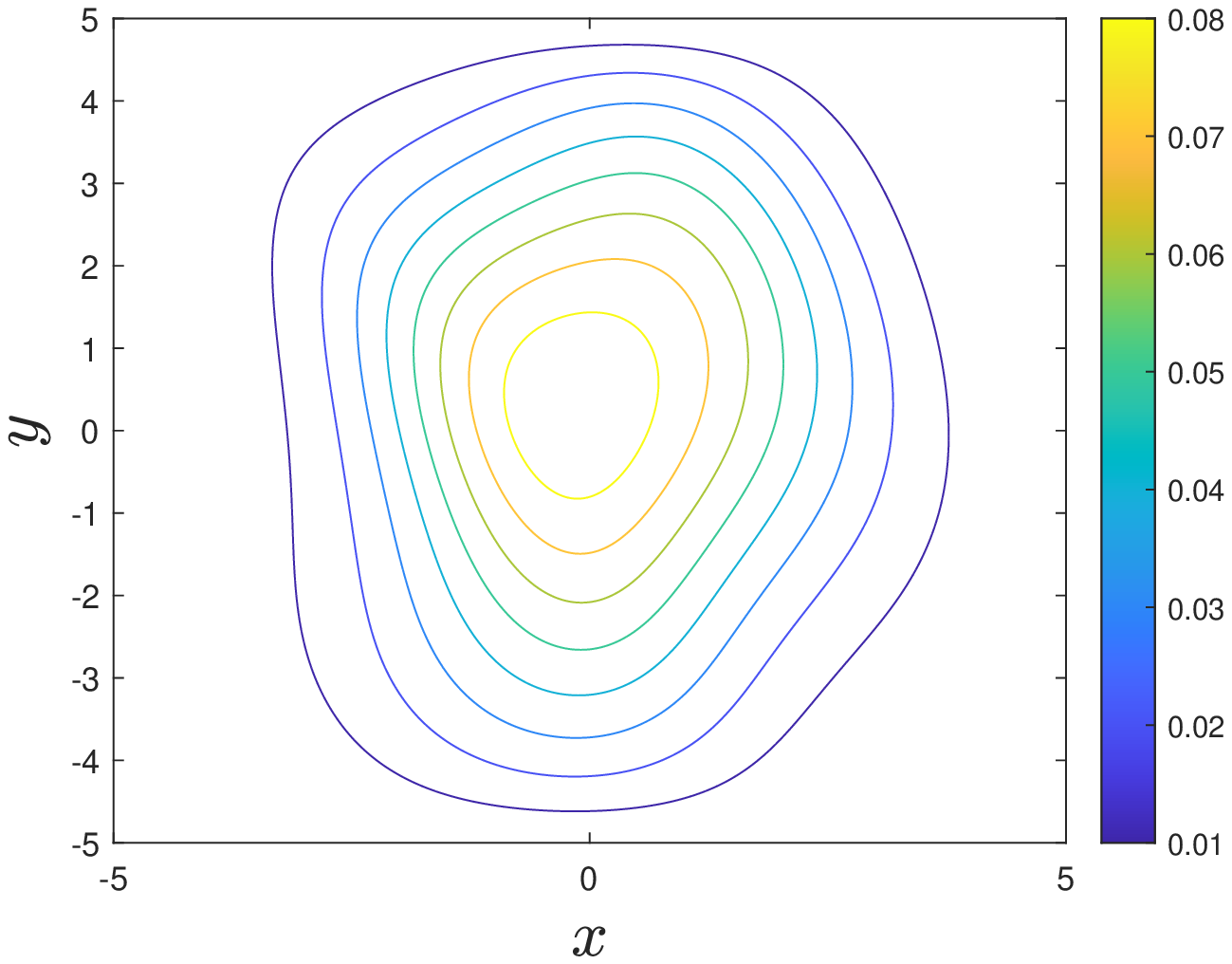}
}\subfigure[Contour]{\centering
\includegraphics[width=0.31\textwidth]{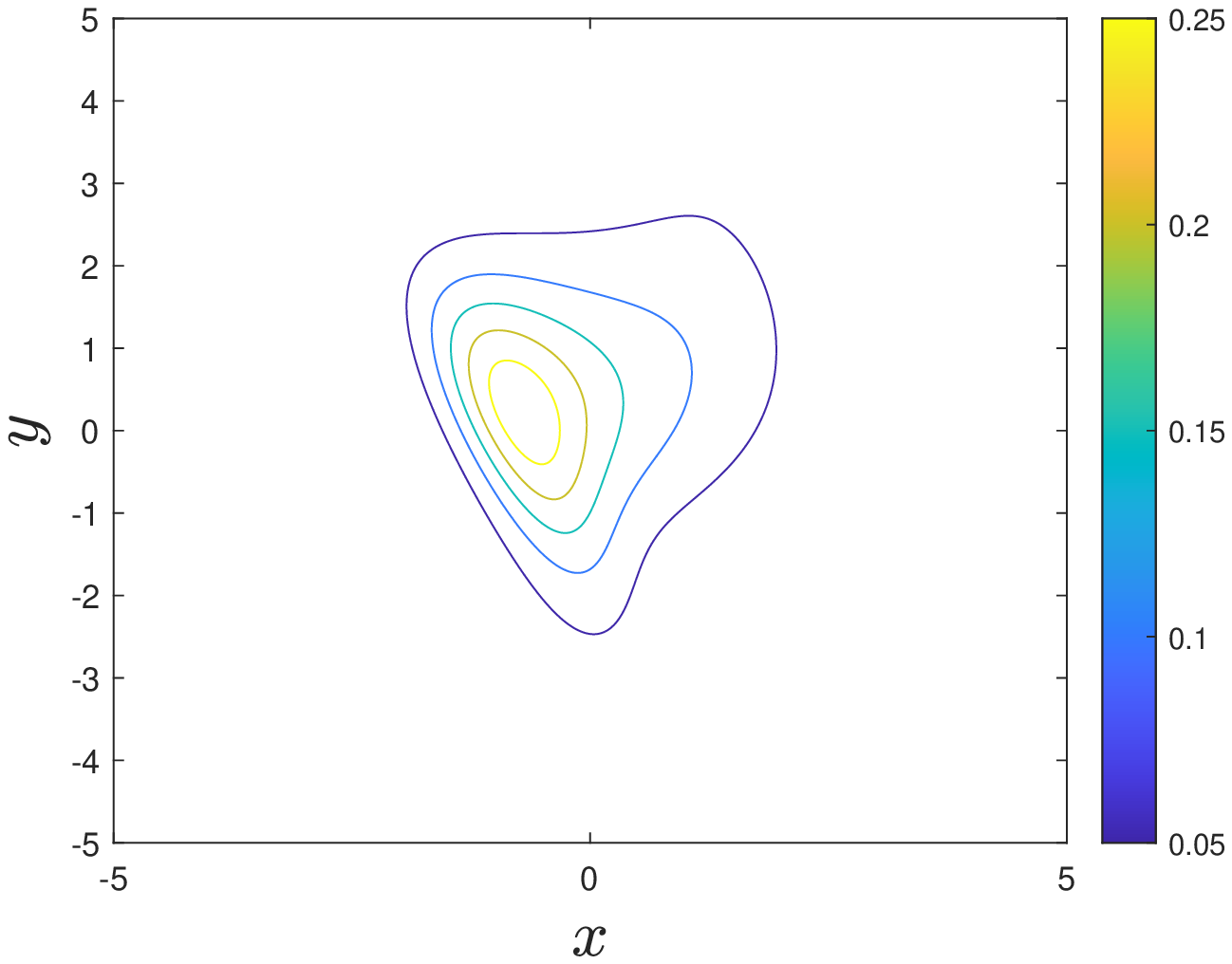}
}\subfigure[Contour]{\centering
\includegraphics[width=0.31\textwidth]{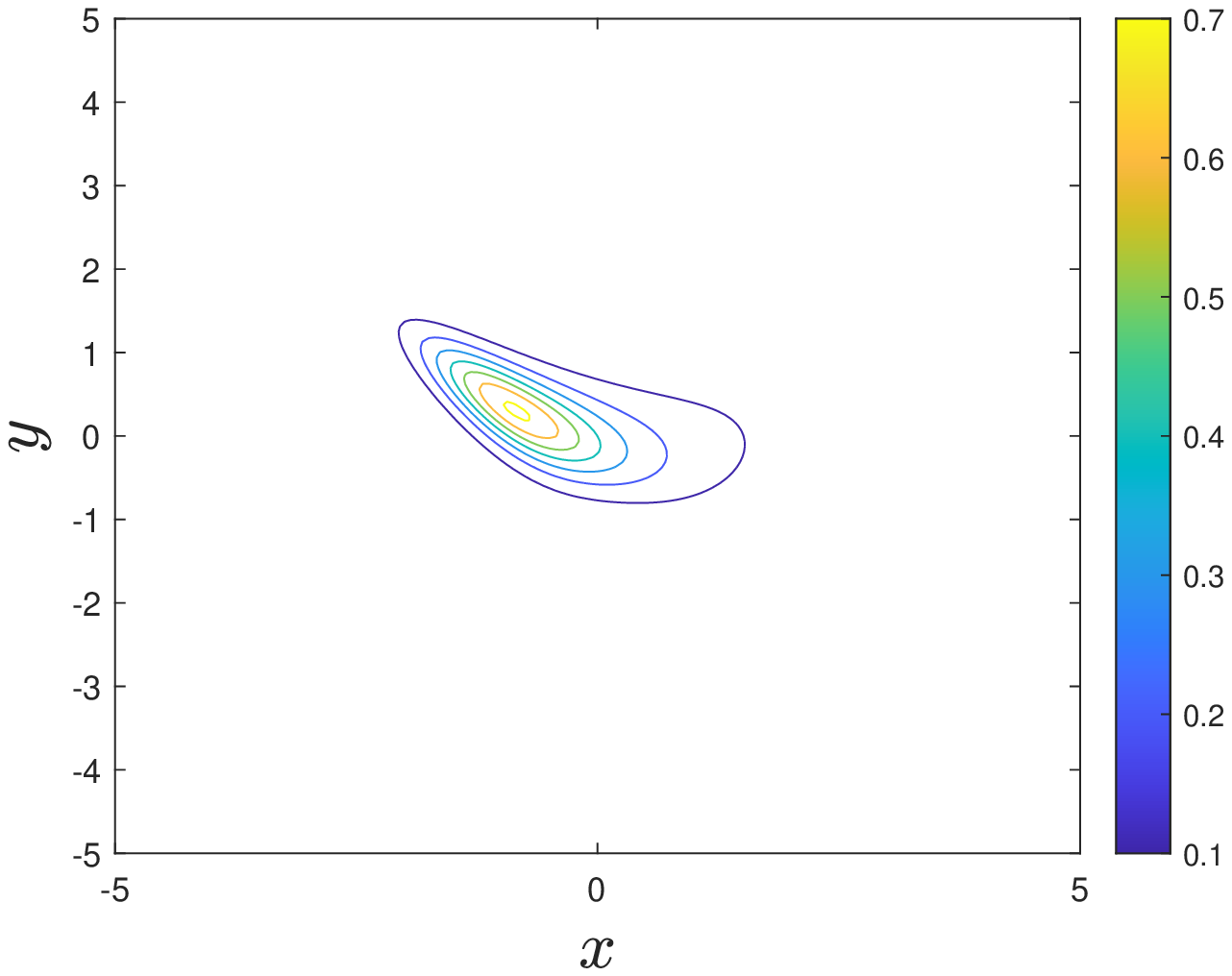}
}
\caption{Numerical surfaces and corresponding contours at the same terminal time $t=1$ in Example \ref{exam3}. The coefficients and parameters are respectively taken as (a) and (d): $a=4$, $b=1$, $m_1=160$, $m_2=320$, $n = 6144$; (b) and (e): $a=1$, $b=0.01$, $m_1=160$, $m_2=1600$, $n=1536$; (c) and (f): $a=0.04$, $b=0.04$, $m_1=320$, $m_2=320$, $n=246$.} \label{fig1}
\end{figure}
\end{example}
\subsection{Nonlinear problems}
We will test Example \ref{Exam4} and Example \ref{Exam5} under Dirichlet boundary conditions by the difference scheme \eqref{eqn4.3}, \eqref{2Deqn1b}--\eqref{2Deqn1c}  and under Neumann boundary conditions by the difference scheme \eqref{eqn4.3}, \eqref{eqn6.2aa}--\eqref{eqn6.5aa} respectively.
\begin{example}\label{Exam4}
  We first consider semi-linear diffusion reaction equations as
  \begin{align*}
    u_t=\Delta u + r(u),\quad  \mathbf{x} \in (0,1)^2, \quad t\in (0,T],
  \end{align*}
  where the initial and boundary value conditions are determined by the exact solution.

  \begin{itemize}
  \item []{\rm \textbf{Case~I:~}} Fisher's equation {\rm \cite{QS1998}}: $r(u)=u(1-u)$. The exact solution is
  \begin{align*}
     u(\mathbf{x},t) = \bigg[1+\exp\Big(-\frac{5}{6}t +\frac{\sqrt{3}}{6}x +\frac{\sqrt{3}}{6}y\Big)\bigg]^{-2};
  \end{align*}
  \item []{\rm \textbf{Case~II:~}} Chafee-Infante equation {\rm \cite{CI1974}}: $r(u)=u(1-u^2)$. The exact solution is
  \begin{align*}
      u(\mathbf{x},t) = \frac12\tanh\Big(\frac14 x+\frac14 y + \frac34 t\Big)+\frac12.
  \end{align*}
\end{itemize}
\vspace{-15mm}
\begin{figure}[htbp]
\centering
 \subfigure[Numerical surface, \textbf{view}(45,36)]{\centering
\includegraphics[width=0.24\textwidth]{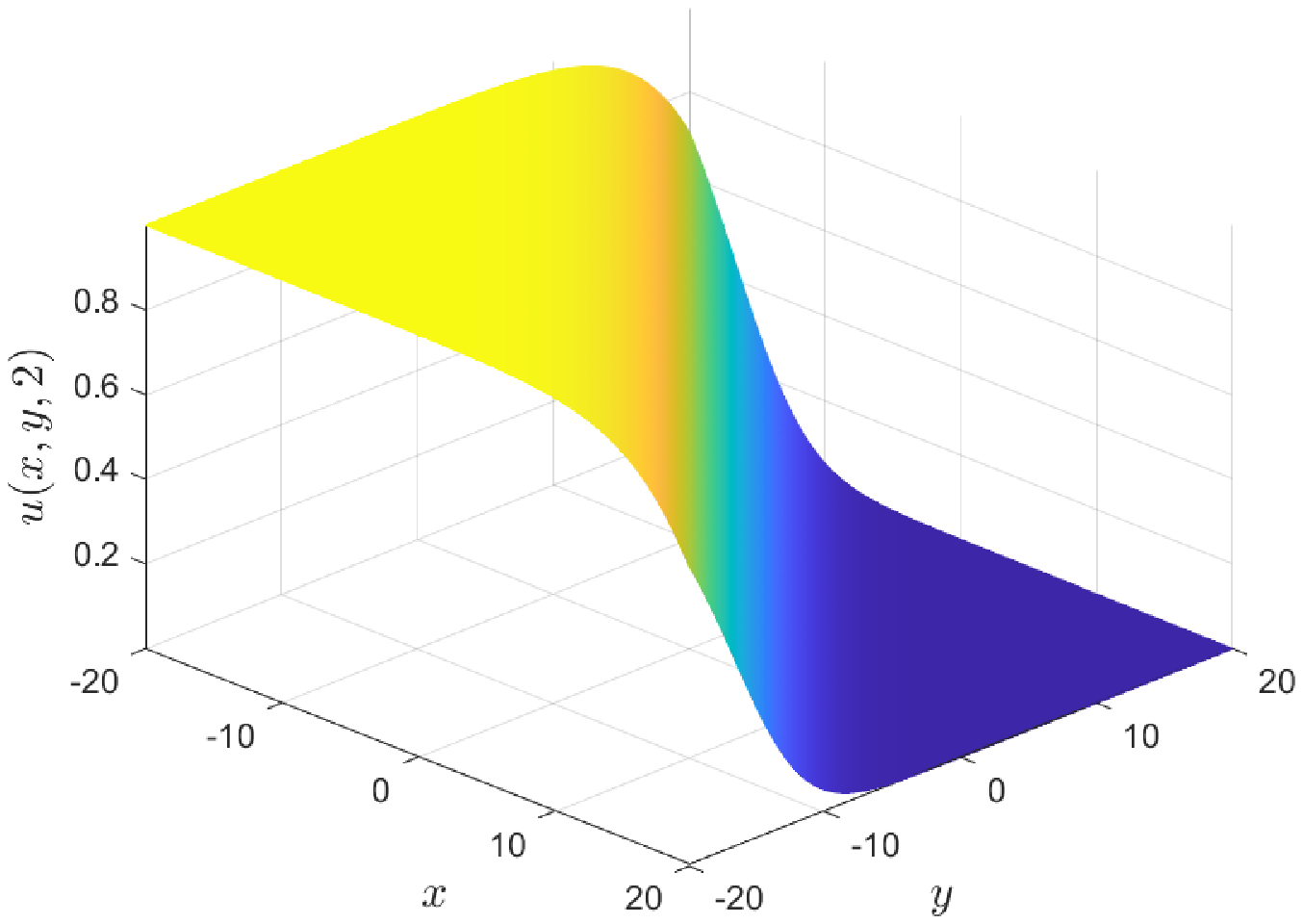}
}\subfigure[Error surface, \textbf{view}(45,36)]{\centering
\includegraphics[width=0.24\textwidth]{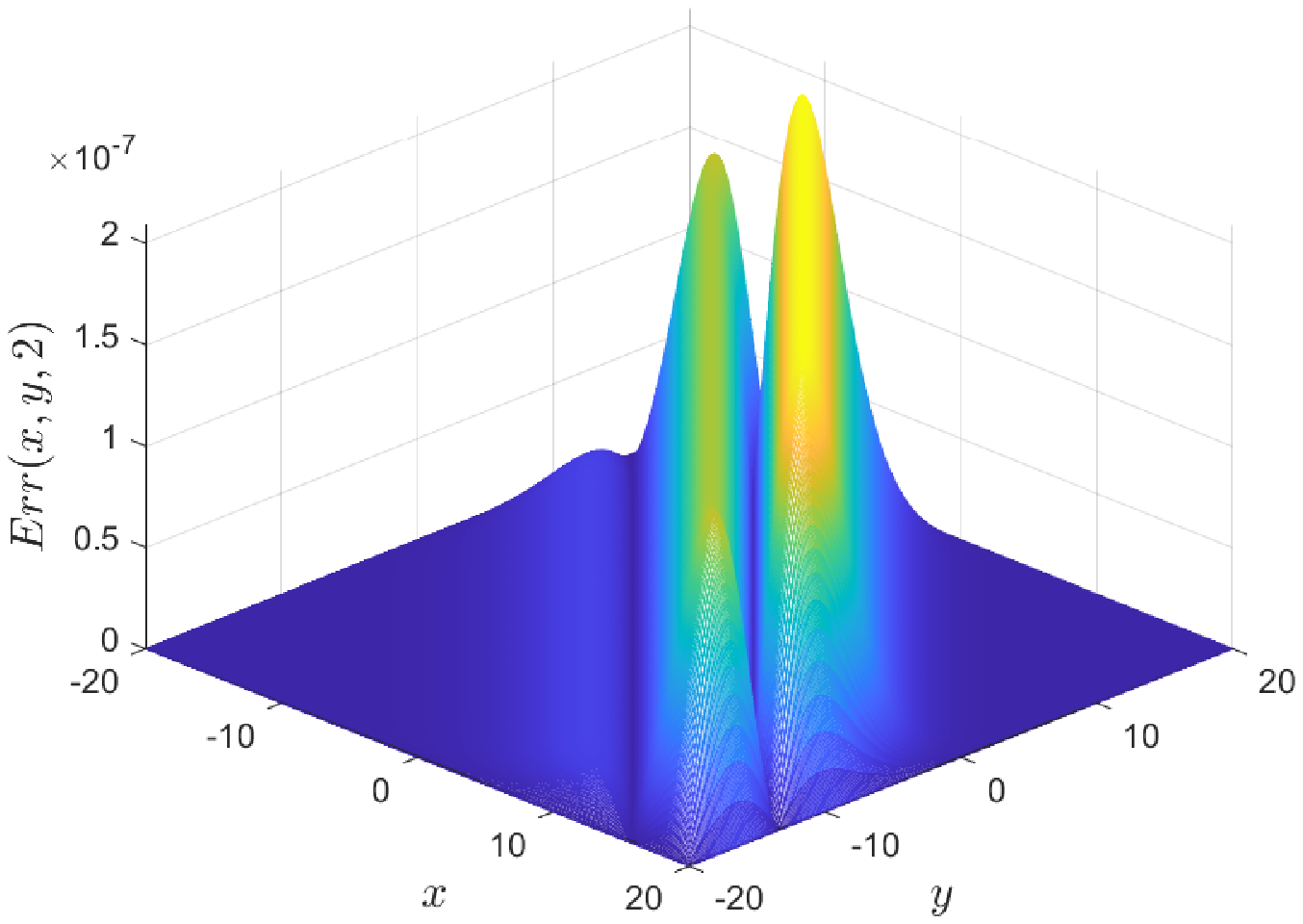}
}
\subfigure[Numerical surface, \textbf{view}(38,30)]{\centering
\includegraphics[width=0.24\textwidth]{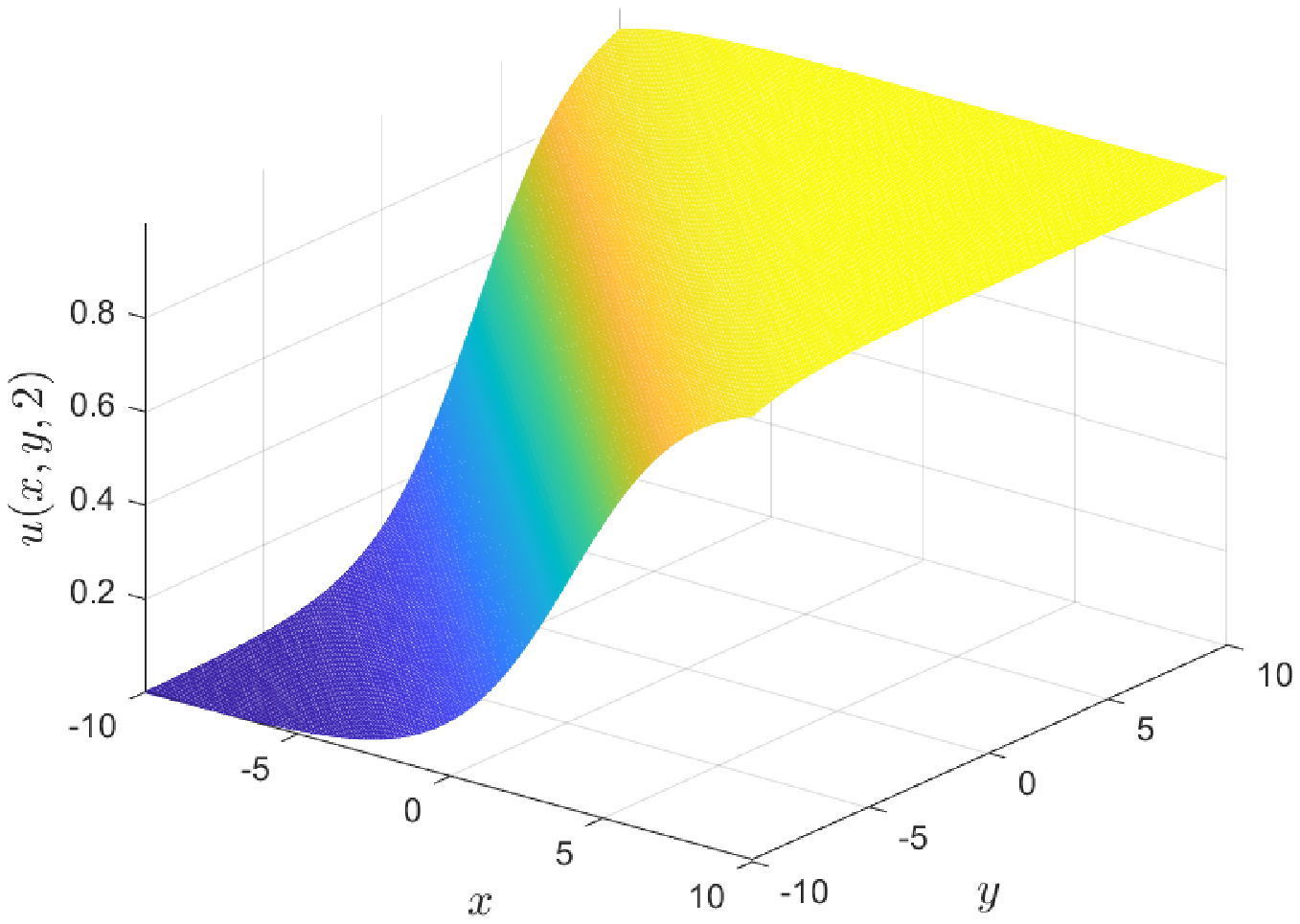}
}\subfigure[Error surface, \textbf{view}(38,30)]{\centering
\includegraphics[width=0.24\textwidth]{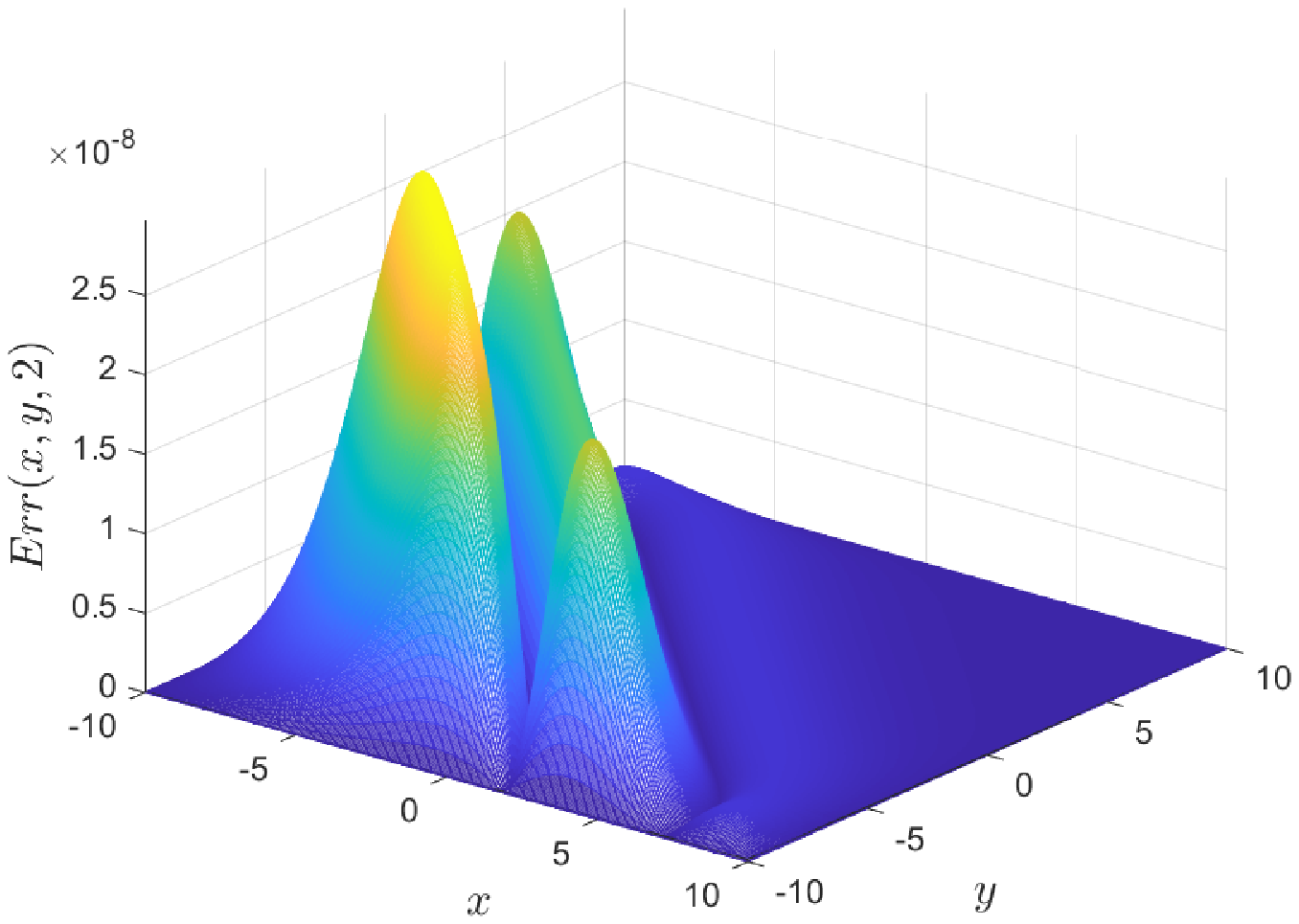}
}
\vspace{-5mm}
\caption{Numerical surfaces and error surfaces at the terminal time $t=2$.
(a)--(b): the Fisher equation, the domain are taken as $\Omega = [-20, 20]\times[-20,20]$ and mesh parameters $h_x=h_y=1/8, \tau=1/384$;
(c)--(d): the Chafee-Infante equation, the domain are taken as $\Omega = [-10, 10]\times[-10,10]$ and mesh parameters $h_x=h_y=1/16, \tau=1/1536$.} \label{fig2}
\end{figure}
     \vspace{-15mm}
\begin{table}[H]
\caption{The errors in $L^{\infty}$-norm versus grid sizes reduction and convergence orders of the corrected
difference scheme \eqref{eqn4.3} for the Fisher equation and Chafee-Infante equation under the \textbf{Dirichlet} boundary conditions in Example \ref{Exam4}}
\vspace{-4mm}
\centering
\setlength\tabcolsep{2.1mm}{
\begin{tabular}{|c|ccccc|ccccc|}
\hline
\multicolumn{0}{|c|}{ }&\multicolumn{5}{c|}{\textbf{Case~I}}&\multicolumn{5}{c|}{\textbf{Case~II}}\\
         \cline{2-6}\cline{7-11}\diagbox{R}{P}
         & $m_1$  & $m_2$ & $n$    & ${E}_{\infty}(h_x,h_y,\tau)$& ${\rm Ord_G}$  & $m_1$  & $m_2$ & $n$  & ${E}_{\infty}(h_x,h_y,\tau)$& ${\rm Ord_G}$  \\
         \cline{1-6}\cline{7-11}
         &$10$ & $10$  & $700$    & $3.2948e-9$  & $*$     & $10$  & $10$    & $700$    & $2.9377e-8$    & $*$      \\
         &$20$ & $20$  & $2800$   & $1.2599e-9$  & $1.3869$& $20$  & $20$    & $2800$   & $7.9078e-9$    & $1.8934$  \\
  $r=1/7$&$40$ & $40$  & $11200$  & $3.4607e-10$ & $1.8642$& $40$  & $40$    & $11200$  & $2.0121e-9$    & $1.9745$  \\
         &$80$ & $80$  & $44800$  & $8.8370e-11$ & $1.9694$& $80$  & $80$    & $44800$  & $5.0524e-10$   & $1.9937$  \\
         &$160$& $160$ & $179200$ & $2.2219e-11$ & $1.9918$& $160$ & $160$   & $179200$ & $1.2645e-10$   & $1.9984$  \\
 \hline
         & $10$ & $10$    & $600$    & $4.2836e-9$    & $*$& $10$ & $10$    & $600$& $4.4406e-9$    & $*$               \\
         & $20$ & $20$    & $2400$   & $2.6769e-10$   & $\textbf{4.0002}$   & $20$ & $20$    & $2400$   & $2.7710e-10$   & $\textbf{3.9890}$  \\
$r=\textbf{1/6}$& $40$&$40$& $9600$  & $1.6730e-11$   & $\textbf{4.0000}$   & $40$ & $40$    & $9600$& $1.7312e-11$      & $\textbf{3.9883}$  \\
         & $80$ & $80$    & $38400$  & $1.0456e-12$   & $\textbf{4.0001}$   & $80$ & $80$    & $38400$  & $1.0825e-12$   & $\textbf{3.9993}$  \\
         & $160$& $160$   & $153600$ & $6.5337e-14$   & $\textbf{4.0003}$  & $160$ & $160$   & $153600$ & $6.7946e-14$   & $\textbf{3.9982}$  \\
 \hline
         & $10$  & $10$   & $500$    & $1.3295e-8$   &  $*$ & $10$  & $10$   & $500$     & $5.1364e-8$   & $*$      \\
         & $20$  & $20$   & $2000$   & $2.3143e-9$   &  $2.5223$    & $20$  & $20$   & $2000$    & $1.1709e-8$   & $2.1331$  \\
  $r=1/5$& $40$  & $40$   & $8000$   & $5.1745e-10$  &  $2.1611$    & $40$  & $40$   & $8000$    & $2.8569e-9$   & $2.0351$  \\
         & $80$  & $80$   & $32000$  & $1.2585e-10$  &  $2.0398$    & $80$  & $80$   & $32000$   & $7.0982e-10$  & $2.0089$  \\
         & $160$ & $160$  & $128000$ & $3.1237e-11$  &  $2.0103$    & $160$ & $160$  & $128000$  & $1.7719e-10$  & $2.0022$  \\
  \hline
\end{tabular}\label{tab6a}
}
\end{table}
\vspace{-12mm}
\begin{table}[H]
\caption{The errors in $L^{\infty}$-norm versus grid sizes reduction and convergence orders of the corrected
difference scheme \eqref{eqn4.3} for the Fisher equation and Chafee-Infante equation with the \textbf{Neumann} boundary conditions in Example \ref{Exam4}}
\vspace{-4mm}
\centering
\setlength\tabcolsep{2.1mm}{
\begin{tabular}{|c|ccccc|ccccc|}
\hline
\multicolumn{0}{|c|}{ }&\multicolumn{5}{c|}{\textbf{Case~I}}&\multicolumn{5}{c|}{\textbf{Case~II}}\\
         \cline{2-6}\cline{7-11}\diagbox{R}{P}
         & $m_1$  & $m_2$ & $n$    & ${E}_{\infty}(h_x,h_y,\tau)$& ${\rm Ord_G}$  & $m_1$  & $m_2$ & $n$  & ${E}_{\infty}(h_x,h_y,\tau)$& ${\rm Ord_G}$  \\
         \cline{1-6}\cline{7-11}
         &$10$ & $10$  & $700$    & $3.7818e-7$ & $*$     & $10$  & $10$    & $700$    & $8.5997e-7$   & $*$       \\
         &$20$ & $20$  & $2800$   & $9.9509e-8$ & $1.9262$& $20$  & $20$    & $2800$   & $2.0788e-7$   & $2.0486$  \\
  $r=1/7$&$40$ & $40$  & $11200$  & $2.5240e-8$ & $1.9791$& $40$  & $40$    & $11200$  & $5.1631e-8$   & $2.0094$  \\
         &$80$ & $80$  & $44800$  & $6.3342e-9$ & $1.9945$& $80$  & $80$    & $44800$  & $1.2890e-8$   & $2.0020$  \\
         &$160$& $160$ & $179200$ & $1.5852e-9$ & $1.9985$& $160$ & $160$   & $179200$ & $3.2215e-9$   & $2.0005$  \\
 \hline
         & $10$ & $10$    & $600$    & $7.7947e-8$   & $*$& $10$ & $10$   & $600$ & $6.8543e-8$    & $*$               \\
         & $20$ & $20$    & $2400$   & $5.0469e-9$   & $\textbf{3.9490}$  & $20$ & $20$    & $2400$    & $4.6223e-9$   & $\textbf{3.8903}$  \\
$r=\textbf{1/6}$& $40$&$40$& $9600$  & $3.2070e-10$  & $\textbf{3.9761}$  & $40$ & $40$    & $9600$   & $2.9933e-10$      & $\textbf{3.9488}$  \\
         & $80$ & $80$    & $38400$  & $2.0186e-11$  & $\textbf{3.9898}$  & $80$ & $80$    & $38400$  & $1.9060e-11$   & $\textbf{3.9731}$  \\
         & $160$& $160$   & $153600$ & $1.2221e-12$  & $\textbf{4.0459}$  & $160$ & $160$   & $153600$ & $1.2601e-12$   & $\textbf{3.9190}$  \\
 \hline
         & $10$  & $10$  & $500$    & $6.3087e-7$ & $*$      & $10$  & $10$  & $500$    & $1.1272e-6$ & $*$   \\
         & $20$  & $20$  & $2000$   & $1.4610e-7$ & $2.1104$ & $20$  & $20$  & $2000$   & $2.8694e-7$ & $1.9739$  \\
  $r=1/5$& $40$  & $40$  & $8000$   & $3.5778e-8$ & $2.0299$ & $40$  & $40$  & $8000$   & $7.2064e-8$ & $1.9934$  \\
         & $80$  & $80$  & $32000$  & $8.8960e-9$ & $2.0078$ & $80$  & $80$  & $32000$  & $1.8034e-8$ & $1.9986$  \\
         & $160$ & $160$ & $128000$ & $2.2209e-9$ & $2.0020$ & $160$ & $160$ & $128000$ & $4.5095e-9$ & $1.9997$  \\
  \hline
\end{tabular}\label{tab6b}
}
\end{table}
\vspace{-5mm}

The numerical results for these two problems with $\tau=h^2$ and $T=1$ are listed in Table \ref{tab6a} for the Dirichlet boundary conditions and in Table \ref{tab6b} for the Neumann Boundary conditions. We clearly see that when the step-ratio $r=1/6$, the numerical accuracy is fourth-order. Meanwhile, the numerical results are much better than that calculated by using other step-ratios.
For example, when $m_1=m_2 = 160$, the maximum numerical error $10^{-14}$ is achieved with the optimal step-ratio $r=1/6$, which is about $1/1000$ of that obtained by the same spatial step sizes and smaller temporal step size. A similar phenomenon is observed under the Neumann boundary conditions.
The numerical solution behavior for both problems and corresponding error surfaces are displayed in Figure \ref{fig2} with the optimal step-ratio $r=1/6$.
\end{example}

\begin{example}\label{Exam5}
  Next, we solve the two-dimensional scalar quasi-linear Burgers' equation {\rm \cite{MS2007}} as
  \begin{align*}
    u_t+ (u^2/2)_x + (u^2/2)_y = \mu\Delta u,\quad \mathbf{x}\in \Omega, \quad t\in (0,1],
  \end{align*}
  where $\mu$ is the viscous coefficient.
  The initial and boundary value conditions are taken from the exact solution
  \begin{align*}
    u(\mathbf{x},t) = \frac{2\mu \pi \sin(\pi(x+y))\exp(-2\mu \pi^2 t)}{2+ \cos(\pi(x+y))\exp(-2\mu \pi^2 t)}.
  \end{align*}
  The following two sets of diffusion coefficients are considered.
\begin{itemize}
  \item []{\rm \textbf{Case~I:~}} $\mu = 1;$
  \quad {\rm \textbf{Case~II:~}} $\mu= 0.01;$
\end{itemize}

  The numerical convergence results on the domain $\Omega=[0,1]\times[0,1]$ are listed in Table \ref{tab7a} for the Dirichlet boundary conditions and in Table \ref{tab7b} for the Neumann boundary conditions.
  They confirm that the corrected difference scheme is still fourth-order accuracy
  when the step-ratio $r=1/6$ and second-order accuracy for other step-ratios,
  which fully demonstrates the good performance in accuracy of the novel corrected scheme \eqref{eqn4.3}.
  Comparing the results in Tables \ref{tab7a} and \ref{tab7b}, we see that the numerical errors for the Neumann boundary conditions are larger than those for the Dirichlet boundary conditions. This is mainly caused by the discretization on the boundary conditions.
  The numerical simulation for the Burgers' equation with viscous coefficients $\mu =1$ $(t=0.1)$,
  $\mu = 0.01$ $(t=0.5)$ is demonstrated in Figure \ref{fig3} respectively on the domain $\Omega= [-3,3]\times[-3, 3]$.
  We see that the numerical error surfaces keeps very low whether the viscous coefficient $\mu$ is large or small.
\vspace{-10mm}
\begin{figure}[htbp]
\centering
 \subfigure[Numerical surface, \textbf{view}(38,30)]{\centering
\includegraphics[width=0.24\textwidth]{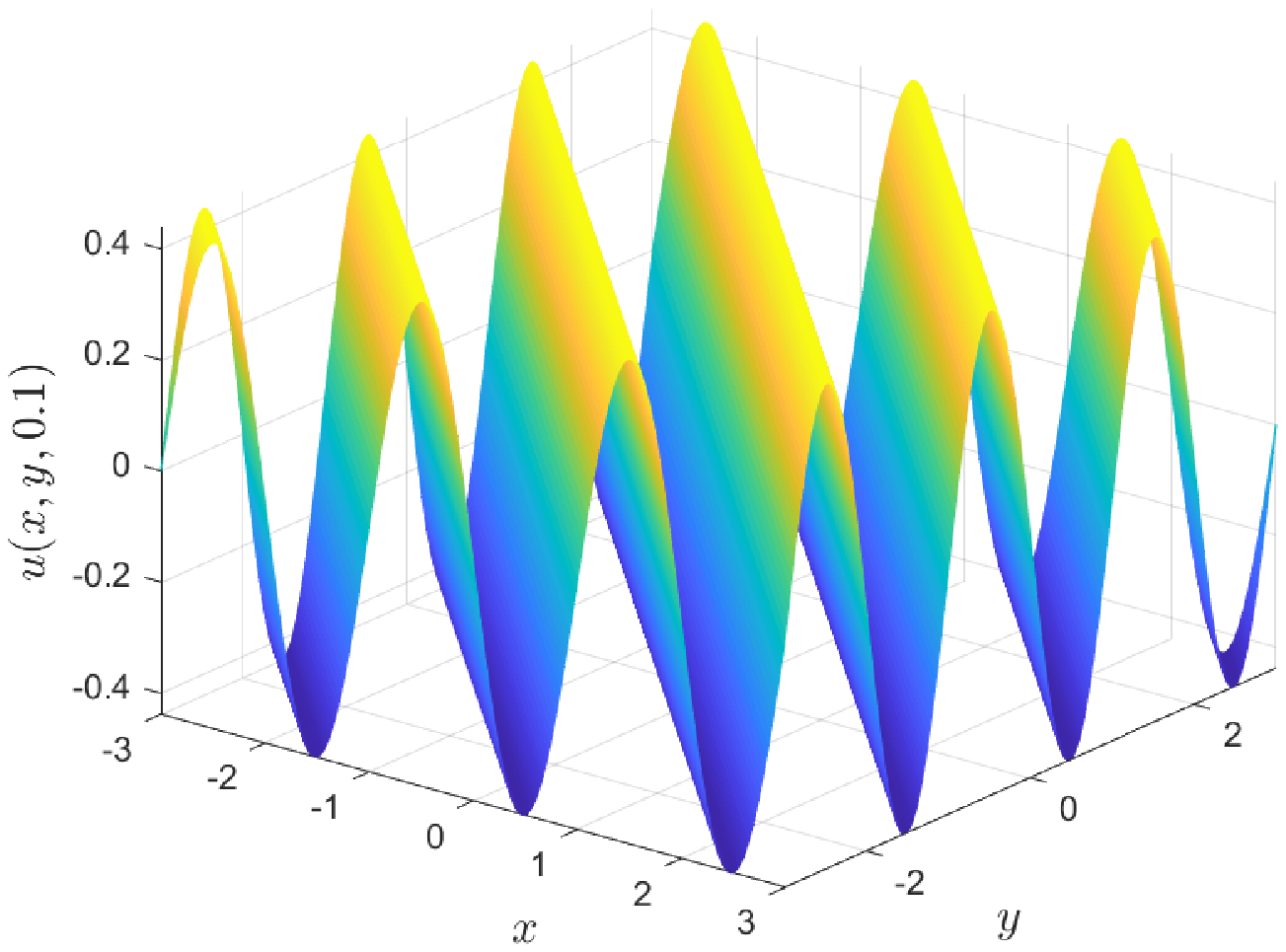}
}\subfigure[Error surface, \textbf{view}(38,30)]{\centering
\includegraphics[width=0.24\textwidth]{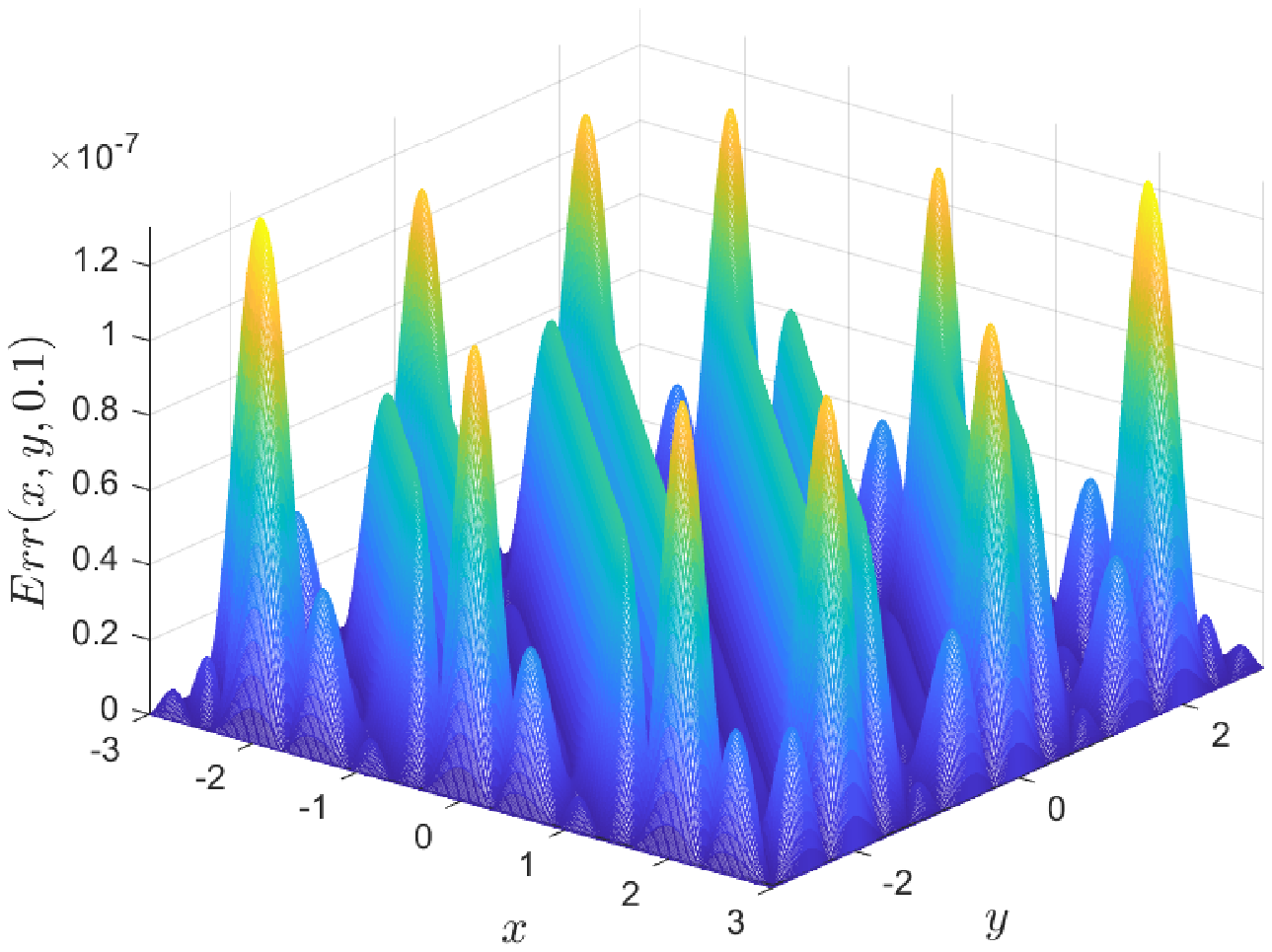}
}
\subfigure[Numerical surface, \textbf{view}(38,30)]{\centering
\includegraphics[width=0.24\textwidth]{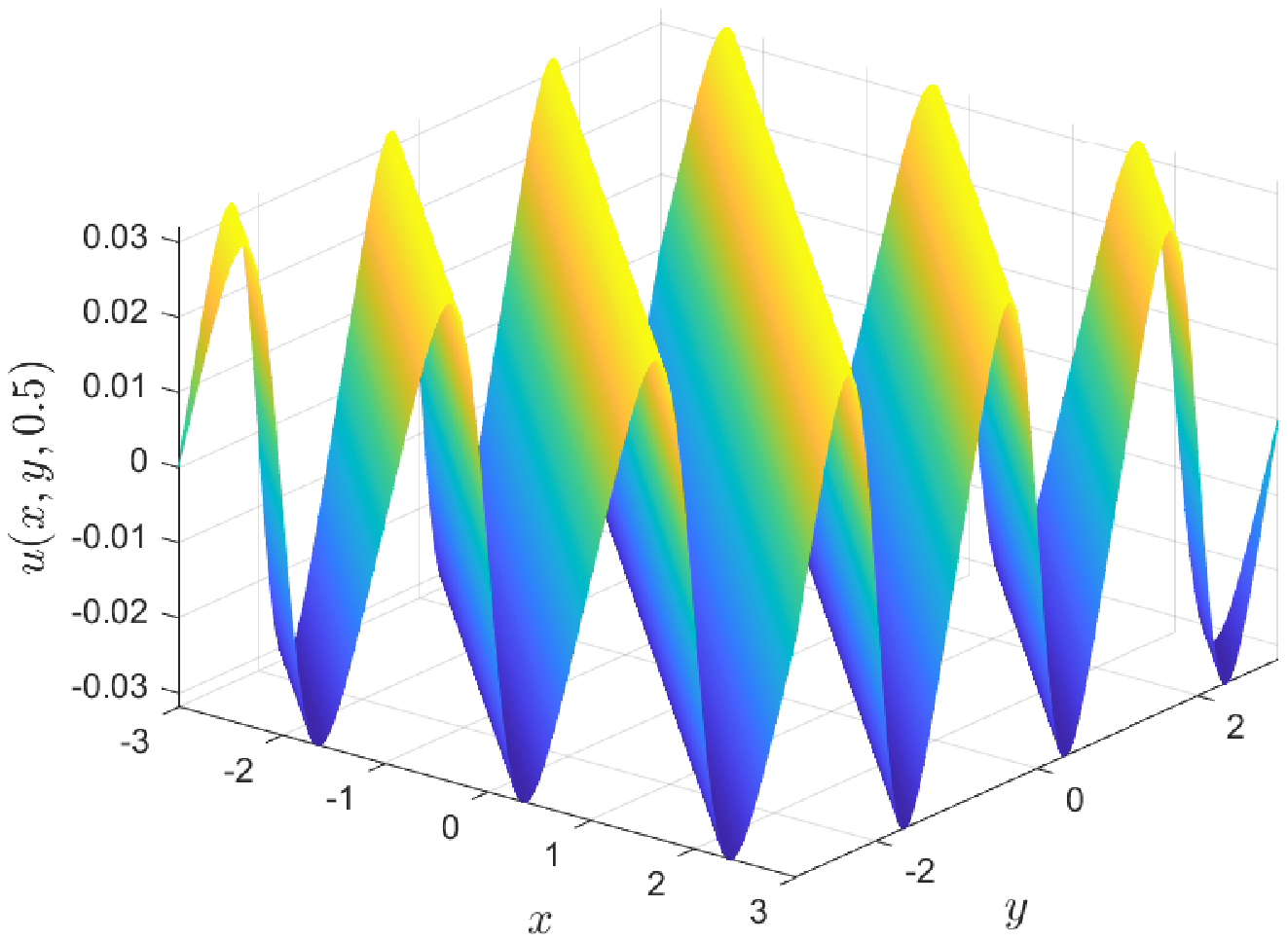}
}\subfigure[Error surface, \textbf{view}(38,30)]{\centering
\includegraphics[width=0.24\textwidth]{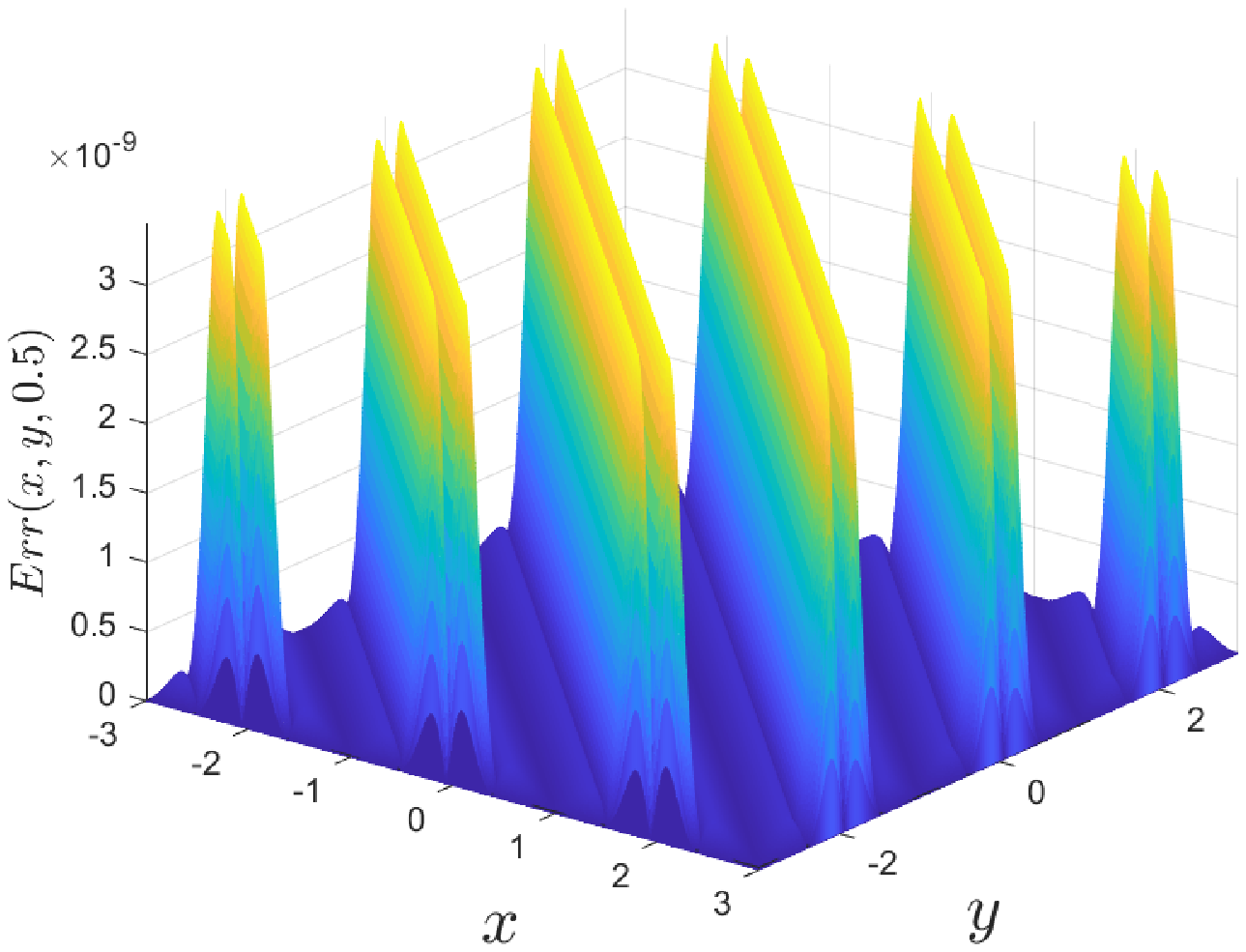}
}
\caption{Numerical surfaces and error surfaces for the Burgers' equation, (a)--(b): $\mu=1$, $M_1=M_2=320$, $N=1707$, $t=0.1$
 (c)--(d): $\mu=0.01$, $M_1=M_2=640$, $N=341$, $t=0.5$.} \label{fig3}
\end{figure}
\vspace{-10mm}
\begin{table}[H]
\caption{The errors in $L^{\infty}$-norm versus grid sizes reduction and convergence orders of the corrected
difference scheme \eqref{eqn4.3} for the viscous Burgers' equation under the \textbf{Dirichlet} boundary conditions in Example \ref{Exam5}}
\vspace{-3mm}
\centering
\setlength\tabcolsep{2.1mm}{
\begin{tabular}{|c|ccccc|ccccc|}
\hline
\multicolumn{0}{|c|}{ }&\multicolumn{5}{c|}{\textbf{Case~I}}&\multicolumn{5}{c|}{\textbf{Case~II}}\\
         \cline{2-6}\cline{7-11}\diagbox{R}{P}
         & $m_1$  & $m_2$ & $n$    & ${E}_{\infty}(h_x,h_y,\tau)$& ${\rm Ord_G}$  & $m_1$  & $m_2$ & $n$  & ${E}_{\infty}(h_x,h_y,\tau)$& ${\rm Ord_G}$  \\
         \cline{1-6}\cline{7-11}
         & $10$  & $10$    & $700$    & $8.3065e-12$   & $*$      &$10$ & $10$  & $7$    & $3.0122e-5$ & $*$      \\
         & $20$  & $20$    & $2800$   & $2.0419e-12$   & $2.0243$ &$20$ & $20$  & $28$   & $1.0200e-5$ & $1.5622$ \\
  $r=1/7$& $40$  & $40$    & $11200$  & $5.0834e-13$   & $2.0061$ &$40$ & $40$  & $112$  & $2.7123e-6$ & $1.9110$ \\
         & $80$  & $80$    & $44800$  & $1.2710e-13$   & $1.9998$ &$80$ & $80$  & $448$  & $6.8987e-7$ & $1.9751$ \\
         & $160$ & $160$   & $179200$ & $3.1782e-14$   & $1.9997$ &$160$& $160$ & $1792$ & $1.7308e-7$ & $1.9949$ \\
 \hline
         & $10$ & $10$    & $600$    & $1.2581e-13$    & $*$      & $10$ & $10$    & $6$   & $2.0126e-5$   & $*$\\
         & $20$ & $20$    & $2400$   & $7.8107e-15$    & $\textbf{4.0096}$ & $20$ & $20$    & $24$   & $1.2675e-6$   & $\textbf{3.9890}$ \\
$r=\textbf{1/6}$& $40$ & $40$& $9600$& $4.8736e-16$    & $\textbf{4.0024}$ & $40$ & $40$& $96$  & $7.8155e-8$   & $\textbf{4.0195}$ \\
         & $80$ & $80$    & $38400$  & $3.0483e-17$    & $\textbf{3.9989}$ & $80$ & $80$    & $384$  & $4.8675e-9$   & $\textbf{4.0051}$ \\
         & $160$ & $160$  & $153600$ & $1.9059e-18$    & $\textbf{3.9994}$ & $160$ & $160$  & $1536$ & $3.0437e-10$  & $\textbf{3.9993}$ \\
 \hline
         & $10$  & $10$   & $500$     & $1.1283e-11$  & $*$      & $10$  & $10$   & $5$      & $9.2303e-5$  &  $*$\\
         & $20$  & $20$   & $2000$    & $2.8372e-12$  & $1.9916$ & $20$  & $20$   & $20$     & $1.7206e-5$  &  $2.4235$ \\
  $r=1/5$& $40$  & $40$   & $8000$    & $7.1034e-13$  & $1.9979$ & $40$  & $40$   & $80$     & $3.9883e-6$  &  $2.1091$ \\
         & $80$  & $80$   & $32000$   & $1.7785e-13$  & $1.9978$ & $80$  & $80$   & $320$    & $9.7742e-7$  &  $2.0287$ \\
         & $160$ & $160$  & $128000$  & $4.4489e-14$  & $1.9992$ & $160$ & $160$  & $1280$   & $2.4304e-7$  &  $2.0078$ \\
  \hline
\end{tabular}\label{tab7a}
}
\end{table}
\vspace{-10mm}
\begin{table}[H]
\caption{The errors in $L^{\infty}$-norm versus grid sizes reduction and convergence orders of the corrected
difference scheme \eqref{eqn4.3} for the viscous Burgers' equation under the \textbf{Neumann} boundary in Example \ref{Exam5}}
\vspace{-3mm}
\centering
\setlength\tabcolsep{2.1mm}{
\begin{tabular}{|c|ccccc|ccccc|}
\hline
\multicolumn{0}{|c|}{ }&\multicolumn{5}{c|}{\textbf{Case~I}}&\multicolumn{5}{c|}{\textbf{Case~II}}\\
         \cline{2-6}\cline{7-11}\diagbox{R}{P}
         & $m_1$  & $m_2$ & $n$    & ${E}_{\infty}(h_x,h_y,\tau)$& ${\rm Ord_G}$  & $m_1$  & $m_2$ & $n$  & ${E}_{\infty}(h_x,h_y,\tau)$& ${\rm Ord_G}$  \\
         \cline{1-6}\cline{7-11}
         & $10$  & $10$    & $700$    & $1.3885e-7$   & $*$      &$40$  & $40$  & $112$   & $2.7129e-6$ & $*$      \\
         & $20$  & $20$    & $2800$   & $1.2689e-7$   & $0.1300$ &$80$  & $80$  & $448$   & $7.2243e-7$ & $1.9089$ \\
  $r=1/7$& $40$  & $40$    & $11200$  & $3.9364e-8$   & $1.6886$ &$160$ & $160$ & $1792$  & $1.9385e-7$ & $1.8979$ \\
         & $80$  & $80$    & $44800$  & $1.0255e-8$   & $1.9406$ &$320$ & $320$ & $7168$  & $4.9489e-8$ & $1.9698$ \\
         & $160$ & $160$   & $179200$ & $2.5871e-9$   & $1.9869$ &$640$ & $640$ & $28672$ & $1.2443e-8$ & $1.9918$ \\
 \hline
         & $10$ & $10$    & $600$    & $7.1556e-7$    & $*$      & $40$ & $40$    & $96$   & $1.5489e-6$   & $*$\\
         & $20$ & $20$    & $2400$   & $4.0472e-8$    & $\textbf{4.1441}$ & $80$ & $80$    & $384$   & $9.9905e-8$   & $\textbf{3.9546}$ \\
$r=\textbf{1/6}$& $40$ & $40$& $9600$& $2.2378e-9$    & $\textbf{4.1768}$ & $160$ & $160$& $1536$  & $6.2942e-9$   & $\textbf{3.9885}$ \\
         & $80$ & $80$    & $38400$  & $1.2872e-10$   & $\textbf{4.1198}$ & $320$ & $320$    & $6144$  & $3.9461e-10$   & $\textbf{3.9955}$ \\
         & $160$ & $160$  & $153600$ & $7.6769e-12$   & $\textbf{4.0676}$ & $640$ & $640$  & $24576$ & $2.4688e-11$  & $\textbf{3.9985}$ \\
 \hline
         & $10$  & $10$   & $500$     & $1.6401e-6$  & $*$      & $40$  & $40$   & $80$      & $5.9410e-6$  &  $*$\\
         & $20$  & $20$   & $2000$    & $2.7321e-7$  & $2.5857$ & $80$  & $80$   & $320$     & $1.2168e-6$  &  $2.2877$ \\
  $r=1/5$& $40$  & $40$   & $8000$    & $6.0382e-8$  & $2.1778$ & $160$  & $160$   & $1280$     & $2.8568e-7$  &  $2.0906$ \\
         & $80$  & $80$   & $32000$   & $1.4660e-8$  & $2.0423$ & $320$  & $320$   & $5120$    & $7.0216e-8$  &  $2.0245$ \\
         & $160$ & $160$  & $128000$  & $3.6400e-9$  & $2.0098$ & $640$ & $640$  & $20480$   & $1.7479e-8$  &  $2.0062$ \\
  \hline
\end{tabular}\label{tab7b}
}
\end{table}
\end{example}
\vspace{-8mm}
\begin{example}\label{exam6}
  We then consider the solution behavior to the nonlinear problem {\rm \cite{KR1997}} as
  \begin{align*}
u_t+f(u)_x + g(u)_y = \varepsilon\Delta u, \quad \mathbf{x}\in \Omega, \; t>0.
\end{align*}
\begin{itemize}
  \item []{\rm \textbf{Case I}}.
   The initial data is given by
  \begin{equation*}
    u_0(\mathbf{x}) = \left\{
    \begin{array}{ll}
      1,& {\rm for\;} (x-0.25)^2 + (y-2.25)^2 < 0.5,\\
      0,& {\rm otherwise}
    \end{array}\right.
  \end{equation*}
  and the fluxes given by
  \begin{equation*}
  \left\{
  \begin{array}{l}
  \displaystyle  g(u) = (u-0.25)^3,\\
  \displaystyle  f(u) = u+u^2.
    \end{array}\right.
  \end{equation*}
  \item []{\rm \textbf{Case II}}.
  The initial data is given by
  \begin{equation*}
    u_0(\mathbf{x}) = \left\{
    \begin{array}{ll}
      1,& {\rm for\;} x^2 + y^2 < 0.5,\\
      0,& {\rm otherwise}
    \end{array}\right.
  \end{equation*}
  and the fluxes given by
  \begin{equation*}
  \left\{
  \begin{array}{l}
  \displaystyle  g(u) = \frac{u^2}{u^2+(1-u)^2},\\
  \displaystyle  f(u) = g(u)(1-5(1-u)^2).
    \end{array}\right.
  \end{equation*}
\end{itemize}

In {\rm \textbf{Case I}}, the boundary conditions are set to be zeros to keep the consistency and the diffusion coefficient is taken as $\varepsilon=0.1$. The temporal step size is $\tau = 0.001$ and the simulation domain is on $\Omega = [-6, 2]\times[0,8]$, which could contain a complete evolution surface. The numerical surfaces and corresponding contours are displayed in Figure \ref{fig4} for different terminal time with the optimal step-ratio $r=1/6$. We see clearly that the numerical solutions are diffusive and move from the bottom right to the upper left.
\begin{figure}[htbp]
\centering
 \subfigure[$t=0$]{\centering
\includegraphics[width=0.26\textwidth]{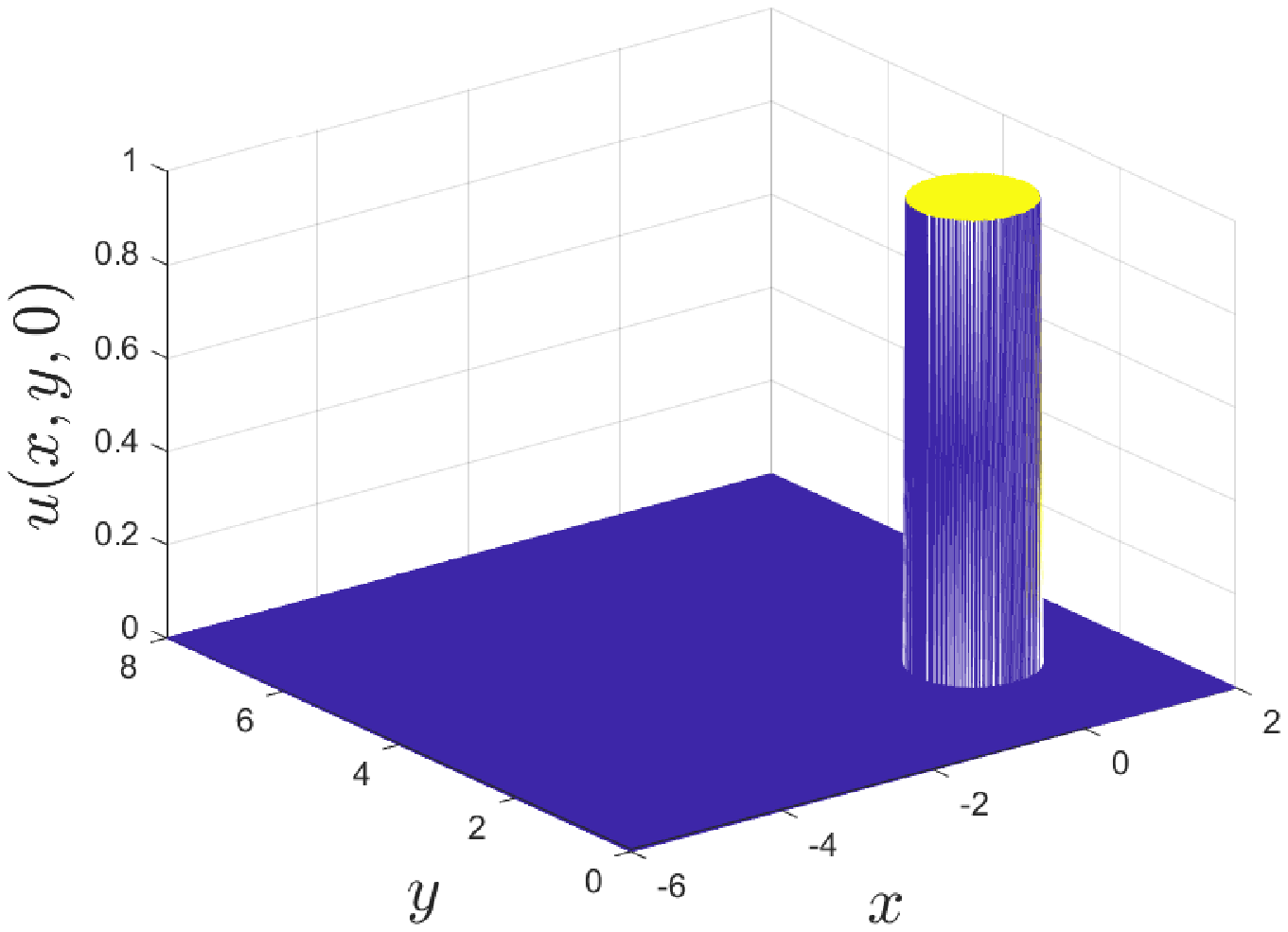}
}\hspace{-4mm}\subfigure[$t=0.5$]{\centering
\includegraphics[width=0.26\textwidth]{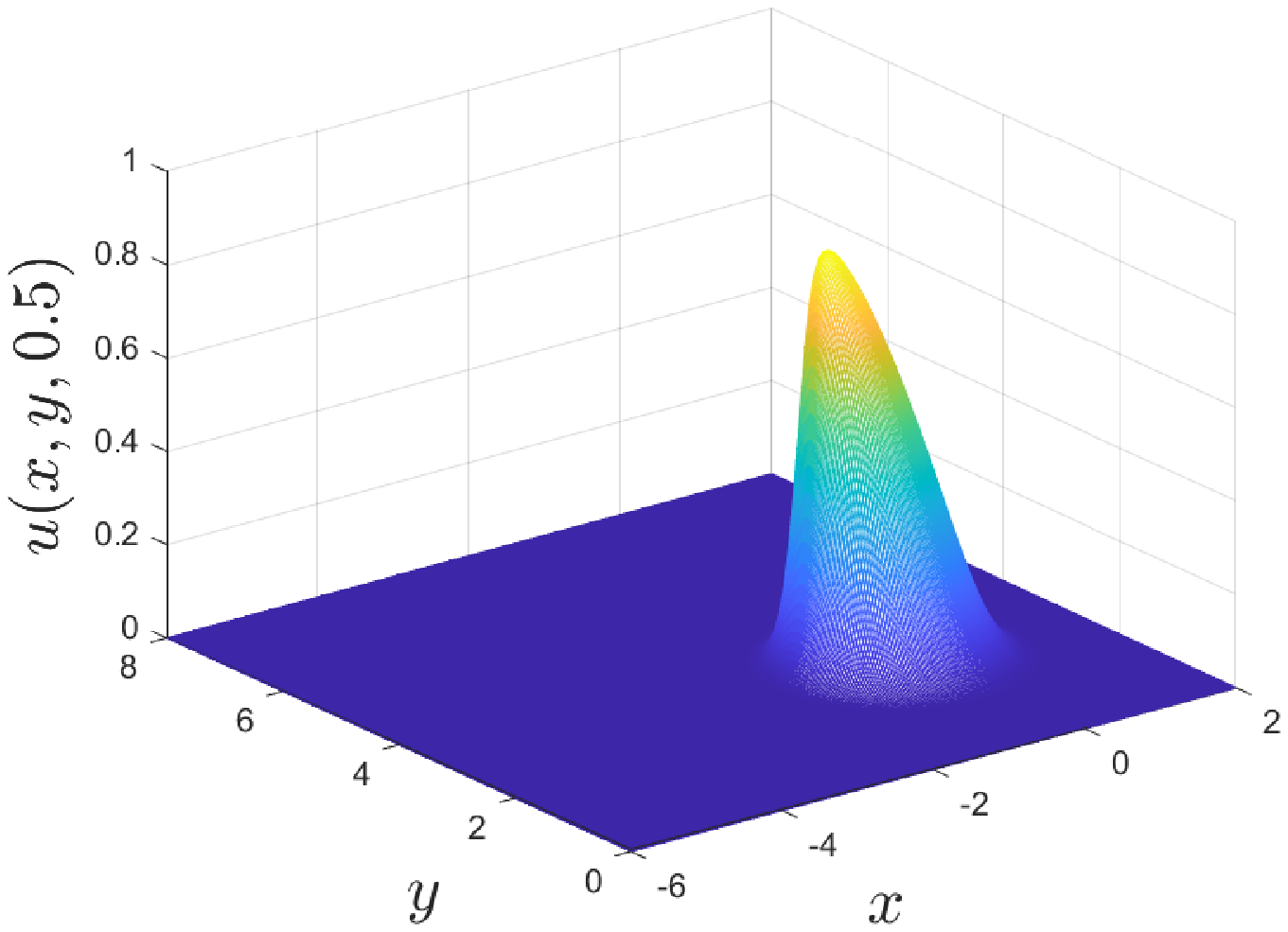}
}\hspace{-4mm}\subfigure[$t=1.5$]{\centering
\includegraphics[width=0.26\textwidth]{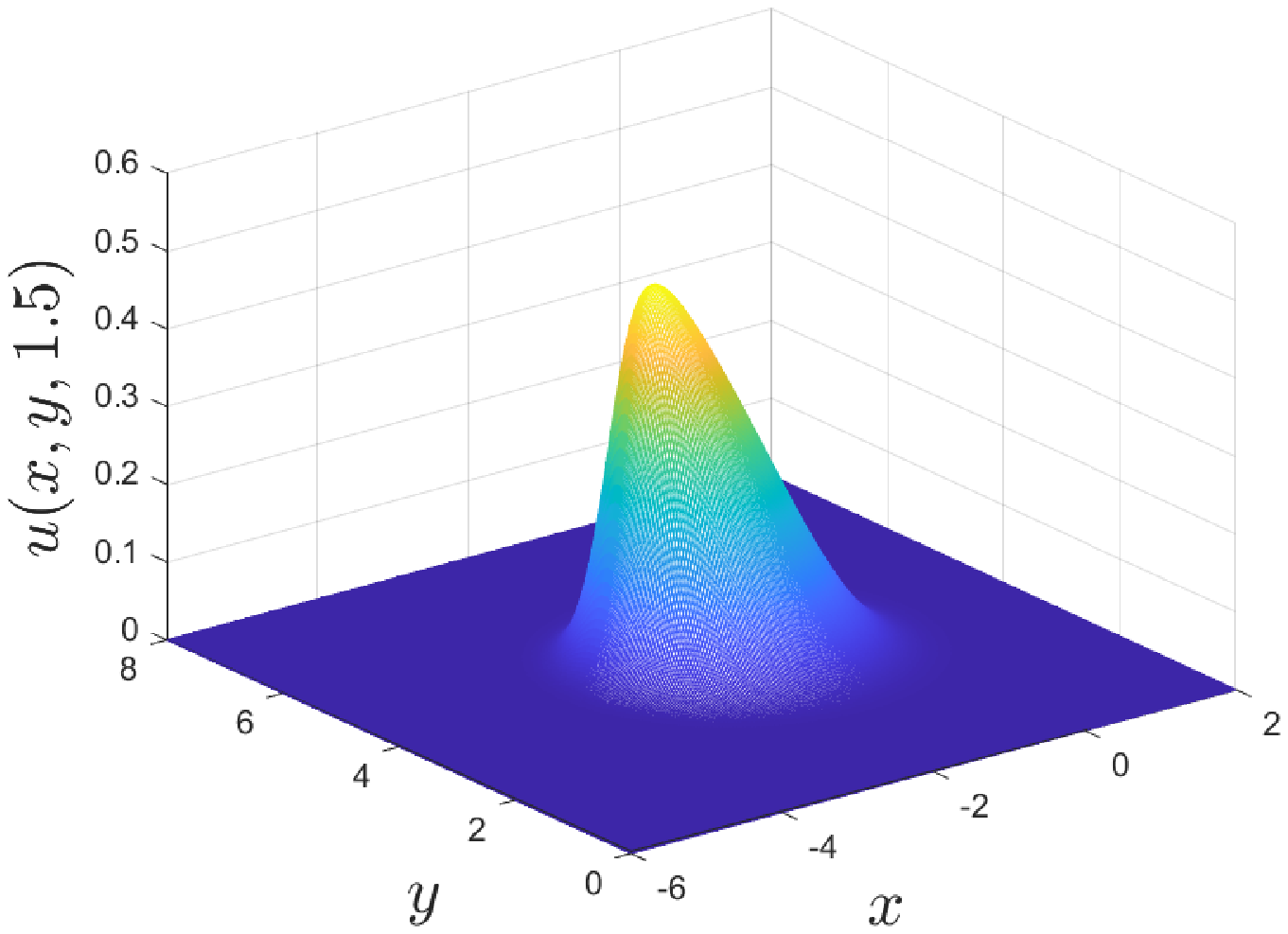}
}\hspace{-4mm}\subfigure[$t=3$]{\centering
\includegraphics[width=0.26\textwidth]{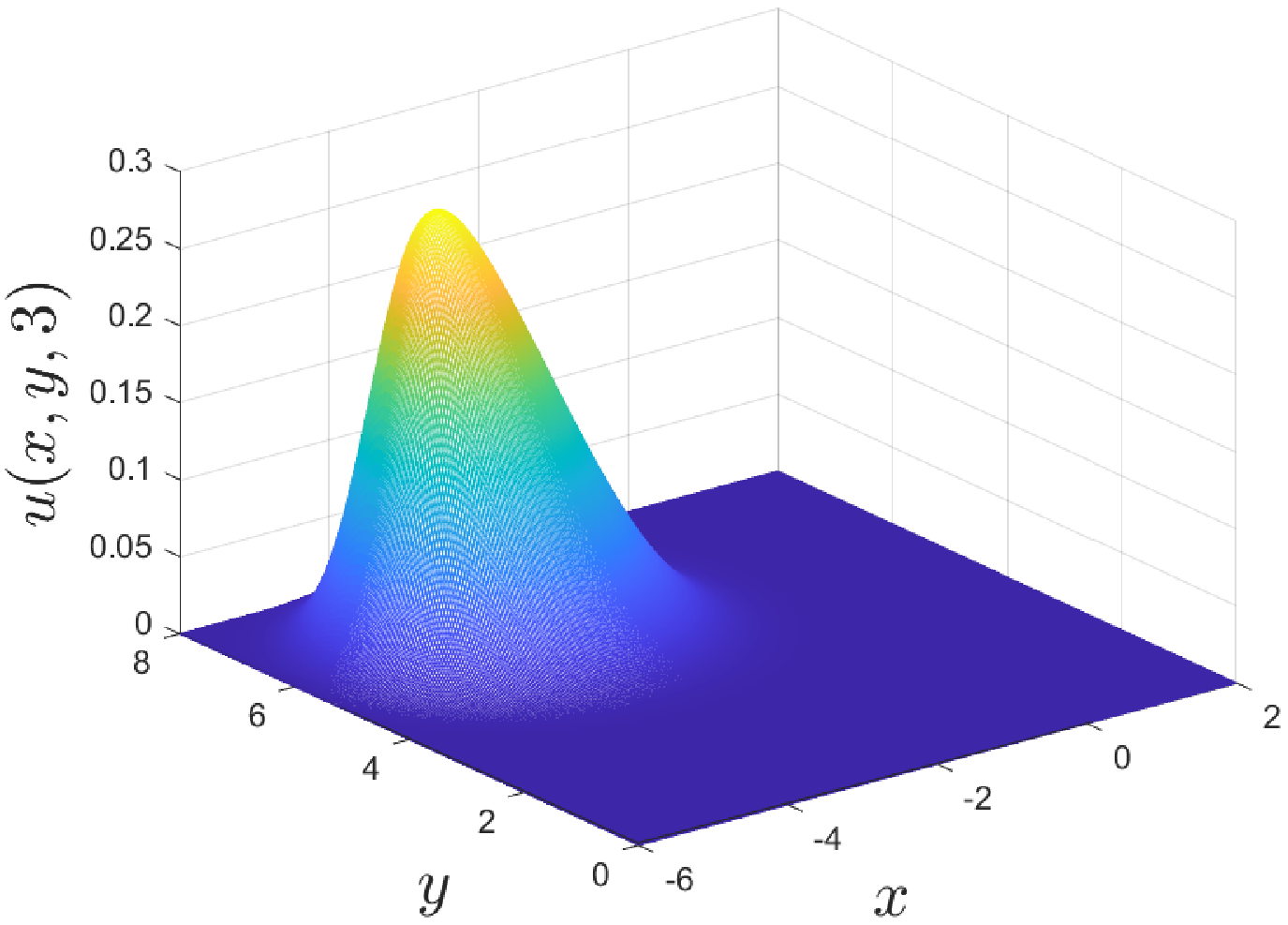}
}\hspace{-4mm}\\
 \subfigure[$t=0$]{\centering
\includegraphics[width=0.26\textwidth]{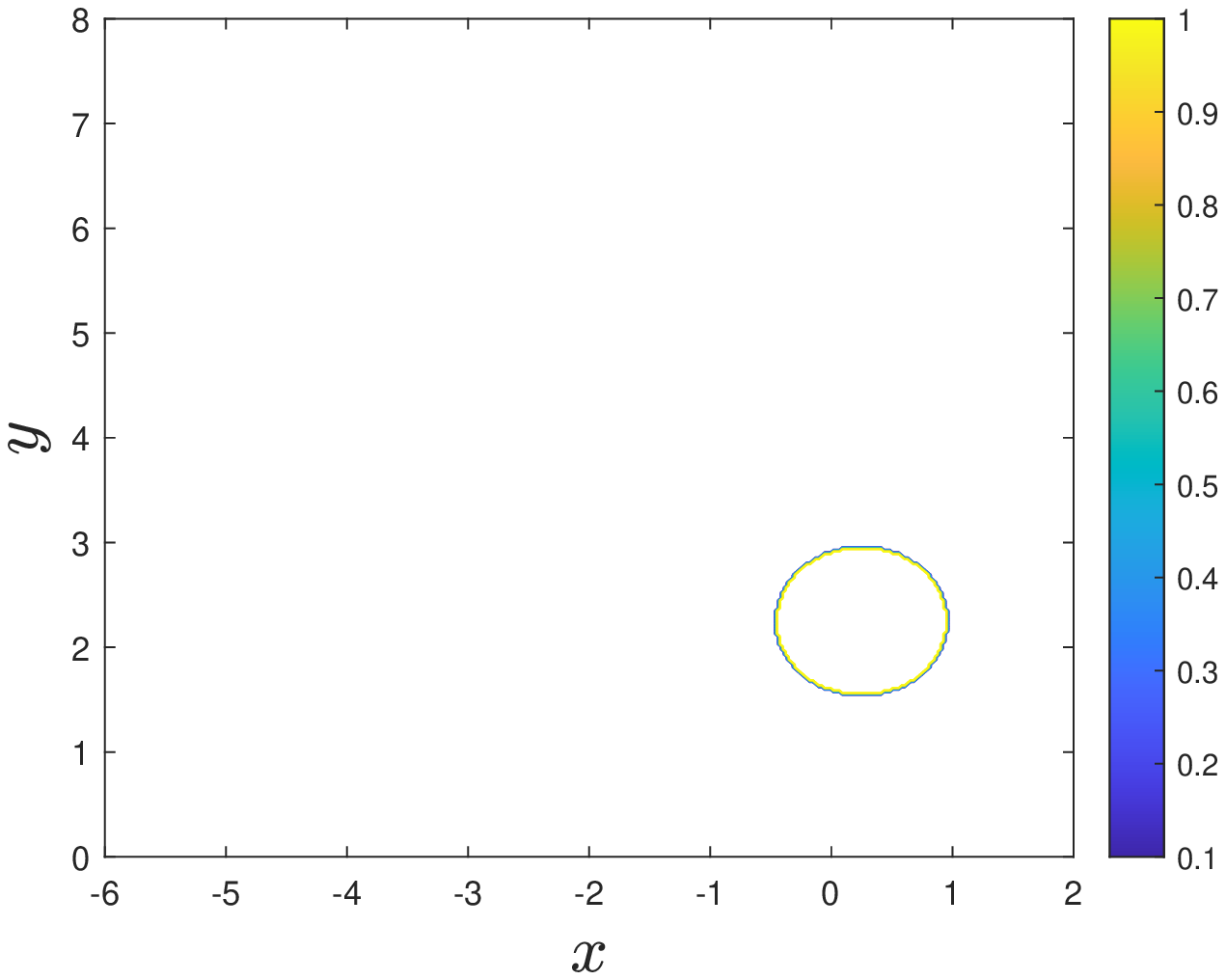}
}\hspace{-4mm}\subfigure[$t=0.5$]{\centering
\includegraphics[width=0.26\textwidth]{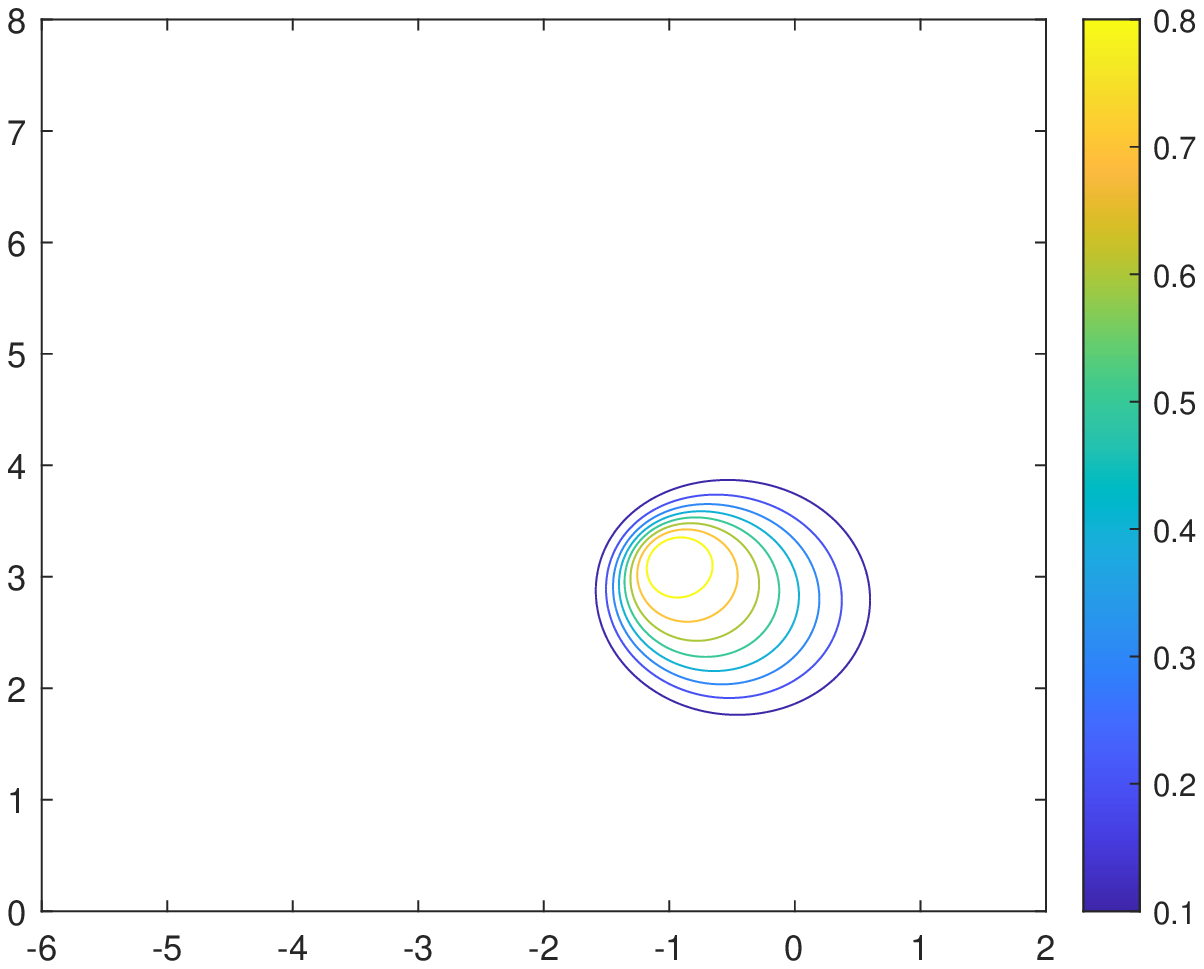}
}\hspace{-4mm}\subfigure[$t=1.5$]{\centering
\includegraphics[width=0.26\textwidth]{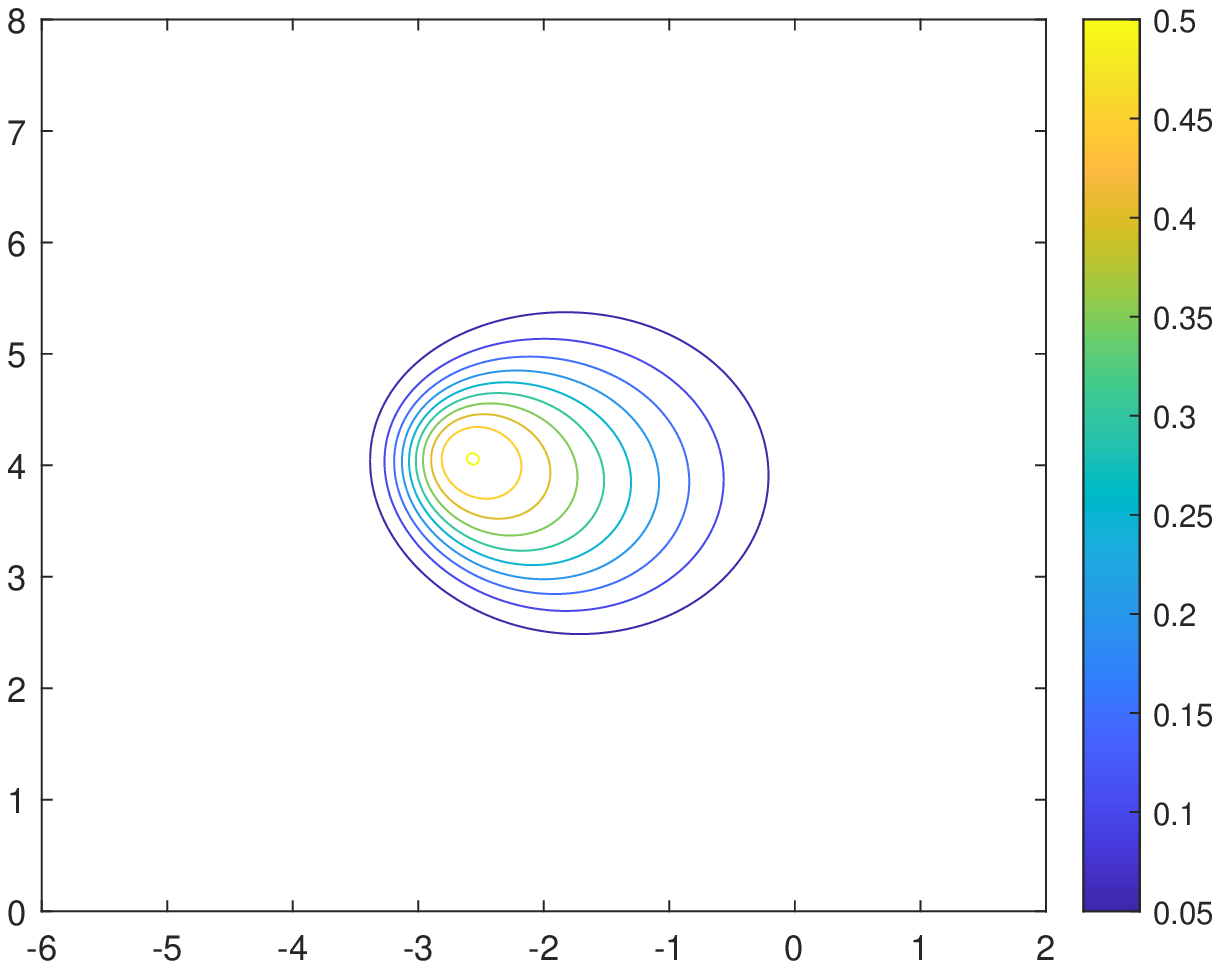}
}\hspace{-4mm}\subfigure[$t=3$]{\centering
\includegraphics[width=0.26\textwidth]{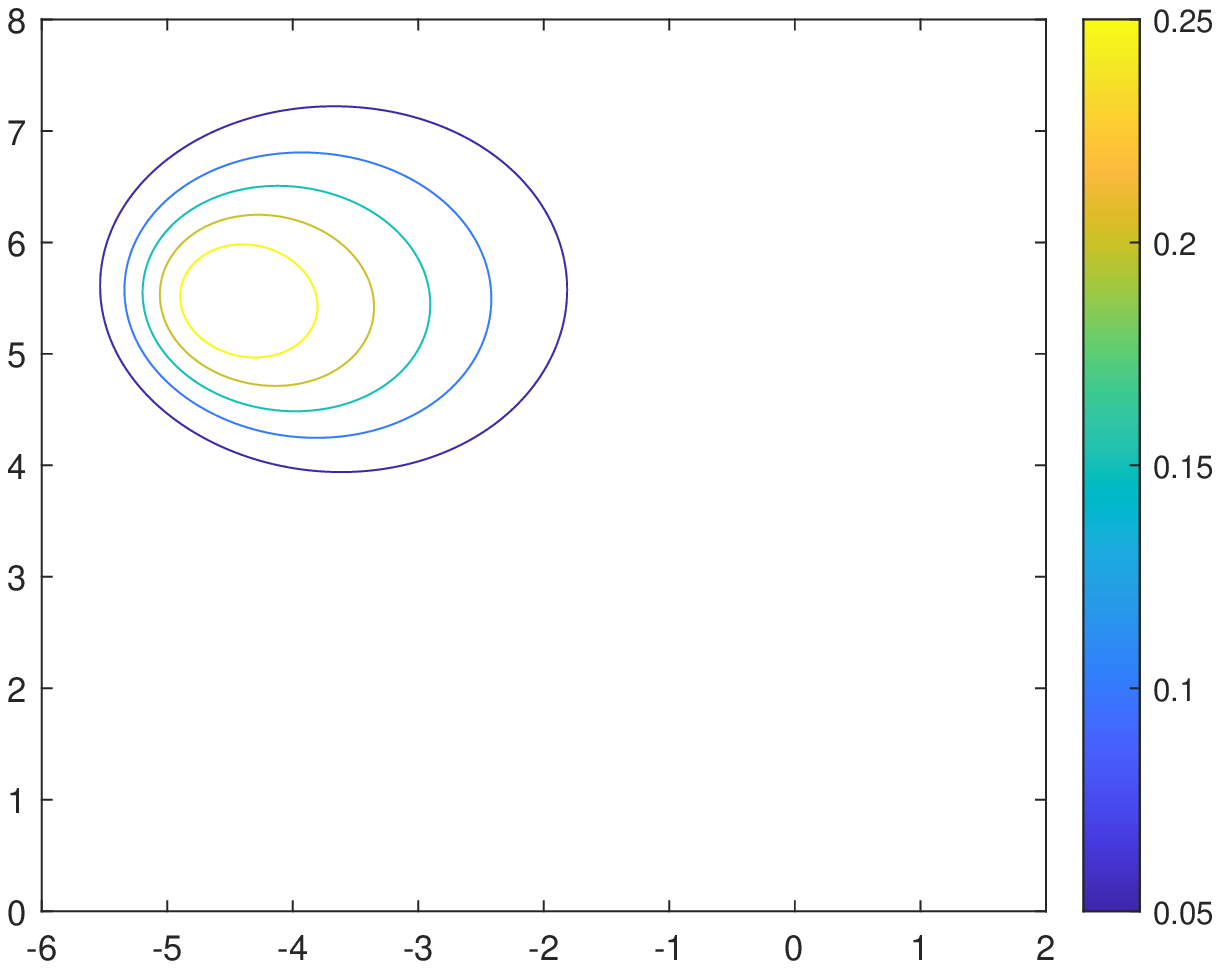}
}
\caption{Numerical evolution surfaces (a)--(d) and the corresponding contours (e)--(h).
(a) and (e): $t=0$, $m_1=m_2 = 327, n = 500$; (b) and (f): $t=0.5$, $m_1=m_2 = 327, n = 500$;
(c) and (g): $t=1.5$, $m_1=m_2 = 327, n = 1500$; (d) and (h): $t=3$, $m_1=m_2 = 327, n = 3000$.} \label{fig4}
\end{figure}
\begin{figure}[htbp]
\centering
 \subfigure[$t=0$]{\centering
\includegraphics[width=0.26\textwidth]{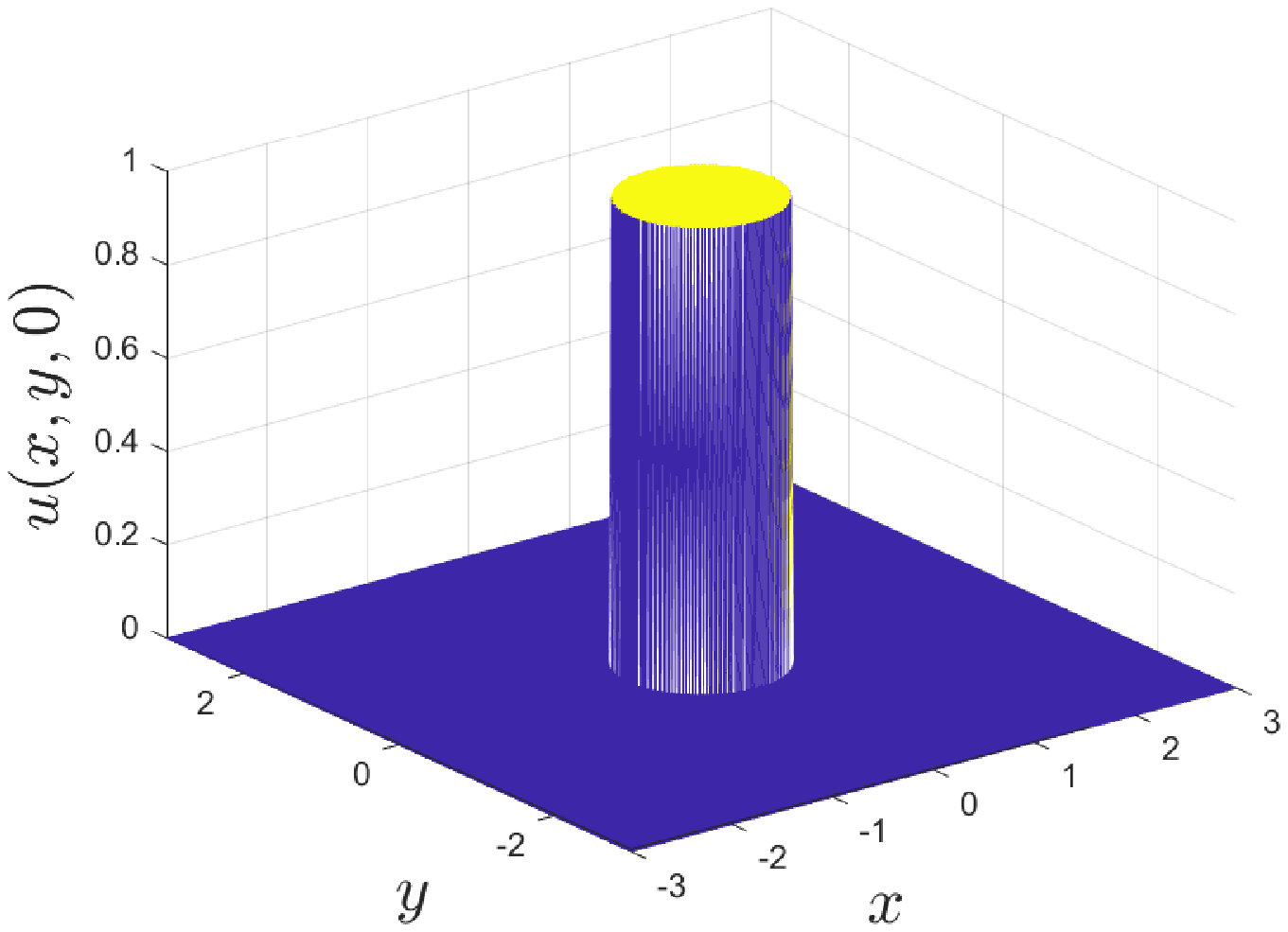}
}\hspace{-4mm}\subfigure[$t=0.5$]{\centering
\includegraphics[width=0.26\textwidth]{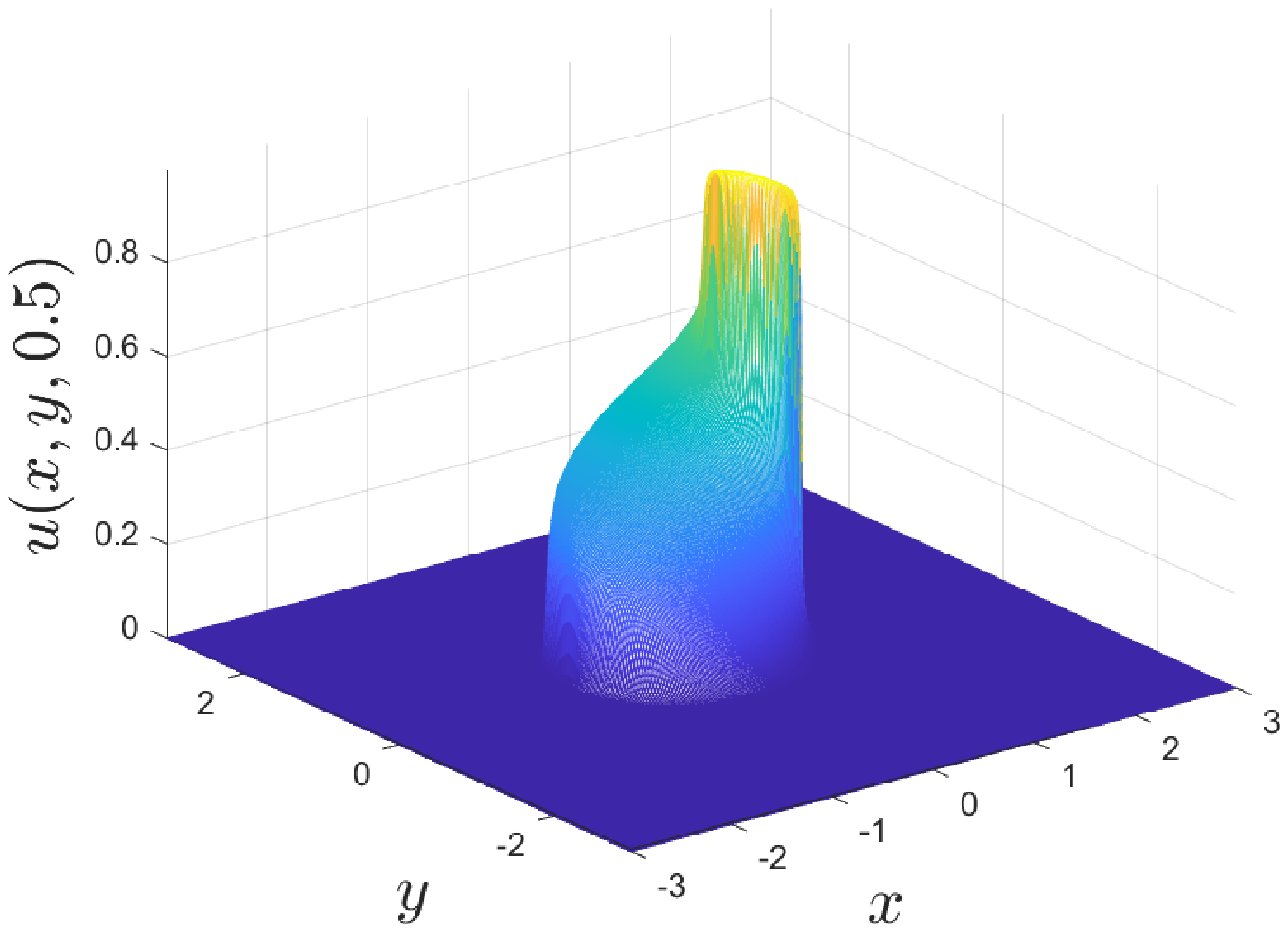}
}\hspace{-4mm}\subfigure[$t=1.5$]{\centering
\includegraphics[width=0.26\textwidth]{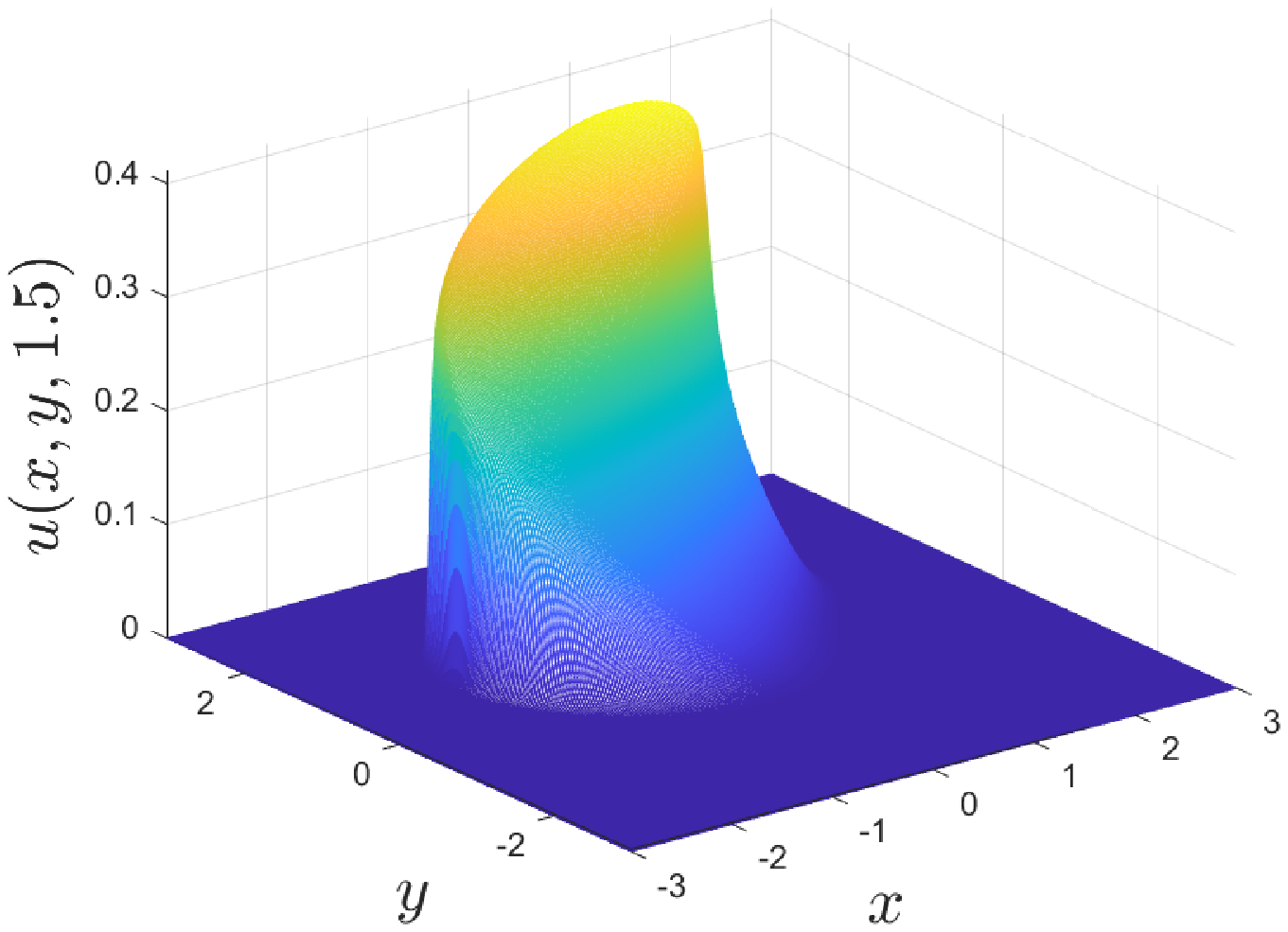}
}\hspace{-4mm}\subfigure[$t=3$]{\centering
\includegraphics[width=0.26\textwidth]{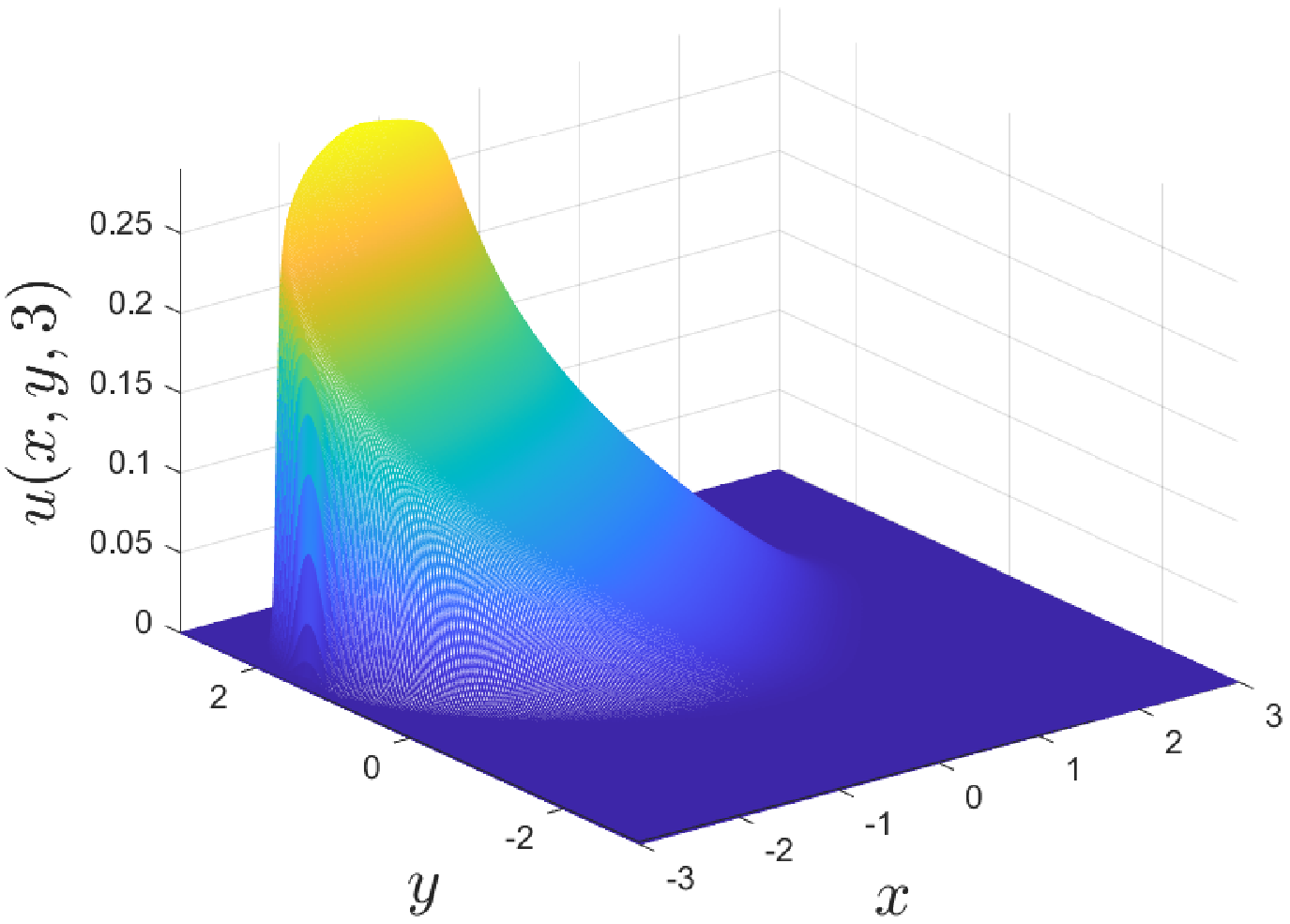}
}\hspace{-4mm}\\
\subfigure[$t=0$]{\centering
\includegraphics[width=0.26\textwidth]{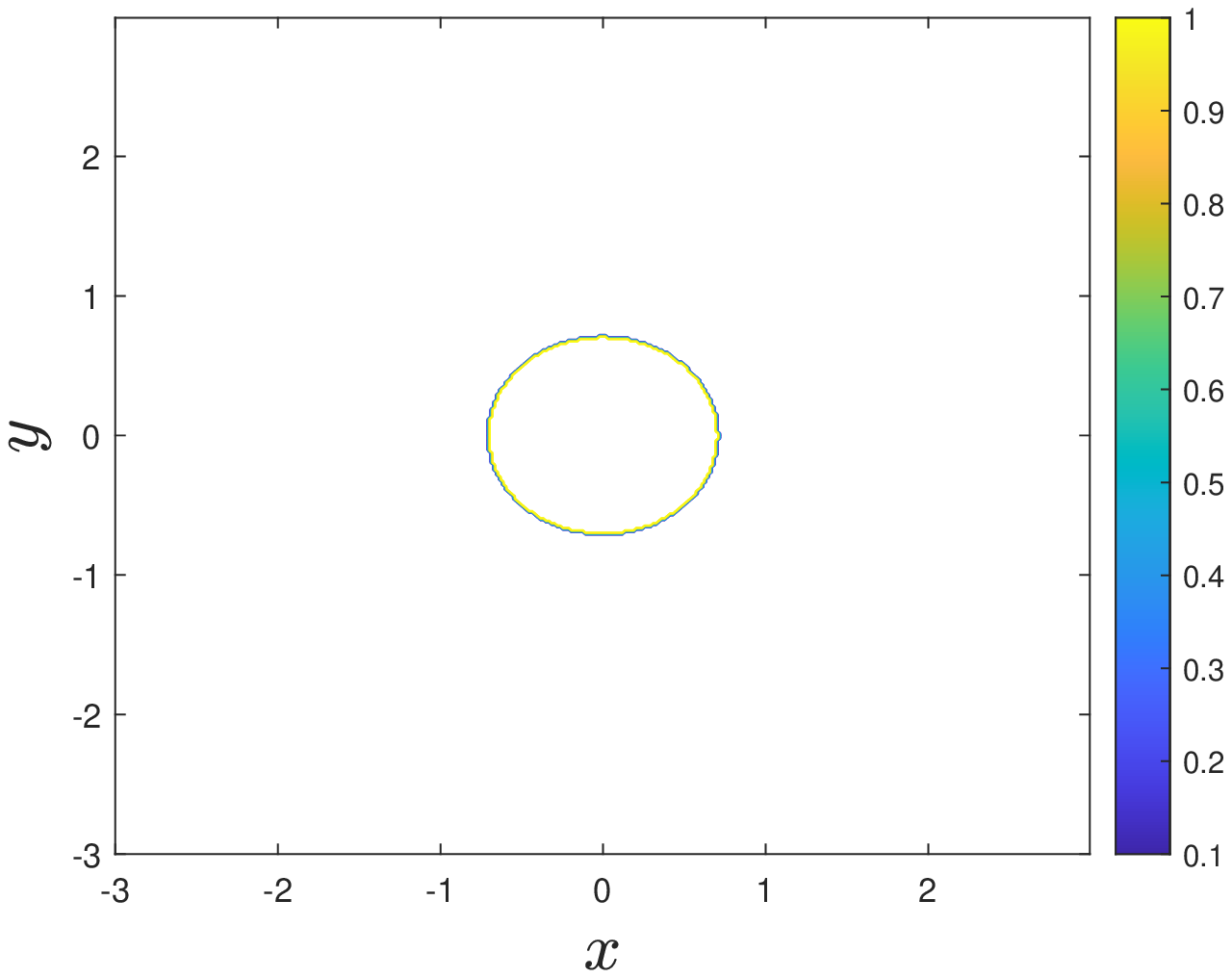}
}\hspace{-4mm}\subfigure[$t=0.5$]{\centering
\includegraphics[width=0.26\textwidth]{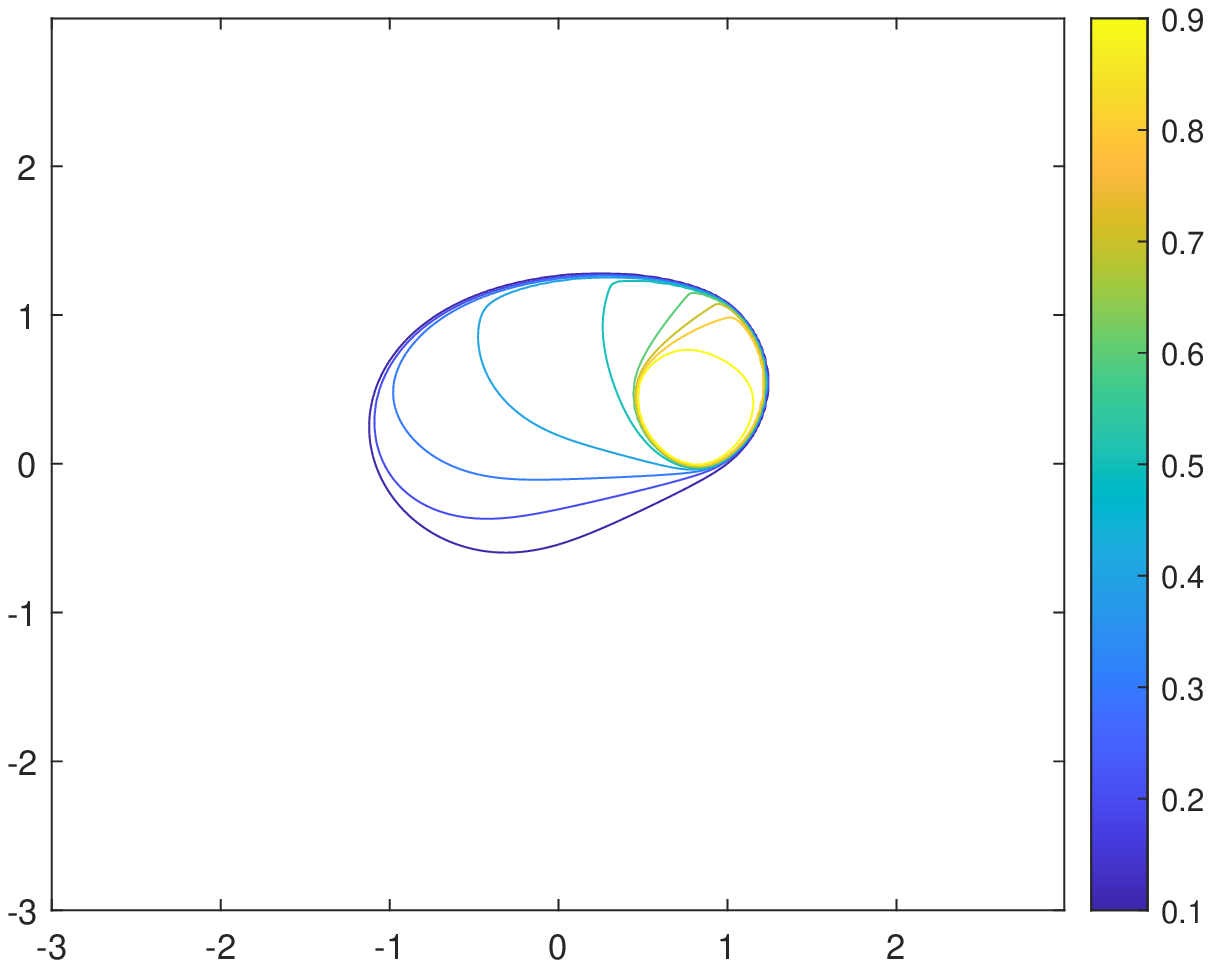}
}\hspace{-4mm}\subfigure[$t=1.5$]{\centering
\includegraphics[width=0.26\textwidth]{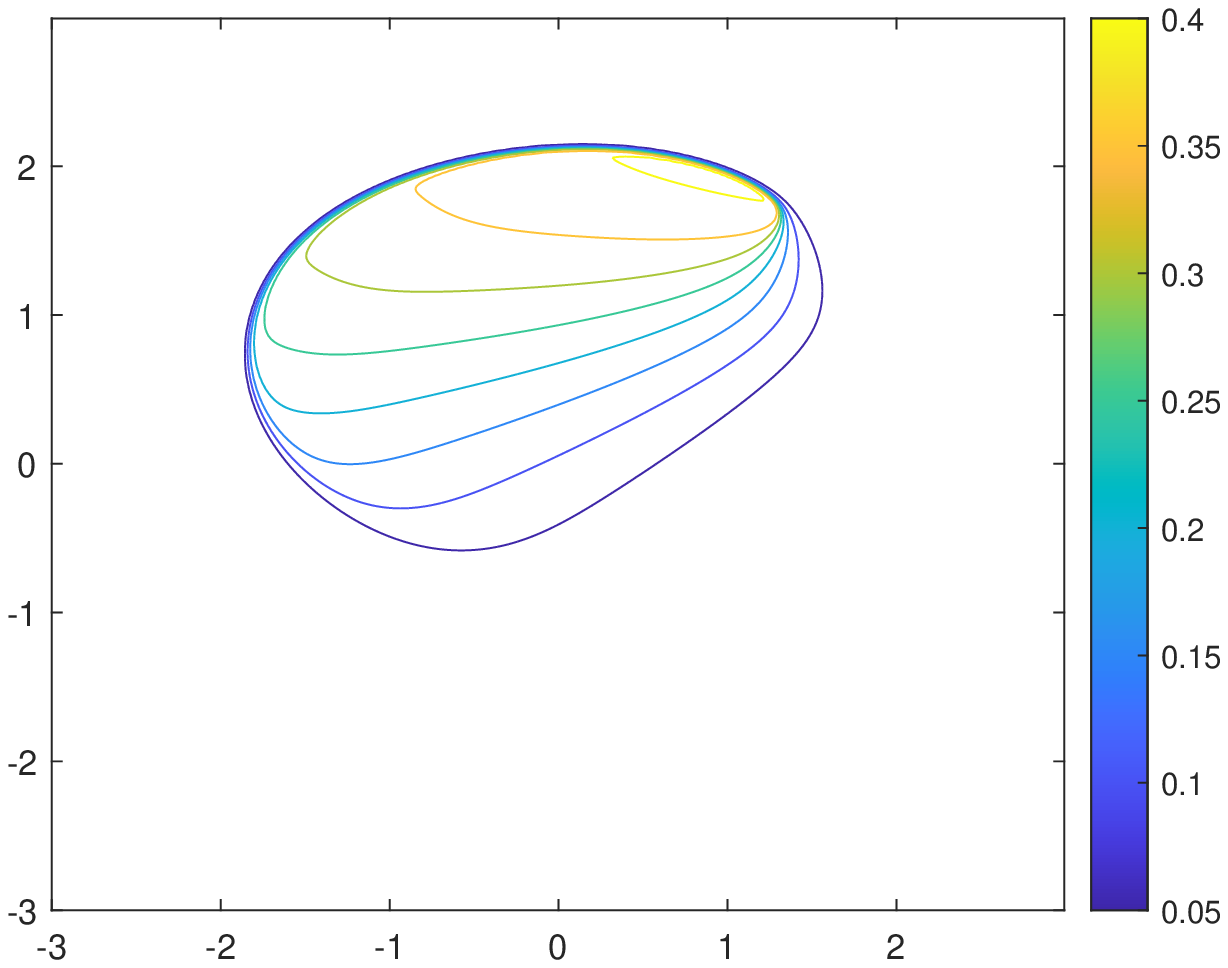}
}\hspace{-4mm}\subfigure[$t=3$]{\centering
\includegraphics[width=0.26\textwidth]{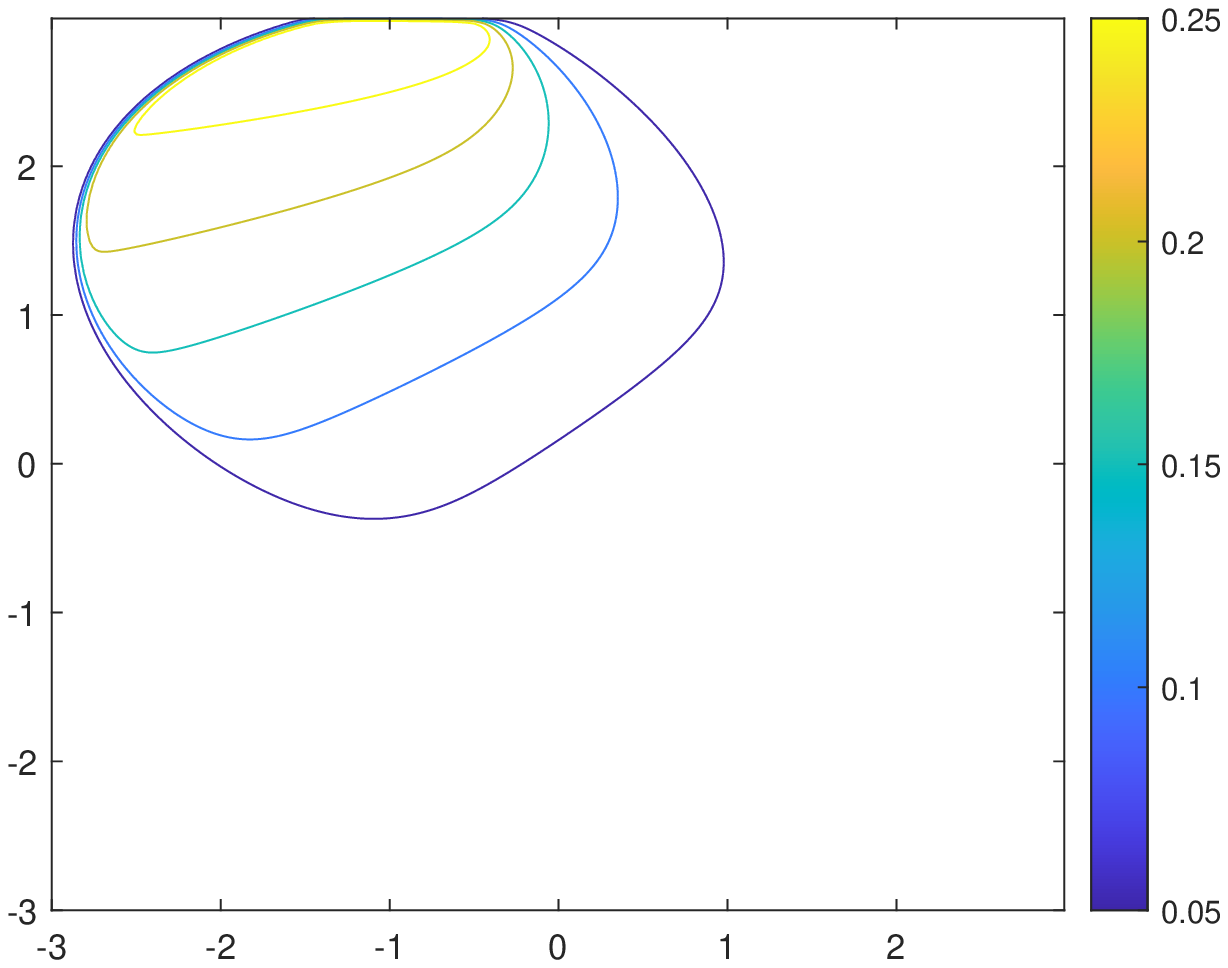}
}
\caption{Numerical evolution surfaces (a)--(d) and contours (e)-(h). Here the parameters are taken as
(a) and (e): $t=0$, $m_1=m_2=3464$, $n=10000$; (b) and (f): $t=0.5$, $m_1=m_2=3464$, $n=10000$; (c) and (g): $t=1.5$, $m_1=m_2=3464$, $n=30000$; (d) and (h): $t=3$, $m_1=m_2=3464$, $n=60000$.} \label{fig5}
\end{figure}

In {\rm \textbf{Case II}}, we describe a problem motivated from two-phase flow in porous
media with a gravitation pull in the $x$-direction. The flux functions $f(u)$ and $g(u)$ are ``{\rm \textbf{S}}-{\rm shape}'' with $f(0)=g(0)=0$ and $f(1)=g(1)=1$. Boundary value conditions are again put equal to zero. The diffusion coefficient is $\varepsilon = 0.01$ and the nonlinearity is very strong. We calculate the problem on the domain $\Omega = [-3,3]\times[-3,3]$ with a temporal step-size $\tau=0.00005$. The numerical results are displayed in Figure \ref{fig5}. In order to demonstrate the numerical surfaces easily, the surfaces are drawn every ten lines to reduce image storage. We observe that the peak decreases gradually from the center to the upper left corner, which is consistent with the result in the reference {\rm \cite{KR1997}}.
\end{example}

\subsection{Three-dimensional case}
\begin{example}\label{exam7}
Finally, we consider a three-dimensional diffusion problem \eqref{3D_diffu1} with the parameters $L_1=L_2=L_3=0$, $R_1=R_2=R_3=1$ and $T=1$. The initial and boundary value conditions and the source term $f(\mathbf{x},t)$ are determined by the exact solution $u(\mathbf{x},t)=\exp(1/2(x+y+z)-t)$.

Two sets of coefficients including 
\begin{itemize}
  \item [] {\rm \textbf{Case~I:}} (isotropic) $a=b=c=1$;
  \item [] {\rm \textbf{Case~II:}} (anisotropic) $a=1,\ b = 0.01,\ c=0.04$
\end{itemize}
are utilized to test the convergence rate and {\rm \textbf{CFL}} condition for the corrected difference scheme and the classical difference scheme, respectively. Numerical results are listed in Tables \ref{tab8a} and \ref{tab8b}.

We clearly see that the corrected difference scheme is fourth-order convergent when the step-ratio $r = 1/6$ and second-order convergent in other cases.
When the step-ratio arrives at the critical value $r=1/4$ of {\rm \textbf{CFL}} condition, the numerical results still have second-order convergence. Once the step-ratio $r>1/4$, the numerical results will irreversibly deviate from the exact solution and the numerical error tends to blow up. On the other hand, for the classical difference scheme, we have tested two sets of data with $r=1/6$ and $r = 1/5.99$. We see that it is second-order convergence without superconvergence when $r=1/6$ and not stable when $r = 1/5.99$. In short, all the data in Tables \ref{tab8a} and \ref{tab8b} are consistent with Theorem \ref{thm2} and the restrictive condition \eqref{3d_CFL}.
\begin{table}[H]
\caption{The errors in $L^{\infty}$-norm versus grid sizes reduction and convergence orders of the corrected difference scheme \eqref{3D_diffu2} for the three-dimensional linear diffusion equation in Example \ref{exam7}}
\vspace{-3mm}
\setlength\tabcolsep{1.2mm}{
\begin{tabular}{|c|cccccc|cccccc|}
\hline
\multicolumn{0}{|c|}{ }&\multicolumn{6}{c|}{\textbf{Case~I}}&\multicolumn{6}{c|}{\textbf{Case~II}}\\
         \cline{2-7}\cline{8-13}\diagbox{R}{P}
         & $m_1$  & $m_2$ & $m_3$ & $n$ & ${E}_{\infty}(h_x,h_y,h_z,\tau)$& ${\rm Ord_G}$ & $m_1$ & $m_2$ & $m_3$ & $n$ & ${E}_{\infty}(h_x,h_y,h_z,\tau)$& ${\rm Ord_G}$  \\
         \cline{1-7}\cline{8-13}
         & $5$  & $5$  & $5$  & $175$   & $4.4376e-6$ & $*$      &$5$ & $50$  & $25$  & $175$   & $4.9047e-6$ & $*$    \\
  $r=1/7$& $10$ & $10$ & $10$ & $700$   & $1.0501e-6$ & $2.0793$ &$10$& $100$ & $50$  & $700$   & $1.1087e-7$ & $2.1453$ \\
         & $20$ & $20$ & $20$ & $2800$  & $2.6476e-7$ & $1.9878$ &$20$& $200$ & $100$ & $2800$  & $2.6789e-7$ & $2.0491$ \\
         & $40$ & $40$ & $40$ & $11200$ & $6.5958e-8$ & $2.0051$ &$40$& $400$ & $200$ & $11200$ & $6.6501e-8$ & $2.0102$ \\
  \hline
         & $5$  & $5$  & $5$  & $150$   & $4.4890e-7$     & $*$               &$5$ & $50$  & $25$  & $150$  & $1.0395e-6$  & $*$             \\
$r=\textbf{1/6}$& $10$ & $10$ & $10$ & $600$& $2.7992e-8$ & $\textbf{4.0033}$ &$10$& $100$ & $50$  & $600$  & $6.6581e-8$  & $\textbf{3.9647}$ \\
         & $20$ & $20$ & $20$ & $2400$  & $1.7891e-9$     & $\textbf{3.9677}$ &$20$& $200$ & $100$ & $2400$ & $4.1608e-9$  & $\textbf{4.0002}$ \\
         & $40$ & $40$ & $40$ & $9600$  & $1.1179e-10$    & $\textbf{4.0003}$ &$40$& $400$ & $200$ & $9600$ & $2.6039e-10$ & $\textbf{3.9981}$ \\
  \hline
         & $5$  & $5$  & $5$  & $100$   & $1.4041e-5$ & $*$      &$5$ & $50$  & $25$  & $100$  & $1.2207e-5$ & $*$    \\
  $r=1/4$& $10$ & $10$ & $10$ & $400$   & $3.5797e-6$ & $1.9717$ &$10$& $100$ & $50$  & $400$  & $3.5630e-6$ & $1.7765$ \\
         & $20$ & $20$ & $20$ & $1600$  & $9.2056e-7$ & $1.9593$ &$20$& $200$ & $100$ & $1600$ & $9.1782e-7$ & $1.9568$ \\
         & $40$ & $40$ & $40$ & $6400$  & $2.3047e-7$ & $1.9979$ &$40$& $400$ & $200$ & $6400$ & $2.3151e-7$ & $1.9871$ \\
  \hline
         & $5$  & $5$  & $5$  & $100$   & $7.6526e-5$ & $*$      &$5$ & $50$  & $25$  & $100$  & $7.9323e-5$ & $*$    \\
$r=1/3.99$& $10$ & $10$ & $10$ & $399$  & $3.6070e-6$ & $4.4071$ &$10$& $100$ & $50$  & $399$  & $3.5901e-6$ & $4.4656$ \\
         & $20$ & $20$ & $20$ & $1596$  & $9.2751e-7$ & $1.9594$ &$20$& $200$ & $100$ & $1596$ & $9.2475e-7$ & $1.9569$ \\
         & $40$ & $40$ & $40$ & $6384$  & $6.5576e-7$ & $0.5002$ &$40$& $400$ & $200$ & $6384$ & $2.8248e-2$ & $-14.899$ \\
  \hline
\end{tabular}\label{tab8a}
}
\end{table}
\vspace{-12mm}
\begin{table}[H]
\begin{center}
\caption{The errors in $L^{\infty}$-norm versus grid sizes reduction and convergence orders of the classical Euler
difference scheme \eqref{scheme_3D_class} for the three-dimensional linear diffusion equation in Example \ref{exam7}}
\vspace{-3mm}
\setlength\tabcolsep{1.35mm}{
\begin{tabular}{|c|cccccc|cccccc|}
\hline
\multicolumn{0}{|c|}{ }&\multicolumn{6}{c|}{\textbf{Case~I}}&\multicolumn{6}{c|}{\textbf{Case~II}}\\
         \cline{2-7}\cline{8-13}\diagbox{R}{P}
         & $m_1$  & $m_2$ & $m_3$ & $n$ & ${E}_{\infty}(h_x,h_y,h_z,\tau)$& ${\rm Ord_G}$ & $m_1$ & $m_2$ & $m_3$ & $n$ & ${E}_{\infty}(h_x,h_y,h_z,\tau)$& ${\rm Ord_G}$  \\
         \cline{1-7}\cline{8-13}
         & $5$  & $5$  & $5$  & $150$  & $1.1712e-4$ & $*$      &$5$ & $50$  & $25$  & $150$  & $4.3049e-4$ & $*$    \\
  $r=1/6$& $10$ & $10$ & $10$ & $600$  & $3.0760e-5$ & $1.9289$ &$10$& $100$ & $50$  & $600$  & $1.1084e-4$ & $1.9574$ \\
         & $20$ & $20$ & $20$ & $2400$ & $7.9604e-6$ & $1.9501$ &$20$& $200$ & $100$ & $2400$ & $2.7742e-5$ & $1.9984$ \\
         & $40$ & $40$ & $40$ & $9600$ & $1.9963e-6$ & $1.9955$ &$40$& $400$ & $200$ & $9600$ & $6.9490e-6$ & $1.9972$ \\
  \hline
         & $5$  & $5$  & $5$  & $148$  & $8.6896e-6$   & $*$       &$5$ & $50$  & $25$  & $148$  & $3.1252e-4$   & $*$     \\
$r=1/5.9$& $10$ & $10$ & $10$ & $590$  & $3.1401e-5$   & $-1.8535$ &$10$& $100$ & $50$  & $590$  & $1.1285e-4$   & $1.4696$  \\
         & $20$ & $20$ & $20$ & $2360$ & $1.9419e+8$   & $-42.492$ &$20$& $200$ & $100$ & $2360$ & $3.2653e+13$  & $-58.006$ \\
         & $40$ & $40$ & $40$ & $9440$ & $2.7614e+108$ & $-332.70$ &$40$& $400$ & $200$ & $9440$ & $1.6859e+114$ & $-334.56$ \\
  \hline
\end{tabular}\label{tab8b}
}
\end{center}
\end{table}
\end{example}
\vspace{-10mm}
\section{Concluding remarks}\label{Sec7}
\setcounter{equation}{0}
\vspace{-3mm}
In closing, we propose a unified framework to construct explicit numerical schemes for convection-diffusion problems in high dimension
based on the forward Euler discretization. We obtain the superconvergence with the step-ratio $r=1/6$ for the corrected difference scheme and display much better numerical behavior, which serve to the theoretical results. Another advantage of the present scheme is that it is an fully explicit numerical method without any matrix by vector operation and pretty convenient in practical implementation.

 Moreover, the corrected difference schemes have essentially improved {\textbf{CFL}} condition and convergence rate of the classical difference scheme, see e.g. \cite{Sun2022} or \cite{LT2003}. The detailed theoretical results with respect to {\rm \textbf{CFL}} conditions and the convergence rate of the corrected difference scheme and classical difference scheme in different settings are listed in Tables \ref{tabConclu1} and \ref{tabConclu2}.

In terms of theoretical analysis, we only discuss the constant convection case in current paper. As for the general nonlinear convection-diffusion equations, it remains an open challenge for the convergence and stability because of the complex discretization of the nonlinear terms, which will leave as the future work.
\begin{table}[H]
\centering
\caption{The {\rm{\textbf{CFL}}} conditions for the diffusion problems in different dimensions with different difference schemes}
\renewcommand\arraystretch{1.4}
\begin{tabular}{|c|c|c|}
 \hline
\diagbox{D}{S}
                     & {\rm \textbf{Corrected difference scheme}} & {\rm \textbf{Classical difference scheme}, \cite{LT2003}}\\
\hline
  ${\rm 1D}$         & $r_x \leq \frac12$, \cite{ZZS2022}  & $r_x \leq \frac12$     \\
  \hline
  ${\rm 2D}$         & $\max\{r_x,r_y\} \leq \frac12$      & $r_x+r_y \leq \frac12$     \\
  \hline
  ${\rm 3D}$         & $\max\{r_x,r_y,r_z\} \leq \frac14$  & $r_x+r_y+r_z \leq \frac12$     \\
  \hline
\end{tabular}\label{tabConclu1}
\begin{tablenotes}
\footnotesize
\item[1]   $*$ ``S'' denotes the difference scheme and ``D'' the dimension.
\end{tablenotes}
\end{table}
\vspace{-10mm}
\begin{table}[H]
\centering
\caption{The convergence rates for the diffusion problems in different dimensions with different difference schemes}
\renewcommand\arraystretch{1.4}
\begin{tabular}{|c|cc|cc|}
\hline
\multicolumn{0}{|c|}{ }&\multicolumn{2}{c|}{{\rm \textbf{Corrected difference scheme}}}&\multicolumn{2}{c|}{\rm \textbf{Classical difference scheme}, \cite{LT2003}}\\
         \cline{2-3}\cline{4-5}\diagbox{D}{S}
          & $\qquad r\neq \frac16$ & $r=\frac16$  & $\qquad r\neq\frac16$& $r=\frac16$  \\
         \cline{1-3}\cline{4-5}
  ${\rm 1D}$ & $\qquad 2$, \cite{ZZS2022}  & $4$, \cite{ZZS2022}    & $\qquad 2$    & $2$   \\
 \hline
  ${\rm 2D}$ & $\qquad 2$  & $4$    & $\qquad 2$    & $2$     \\
 \hline
  ${\rm 3D}$ & $\qquad 2$  & $4$    & $\qquad 2$    & $2$  \\
  \hline
\end{tabular}\label{tabConclu2}
\end{table}

\begin{acknowledgements}
The first author is very grateful to Dr. Zhifeng Weng in Huaqiao Unviersity for his useful suggestion and comments for the nonlinear problem.
\end{acknowledgements}

\end{document}